\documentclass [12pt, makeidx]{amsart}
\pdfoutput=1
\usepackage{thmtools}
\usepackage{graphicx}
\usepackage{amsthm}
\usepackage{amssymb}
\usepackage[all]{xy}
\usepackage{verbatim}
\usepackage{float}
\usepackage{enumitem}
\usepackage{url}
\usepackage{float}
\usepackage{fourier}
\usepackage{stmaryrd}
\usepackage{scrtime}
\usepackage{soul}
 \usepackage[mathscr]{eucal}
 \usepackage{calligra}
 \usepackage[bbgreekl]{mathbbol}
 \usepackage{todonotes}
\usepackage{xcolor}
\newcommand{\hHom}{{\mathcal Hom}}
\newcommand{\myRed}[1]{\textcolor{red}{#1}}
\newcommand{\myBlue}[1]{\textcolor{blue}{#1}}
\newcommand{\CorrectionBlue}[1]{\textcolor{black}{#1}}
\newcommand{\CorrectionRed}[1]{\textcolor{black}{#1}}

\definecolor{Coneblue}{RGB}{0,255,255}
\definecolor{Coneorange}{RGB}{255,102,0}
\newtheorem{thm}{Theorem}[section]
\newtheorem{cor}[thm]{Corollary}
\newtheorem{lem}[thm]{Lemma}

\newtheorem{prop}[thm]{Proposition}

\newtheorem{defn}[thm]{Definition}
\theoremstyle{definition}
\newtheorem{prop-defn}[thm]{Proposition and Definition}

\theoremstyle{remark}
\newtheorem{rem}[thm]{Remark}
\newtheorem{rems}[thm]{Remarks}
\newtheorem{example}[thm]{Example}
\newtheorem{examples}[thm]{Examples}

\newtheorem{Question}[thm]{Question}
\newtheorem*{Question*}{\bf Question}

\numberwithin{equation}{section}

\newcommand{\thistheoremname}{}
\newtheorem*{genericprop*}{\thistheoremname}
\newenvironment{namedprop*}[1]
  {\renewcommand{\thistheoremname}{#1}%
   \begin{genericprop*}}
  {\end{genericprop*}}

\newtheorem*{genericlem*}{\thistheoremname}
\newenvironment{namedlem*}[1]
  {\renewcommand{\thistheoremname}{#1}%
   \begin{genericlem*}}
  {\end{genericlem*}}
  
  \newtheorem*{genericthm*}{\thistheoremname}
\newenvironment{namedthm*}[1]
  {\renewcommand{\thistheoremname}{#1}%
   \begin{genericthm*}}
  {\end{genericthm*}}

  \newtheorem*{genericcond*}{\thistheoremname}
\newenvironment{namedcond*}[1]
  {\renewcommand{\thistheoremname}{#1}%
   \begin{genericcond*}}
  {\end{genericcond*}}

\newcommand{\Real}{\mathbb R}
\newcommand{\eps}{\varepsilon}

\newcommand{\Id}{{\rm Id}}

\newcommand{\Hom}{{\rm Hom}}

\newcommand{\Ham}{{\rm  Ham}}

\newcommand{\Mor}{{\rm Mor}}

\newcommand{\C}{{\mathcal C}}

\newcommand{\F}{{\mathcal F}}
\newcommand{\G}{{\mathcal G}}
\newcommand{\calS}{{\mathcal S}}

\newcommand{\I}{{\mathcal I}}
\newcommand{\cI}{{\mathcal I}^{\bullet}}

\newcommand{\K}{{\mathcal K}}
\newcommand{\T}{{\mathcal T}}
\newcommand{\J}{{\mathcal J}}

\newcommand \id {{\rm id}}

\newcommand{\supp}{{\rm supp}}

\def \dispdot {}

\newcommand{\cD}{\ensuremath{\mathcal{D}}}
\newcommand{\SD}{\mathsf{D}}
\makeindex

\newcommand {\gr}{\mathrm {gr}}
\newcommand {\gra}{\mathrm {graph}}
\newcommand {\area}{\mathrm {area}}

\newcommand {\GFQI}{G.F.Q.I. }

\usepackage{hyperref}
\DeclareMathAlphabet{\mathpzc}{OT1}{pzc}{m}{it}
\usepackage[percent]{overpic}
\usepackage{datetime}

\newcommand{\DHam }{\mathfrak{DHam}}

\newcommand {\LL}{{\mathcal L}}
\makeindex
\usepackage[
  backend=biber,
     style=ext-alphabetic,
    sorting=nyt,
    giveninits,
    articlein=false,
    url=true, 
    doi=true,
    eprint=true
]{biblatex}
\DeclareFieldFormat[article, incollection, unpublished]{title}{\sl{#1}}

\definecolor{stephane}{rgb}{0,0,1}
\definecolor{claude}{rgb}{0,1,0}

\renewcommand{\to}[1][]{\xrightarrow[]{#1}}
\newcommand{\from}[1][]{\xleftarrow[]{#1}}
\newcommand{\isoto}[1][]{\xrightarrow[#1]%
{{\raisebox{-.6ex}[0ex][-.6ex]{$\mspace{1mu}\sim\mspace{2mu}$}}}}
\newcommand{\isofrom}[1][]{\xleftarrow[#1]%
{{\raisebox{-.6ex}[0ex][-.6ex]{$\mspace{1mu}\sim\mspace{2mu}$}}}}

\newcommand{\fai}[1]{K_{#1}}
\newcommand{\N}{\mathbb{N}}
\newcommand{\Q}{\mathbb{Q}}
\newcommand{\hocolim}{\operatorname{hocolim}}
\newcommand{\holim}{\operatorname{holim}}
\newcommand{\rsect}{\operatorname{R\Gamma}}
\newcommand{\RHom}{\operatorname{RHom}}
\newcommand{\rhom}{\operatorname{R\mathcal{H}om}}
\newcommand{\homst}{\operatorname{\mathcal{H}om}^\cstar}
\newcommand{\Ab}{\operatorname{Ab}}

\title{The singular support of sheaves is $\gamma$-coisotropic}
\author{St\'ephane  Guillermou and Claude Viterbo }
\thanks{SG: UMR 6629 du CNRS - Laboratoire de Mathématiques Jean LERAY, Nantes Université, France.
 Part of this paper was written as the author was a member of
  UMR 5582 du CNRS - Institut Fourier, Université Grenoble Alpes, France.
  \\ CV: Université Paris-Saclay, CNRS, Laboratoire de mathématiques d’Orsay, 91405, Orsay, France. Part of this paper was written as the author was a member of DMA, \'Ecole Normale Sup\'erieure, 45 Rue d'Ulm, 75230 Cedex 05, FRANCE. \\ Both authors acknowledge support from ANR MICROLOCAL (ANR-15-CE40-0007) and ANR COSY (ANR-21-CE40-0002).The first author also acknowledges support from  the Centre Henri Lebesgue  ANR-11-LABX-0020-01.}	

\begin{document}
\def \Z {\mathbb Z}
\def \H {\mathcal H}
\def \cI {\I^{\bullet}}
\def \cJ {\J^{\bullet}}
\def \Char {{\rm Char}}
\def \card {{\rm card}}
\def \cstar {\ast}
\def \convstar {\varoast}

\maketitle
\today ,\;\; \currenttime

\tableofcontents
\section{Introduction}\index{$\LL(T^*N)$} \index{$\DHam_c (T^*N)$}  \index{$\widehat \DHam_c (T^*N)$}  \index{$\widehat \LL (T^*N)$}
In \cite{Viterbo-STAGGF}, a metric, denoted $\gamma$,  was introduced on the set ${\mathfrak L} (T^*N)$ of Lagrangians Hamiltonianly isotopic to the zero section in $T^*N$, where $N$ is a compact manifold and on $\DHam_c (T^*N)$ the group of Hamiltonian maps with compact support for $N= {\mathbb R}^n$ or $T^n$ (and in  \cite{Viterbo-Montreal} it was extended to $\DHam_c (T^*N)$ for general compact $N$). The metric was generalized by Schwartz and Oh to general symplectic manifolds $(M, \omega)$ using Floer cohomology  for the Hamiltonian case (see \cite{Schwarz, Oh-spectrum}), and by \cite{Leclercq}  for the Lagrangian case\footnote{Assuming  $[\omega]\pi_2(M,L)=0$, $\mu_L \pi_2(M,L)=0$, where $\mu_L$ is the Maslov class of $L$. }. 
The notion of $\gamma$-coisotropic sets in a symplectic manifold was defined  in \cite{Viterbo-gammas}. A similar definition was due to Usher \cite{Usher} in the setting of Hofer distance under the name of ``locally rigid''. Notice that this yields a weaker notion : $\gamma$-coisotropic implies locally rigid in the sense of Usher.  

In the context of sheaves on manifolds, there is a notion of coisotropic sets in the sense of Kashiwara-Schapira (see \cite{K-S}, definition 6.5.1, p.~271 and Definition \ref{Def-cone-coisotropic}),  that the authors call {\bf involutivity}. To avoid any confusion with other notions, we shall use the term cone-coisotropic for their definition as it is defined using the contingent and paratingent cones (see  \cite{Bouligand} and Section \ref{Section-8}). 
With this notion, Kashiwara and Schapira proved

\begin{thm} [\cite{K-S}, theorem 6.5.4, p. 272]\label{Thm-K-S-involutive}
Given  a sheaf $\F \in D^b(N)$, its singular support $SS(\F)$  is cone-coisotropic. 
 \end{thm} 

The aim of this paper is to prove that

 \begin{thm} \label{Thm-sheaf-gamma}
 Given  a sheaf $\F \in D^b(N)$, its singular support $SS(\F)$ is $\gamma$-coisotropic. 
 \end{thm} 
 
 Theorem \ref{Thm-sheaf-gamma} implies Theorem  \ref{Thm-K-S-involutive}, since we shall prove the following connection between the two notions
  \begin{prop} \label{Prop-gamma-to-cone}
  Let $V$ be a closed subset in $(M,\omega)$. If $V$ is $\gamma$-coisotropic then it is cone-coisotropic. 
  \end{prop} 

\begin{rem} This is related to questions asked by Vichery in section 4 of \cite{Vichery}.
Our result, together with Proposition \ref{Prop-gamma-to-cone}, stating  that a $\gamma$-co\-iso\-tropic set is cone-coisotropic, gives a more natural proof of the Kashiwara-Schapira theorem. Moreover, as opposed to the notion of cone-coiso\-tro\-pic, which is only invariant by $C^1$ symplectic diffeomorphisms, the notion of  $\gamma$-coisotropic set is invariant by homeomorphisms preserving $\gamma$, in particular by symplectic homeomorphisms, i.e. homeomorphisms which are  $C^0$ limits of symplectic diffeomorphisms. 
Along the way we prove a number of results relating the singular support and the spectral norm $\gamma$. 
\end{rem} 

\section{Notation, comments and acknowledgements}
Since our results are local, we shall assume from now on that $N$ is a compact manifold. The following definitions will turn out useful : 
\begin{itemize} 
\item 
The category $D(N)$ is the unbounded derived category of sheaves of
$k$-vector spaces on $N$, for some given field $k$, but except in Appendix \ref{Appendix-B} the  reader may assume we are dealing with its bounded version. For the operations on sheaves associated to a continuous map $f: M \longrightarrow N$ that is  $f^{-1}, f_{*}, f^{!}, f_{!}$  and the two operations $\otimes, \operatorname{\mathcal{H}om}$ and their derived versions, $f^{-1},Rf_{*}, Rf_{!}, f^{!}$ and $\otimes, \rhom$ on
$D(N)$ we refer to \cite{K-S}. 
\item
Choosing a real analytic structure on $N$ we set $D_c(N)$ to be the category of constructible complexes $\F$, i.e. complexes such that, for some subanalytic stratification, the restriction of $\F$ to each stratum is locally constant and of finite rank. 
\item  For $\F\in D(N)$, $SS(\F)$ is the singular support (also called microsupport) defined by Kashiwara-Schapira in \cite{K-S}.
We set for short $SS^\bullet(\F)= SS(\F)\cap (T^*N\setminus 0_N)$.
\item
For an open set $U$ in the symplectic manifold $(M, \omega)$ we denote by $\DHam_c(U)$ the set of time one flows of Hamiltonian with compact support contained in $U$. 
\end{itemize}

We thank Pierre Schapira for useful conversations and Nicolas Vichery for Remark~\ref{rem:gamma_g=gamma_tau}.
While writing this paper, we realized that some of the results in  Subsection \ref{subsection-6.3} were independently discovered by Asano and Ike in \cite{Asano-Ike2}.
We thank Yuichi Ike for a careful reading of the manuscript.

\section{Reminders on the singular support of sheaves}

Let $N$ be a manifold, $k$ a field and $D(N)$ the derived category of sheaves of $k$-vector spaces on $N$. For $\F
\in D(N)$ we recall that $SS(\F) \subset T^*N$ is the closed conic subset defined by Kashiwara-Schapira as the
closure of the set of points $(x;\xi)$ such that $(\rsect_{\{f\geq 0\}}\F)_x \not= 0$ for some function $f$ of class
$C^1$ such that $f(x)=0$ and $df_x=\xi$. It is called microsupport or singular support of $\F$.

The functor $\rsect_{\{f\geq 0\}}$ does not commute with infinite direct sums and
$(-)_x$ does not commute with infinite direct products. Here are two variations on
the definition which have these commutation properties.  For the statement of the
lemmas we introduce the following definition.
\begin{defn}\label{Def-lens}  Let $\Omega \subset T^*N \setminus 0_N$ be an
open conic subset.  We call {\em $\Omega$-lens} a locally closed subset $\Sigma$ of $N$
with the following properties: $\overline{\Sigma}$ is compact and there exists an open
neighbourhood $U$ of $\overline{\Sigma}$ and a function $g\colon U \times [0,1] \to \Real$
of class $C^1$ such that
\begin{enumerate}
\item $dg_t(x) \in \Omega$ for all $(x,t) \in U \times [0,1]$, where
  $g_t = g|_{U\times\{t\}}$,
\item $\{g_t<0\} \subset \{g_{t'}<0\}$ if $t\leq t'$,
\item the hypersurfaces $\{g_t=0\}$ coincide on $U\setminus \overline{\Sigma}$,
\item $\Sigma = \{g_1<0\} \setminus \{g_0<0\}$.
\end{enumerate}
 \end{defn} 
We recall that $SS^\bullet(k_{\{g_t<0\}} ) = \{(x;\lambda dg_t(x))$; $g_t(x)=0$, $\lambda>0\}$.
 Using the triangle
$k_{\{g_0<0\}} \to k_{\{g_1<0\}} \to k_\Sigma \to[+1]$  and the ``triangle inequality''  for the singular support (see \cite{K-S})  we obtain
$SS^\bullet(k_\Sigma) \subset \Omega$.
\begin{figure}[ht]
 \begin{overpic}[width=6cm]{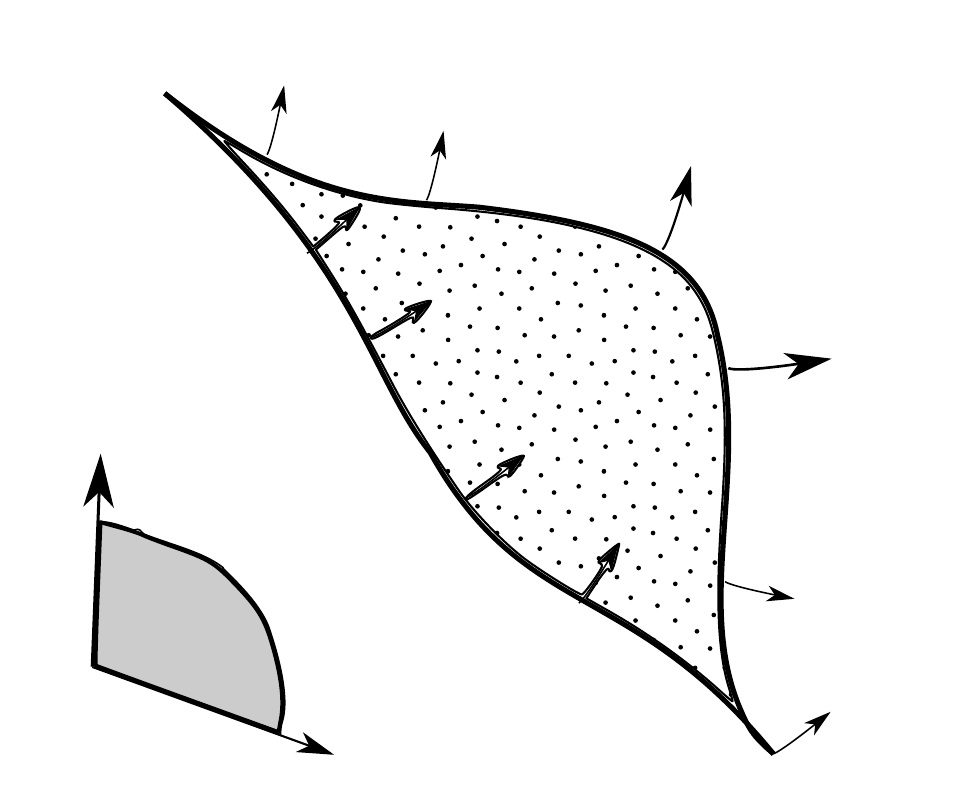}
 \put (78,35) {$g_1=0$} 
  \put (17,45) {$g_0=0$} 
   \put (20,0) {$ {\Omega}$} 
 \end{overpic}
 \caption{A lens. The arrows represent $\nabla g_1$ and $\nabla g_0$.}
 \end{figure}
The next lemma is essentially a reformulation of the definition of the singular support.
\begin{lem}\label{lem:def-micsupp-bis}
  Let $\F \in D(N)$ and let $\Omega \subset T^*N \setminus 0_N$ be an open conic
  subset.  Then $SS(\F) \cap \Omega = \emptyset$ if and only if
  $\RHom(k_\Sigma, \F) \simeq 0$ for any $\Omega$-lens $\Sigma$.
\end{lem}
\begin{proof}
  (i) We first assume $SS(\F) \cap \Omega = \emptyset$.  We recall that $\RHom(k_V,\F)
  \simeq \rsect(V;\F)$ for any open subset $V$ of $N$. Using the notations in the
  definition of $\Omega$-lens we have a distinguished triangle $k_{\{g_0<0\}} \to
  k_{\{g_1<0\}} \to k_\Sigma \to[+1]$.  Applying $\RHom(-,\F|_U)$ to this triangle we obtain
  the result, using the non-characteristic deformation lemma (\cite[Prop.~2.7.2]{K-S}) and
  the definition of the singular support.

  \medskip\noindent (ii) We assume that there exists
  $(x_0;\xi_0) \in SS(\F) \cap \Omega$.  Hence we can find $(x_1;\xi_1) \in \Omega$
  (close to $(x_0;\xi_0)$) and a function $f \colon N \to \Real$ of class $C^1$
  such that $f(x_1)=0$, $df(x_1)=\xi_1$ and
  $(\rsect_{\{f\geq 0\}}\F)_{x_1} \not= 0$.  Now we can find a basis of open
  neighbourhoods of $x_1$, say $U_n$, $n\in \N$, such that
  $\Sigma_n := \{f\geq 0\} \cap U_n$ is an $\Omega$-lens.  Then there exists $n$ such
  that $\rsect(U_n ; \rsect_{\{f\geq 0\}}\F) \not=0$ and the result follows from
  \begin{multline*}
    \rsect(U_n ; \rsect_{\{f\geq 0\}}\F)
    \simeq \rsect(N; \rhom(k_{U_n}, \CorrectionBlue{ \rhom(k_{\{f\geq 0\}}, \F) } )) \\
\simeq \rsect(N; \rhom(k_{\Sigma_n},\F)) \simeq  \RHom(k_{\Sigma_n},\F) \dispdot
\end{multline*}
\end{proof}

The functor $\RHom(k_\Sigma,-)$ commutes with direct products but not with direct sums.  Here
is a dual version of the previous lemma using the functor $\rsect(N; k_\Sigma \otimes -)$ which
commutes with direct sums.

\begin{lem}\label{lem:def-micsupp-ter}
  Let $\F \in D(N)$ and let $\Omega \subset T^*N \setminus 0_N$ be an open conic
  subset.  Then $SS(\F) \cap \Omega = \emptyset$ if and only if
  $\rsect(N; k_\Sigma \otimes \F) \simeq 0$ for any $\Omega^a$-lens $\Sigma$, where $\Omega^a$ is the antipodal
  set of $\Omega$, that is, its image by $(x;\xi) \mapsto (x;-\xi)$.
\end{lem}
The first author thanks Pierre Schapira for a suggestion simplifying the next proof.
\begin{proof}
  Let $\Sigma$ be an $\Omega^a$-lens. We write it as the difference of the two closed sets $\overline{\Sigma}$ and
  $Z= \overline{\Sigma} \setminus \Sigma$.  We can find two families of open neighbourhoods of $\overline{\Sigma}$,
  say $\{U_n\}_{n\in \N}$, and $Z$, say $\{V_n\}_{n\in \N}$, such that, for all $n$, we have: $V_n \subset U_n$ and
  $U_n\setminus V_n$ is an $\Omega$-lens.  For each $n$ we have the commutative diagram of restriction maps
  $$
  \xymatrix{
    \rsect(U_n; \F) \ar[r]^{u_n} \ar[d]  & \rsect(V_n; \F) \ar[d] \\
   \rsect(N; \F \otimes k_{\overline{\Sigma}}) \ar[r]^u  & \rsect(N; \F \otimes k_Z) \dispdot
   }
  $$
  The cone of $u$ is $\rsect(N; k_\Sigma \otimes \F)$, so we want to prove that $u$ is an isomorphism, that is, $u$
  induces an isomorphism on all cohomology groups.  The cone of $u_n$ is $\RHom(k_{U_n\setminus V_n}, F)$ and the
  previous lemma says that $u_n$ is an isomorphism.  Now the result follows from $H^i(N; \F \otimes
  k_{\overline{\Sigma}}) \simeq \varinjlim_n H^i(U_n; \F)$ and $H^i(N; \F \otimes k_Z) \simeq \varinjlim_n H^i(V_n;
  F)$, for all $i$ (see~\cite[Rem.~2.6.6]{K-S}).
\end{proof}

The previous two lemmas give the following proposition (see~\cite[Ex. 5.7]{K-S}), away from the zero-section.
To deal with the zero-section we use $SS(\F) \cap 0_N = \supp(\F)$.
\begin{prop}\label{Prop-2.3}
  Let $\F_n \in D(N)$, $n\in\N$, be given. Then $SS(\bigoplus_n \F_n) \subset \overline{\bigcup_n SS(\F_n)}$ and
  $SS(\prod_n \F_n) \subset \overline{\bigcup_n SS(\F_n)}$.
\end{prop}

\section{Quantization of Hamiltonian isotopies}\label{sec:qhi}
Everything in the next subsection is from \cite{G-K-S} or \cite{K-S}.  Given $\varphi$, a homogeneous Hamiltonian map
from $T^*N\setminus 0_N$ to $T^*N\setminus 0_N$, Guillermou, Kashiwara and Schapira associate to $\varphi$ a Kernel, that is an
element of $D^b(N\times N)$, noted $K_\varphi$ having the following properties\footnote{In fact in \cite{G-K-S} the
  authors associate to a Hamiltonian isotopy $\varphi^s$ a Kernel $\K_\Phi \in D^b(N\times N \times [0,1])$ such that
  ${\K_\Phi}_{\mid N\times N \times \{s\}}=K_{\varphi^s}$. Note that $K_{\varphi^1}$ depends on $\Phi$ not only $\varphi^1$ and our notation is slightly abusive.}.  First $K_\varphi$ is a {\em quantization} of the
homogeneous Lagrangian $$\Lambda_\varphi = \{(z,\varphi(z)) \mid z\in T^*N\setminus 0_{N} \}\subset {\overline {T^*N}}\times T^*N$$ in
the sense that $SS^\bullet(K_\varphi)= \Lambda_\varphi$, where we defined, for a sheaf $\F$ on a manifold $N$,
$$
SS^\bullet(\F)= SS(\F)\cap (T^*N\setminus 0_N)\dispdot
$$
Before we state the second property we recall the {\em composition} of sheaves:
\begin{defn} 
  Let $\F \in D^b(M)$ and $K \in D^b(M\times N)$. We define $$K(\F)=\F \circ K=R{q_N}_!(K\otimes q_M^{-1}\F) \in D^b(N),$$
  where  $q_M, q_N$  are the projections of $M\times N$ on $M$ and $N$ respectively.
\end{defn} 
The second property is then : 
\begin{prop} [See Proposition 7.1.2 (ii) in \cite{K-S} and \cite{G-K-S}, formula (1.12) ]\label{prop:action_noyau}
We have $$SS^\bullet(K_\varphi ( \F))=\varphi (SS^\bullet(\F))$$
\end{prop} 
Now if $\varphi$ is an exact non-homogeneous symplectic map from $T^*M$ to $T^*N$ (i.e. $PdQ-pdq$ is exact), we
replace $\Lambda_\varphi$ by the homogeneous lift of the graph of $\varphi$, say $\widehat \Lambda_\varphi \subset
\overline{T^*M}\times T^*N \times T^* {\mathbb R} $, given by
$$\widehat \Lambda_\varphi=\left \{(q,-\tau p,Q(q,p),\tau P(q,p), F(q,p), \tau ) \mid \varphi (q,p)=(Q,P), \tau > 0 \right\}$$
where $dF=PdQ-pdq$. 
Note that considering $\widehat \Lambda_\varphi$ as a correspondence in $\overline{T^*M}\times T^*N \times T^*
{\mathbb R} $, if $L$ is an exact Lagrangian  in $T^*M$ and $$\widehat L=\left\{(q,\tau p, f_L(q,p), \tau) \mid (q,p)\in L, \tau > 0\right \}$$
 we have $\widehat \Lambda_\varphi \circ \widehat L=\widehat {\varphi (L)}$.

%

\CorrectionBlue{
  From now on we assume $M=N$ and $\varphi$ is a Hamiltonian map with compact support (which will always be
  the case in this paper), the existence of a quantization $\K_\varphi$ for $\widehat \Lambda_\varphi$ follows
  from the homogeneous case by considering $\widehat \Lambda_\varphi$ as a deformation of $\widehat
  \Lambda_\id$.  However we need to take care of two points.  The first one is that the natural map which
  realizes this deformation, $\widehat\Phi=\id_{T^*N} \times \widehat\varphi$, given by
$$\widehat\Phi(q,p,q',p',t,\tau)=(q,p, Q(q',p'/\tau), \tau P(q',p'/\tau), t+F(q', p'/\tau), \tau)$$
is not defined everywhere outside the zero section (for $p'\not=0$ and $\tau=0$).  The second point is that it
is not enough to deform $\widehat \Lambda_\id$ because the singular support of $k_{\Delta_N \times
  [0,+\infty[}$ is slightly bigger than desired:
$$
SS^\bullet(k_{\Delta_N \times [0,+\infty[}) = \widehat \Lambda_\id \cup \{(q,p,q,-p,t,0) \mid t\geq 0\}\dispdot
$$
(only its intersection with $\{\tau >0\}$ is equal to $\widehat \Lambda_{\id}$).  Taking these points into
account, we can find a homogeneous map $\psi$ acting on $T^*(N^2\times\Real) \setminus 0_{N^2\times\Real}$
such that $\psi$ is the identity map near $\{\tau=0\}$ and $\psi(\widehat \Lambda_\id) = \widehat
\Lambda_\varphi$.  We set
}
$$\K_\varphi = K_\psi (k_{\Delta_N \times [0,+\infty[})$$ and we find
$$
SS^\bullet(\K_\varphi) = \widehat \Lambda_\varphi \cup \{(q,p,q,-p,t,0) \mid t\geq 0\}\dispdot
$$
If we consider the whole isotopy $\varphi_s$ and choose $t_0$ big enough so that $\widehat \Lambda_{\varphi_s}$ is
contained in $T^*(N^2 \times \mathopen]-t_0,t_0[)$ for all $s\in [0,1]$, then we see that $SS^\bullet(\K_{\varphi_s}) \cap
T^*U$ is independent of $s$, where $U = N^2\times (\Real\setminus [-t_0,t_0])$.  By~\cite[Prop.~5.4.5]{K-S} we deduce
that
$$
\K_\varphi|_{N^2 \times ]-\infty,-t_0[} \simeq 0, \qquad
\K_\varphi|_{N^2 \times ]t_0,+\infty[} \simeq k_{\Delta_N \times [0,+\infty[}|_{N^2 \times ]t_0,+\infty[} \dispdot
$$
In other words $\K_\varphi=\K_\id=k_{\Delta_N\times [0,+\infty[}$ outside of a compact set. 

\begin{rem}
  We could obtain the existence of $\K_\varphi$ when $M=N$ and $\varphi$ has compact support without assuming that
  $\varphi$ is Hamiltonian by adapting the result of \cite{Guillermou, Viterbo-Sheaves} (see
  Theorem~\ref{Thm-Quantization}).  In {\em loc. cit.}  the existence of a quantization is proved for a Legendrian of
  $J^1N$ which is the lift of a compact exact Lagrangian of $T^*N$. Here the graph of $\varphi$ is not compact but
  coincides with the diagonal $T^*_{\Delta_N}N^2$ outside a compact set (if we choose $F=0$ at infinity) so it should
  not be too difficult to extend the result in our situation.
\end{rem}

Since we use $\K_\varphi$, defined over $N^2 \times\Real$, instead of sheaf defined over $(N\times\Real)^2$, we
change slightly the composition functor as follows, which becomes a mixture of $\circ$ and the convolution $\cstar$
(see Section \ref{sec:specinv} for the definition)
\begin{defn}\label{def:compo-MNR}
    Let $\F \in D^b(M\times\Real)$ and $\K \in D^b(M\times N\times\Real)$. We define
$$\K^{\convstar}(\F)=\F \convstar \K= Rs_! R{q_{2}}_{!}(q_{1,2}^{-1}\K \otimes q_1^{-1}\F)\in D^b(N\times {\mathbb R} ),$$
where the maps
\begin{gather*}
q_{1,2} \colon M \times N \times {\mathbb R}^2 \to M\times N \times {\mathbb R}, \quad
q_1 \colon M\times N \times {\mathbb R}^2 \to M\times {\mathbb R}, \\
q_2 \colon M\times N \times {\mathbb R}^2 \to N\times {\mathbb R}^2, \qquad
s \colon N\times {\mathbb R}^2 \to N\times {\mathbb R}  
\end{gather*}
are given by $q_{1,2}(x,y,t_1,t_2) = (x,y,t_1)$, $q_1(x,y,t_1,t_2) = (x,t_2)$,
$q_2(x,y,t_1,t_2) = (y,t_1,t_2)$,
$s(y,t_1,t_2) = (y,t_1+t_2)$.
\end{defn}
Now we obtain a non homogeneous version of Proposition~\ref{prop:action_noyau}.
\begin{prop}
  For a compactly supported Hamiltonian map $\varphi$ from $T^*N$ to itself and any $\F\in D^b(N\times\Real)$ with
  $SS(\F) \subset \{\tau \geq 0\}$, we have $SS^\bullet(\K_\varphi^\convstar (\F)) = \widehat\varphi( SS^\bullet(\F) )$.
\end{prop}
In particular, for $L\in \LL(T^*N)$ we have $SS^\bullet(\K_\varphi ( \F_L))= SS^\bullet(\F_{\varphi (L)})$
and, by uniqueness of the quantization (see Theorem \ref{Thm-Quantization}, (\ref{4})),  $\K_\varphi^\convstar (\F_L) \simeq \F_{\varphi (L)}$.

\section{Spectral invariants for sheaves and Lagrangians}\label{sec:specinv}
For $M$ a symplectic manifold, (here we only need $M=T^*N$), if $\Lambda(M)$ is the bundle of Lagrangians subspaces of the tangent bundle to $M$, with fibre the Lagrangian Grassmannian $\Lambda (T_zM)\simeq \Lambda (n)$, we denote by $\widetilde \Lambda (M)$ the bundle induced by the universal cover $\widetilde \Lambda (n) \longrightarrow \Lambda(n)$. 

When using coefficients others than $\mathbb Z/2\mathbb Z$ we assume we have a lifting $\widetilde G_L$ of the Gauss map $G_L: L \longrightarrow \Lambda(T^*N)$ given by $x\mapsto T_xL$ to $\widetilde\Lambda(T^*N)$. This is called a grading. Given a graded $L$, the canonical deck transformation of the covering  induces a new grading and we denote it as $L[1]$ (or $\widetilde L[1]$), and its $k$-th iteration as $L[k]$ (or $\widetilde L[k]$). The grading yields an absolute grading for the Floer homology\footnote{Without this extra piece of information, the Floer cohomology $FH^{k}(L_{1},L_{2})$ is generally defined only up to a shift in grading. However, when $L_{2}=\varphi_{H}^{1}(L_{1})$, there is an absolute grading, {\it a priori} depending on the choice of $H$. However in our situation we do not assume $L_{1} $  is  Hamiltonianly isotopic to $L_{2}$, so the grading is required to have an absolute grading of the Floer cohomology.} and hence for the complex of sheaves in the Theorem stated below. We refer to \cite{Seidel-graded} and \cite{Viterbo-gammas} for more details on this, but point out that we shall never mention explicitly the grading. 

An  exact Lagrangian in $T^*N$ is a pair $(L,f_L)$ such that $df_L=\lambda_{\mid L}$ where $\lambda=pdq$ is the Liouville form. 
For an exact Lagrangian in $T^{*}N$,  a grading always exists since the obstruction to its existence is given by the Maslov class, and for exact Lagrangians in $T^*N$ the Maslov class vanishes, as was proved by  Kragh and Abouzaid (see \cite{Kragh3} and also the sheaf-theoretic proof by \cite{Guillermou}).
When $f_L$ is implicit we only write $L$, for example  $0_N$ means $(0_N, 0)$. 
A Lagrangian brane is a triple  $\widetilde L=(L,f_L, \widetilde G_L)$, where $L$ is a compact Lagrangian. For  $c$ a real constant, we write $T_c\widetilde L=(L, f_L+c, \widetilde G_L)$. We also use the notation $\widetilde L +c$ or $L+c$ if the grading is irrelevant (as it will be most of the time). 
\begin{defn} We set $\LL (T^{*}N)$ to be the set of Lagrangian branes in $T^{*}N$. 
\end{defn} 
Let $L$ be  a Lagrangian brane in $T^*N$ and $$\widehat L=\left\{(q,\tau p, f_L(q,p), \tau) \mid (q,p)\in L, \tau > 0\right \}$$
the homogenized Lagrangian in $T^*(N\times {\mathbb R})$.

Similarly for Hamiltonian maps in an exact manifold $(M, \omega= d\lambda)$ we look for pairs $(\varphi, F)$ such that $\varphi^*\lambda-\lambda = dF$ on $M$.

To state the next result we recall the two adjoint functors, $\cstar$ and $\homst$, on the category
$D(N\times\Real)$ introduced by Tamarkin in~\cite{Tamarkin}. First we set $s, q_1, q_2: N\times ({\mathbb R})^2
\longrightarrow N \times {\mathbb R}$ given by $s(x,t_1,t_2)=(x,t_1+t_2)$, $q_1(x,t_1,t_2)=(x,t_1)$,
$q_2(x,t_1,t_2)=(x,t_2)$
and we define
\begin{align*}
  \F \cstar \G &= Rs_! (q_1^{-1}\F \otimes q_2^{-1}\G) , \\
   \homst(\F, \G) &= Rq_{1*} \rhom(q_2^{-1}\F, s^!\G) \dispdot 
\end{align*}
Then $\homst$ is the adjoint of $\cstar$ in the sense that
$$\Mor_{D(N\times {\mathbb R}) } (\F, \homst(\G, \H))=\Mor_{D(N\times {\mathbb R})} (\F \cstar \G, \H) \dispdot$$

We recall some properties of $\homst$.  We denote by $D_{\tau\geq 0}(N\times\Real)$ the subcategory of
$D(N\times\Real)$ of objects with singular support contained\footnote{This category  contains the ``Tamarkin category''
which is the left orthogonal of $D_{\tau\leq 0}(N\times\Real)$.  The sheaves $\F_L$ belongs to the Tamarkin category.
The category $D_{\tau\geq 0}(N\times\Real)$ is stable by $\cstar$ and $\homst$.  The Tamarkin category is
stable by $\cstar$ but not $\homst$: typically when $N$ is a point we have $\homst(k_{[a,\infty[},
k_{[b,\infty[}) \simeq k_{]-\infty, b-a[}[1]$.} in $\{\tau\geq 0\}$.  A base change formula gives, for any $c\in \Real$, and $\F, \G \in D^{b}(N\times {\mathbb R})$
\begin{equation}\label{eq:homcstar-Hom}
\rsect_{N \times \{-c\}}(N \times\Real; \homst(\F, \G)) \simeq \RHom(\F, T_{c*}(\G)),
\end{equation}
with $T_c\colon N\times\Real \to N\times\Real$, $(x,t) \mapsto (x,t+c)$. 
Note that $\widehat{T_cL}=T_c \widehat L$.  

For $\H$ with singular support in $\{\tau\geq 0\}$, we have a morphism
$$
 \rsect_{N \times \{b\}}(N \times\Real; \H) \to \rsect_{N \times [a,b]}(N \times\Real; \H)
\isofrom \rsect_{N \times \{a\}}(N \times\Real; \H) 
$$
for any $a \leq b$.  We remark that the cone of this morphism is $\rsect_{N \times [a,b[}(N \times\Real; \H)$.
Using~\eqref{eq:homcstar-Hom} we deduce a morphism, say $u$, from $\RHom(\F, T_{c*}(\G))$ to $\RHom(\F, T_{d*}(\G))$
for any $c\leq d$. Taking $\F=\G$, $c=0$ the image of $\id_\G$ gives
\begin{equation}\label{eq:tau-version1}
   \tau_{0,d}\colon \G \to T_{d*}(\G) \qquad \text{for any $d\geq 0$}
\end{equation}
(we come back to this in~\S\ref{sec:Gtopology}).  Then $u$ is nothing but the composition with
$T_{c*}(\tau_{0,d-c})$.

According to \cite{Guillermou} and  \cite{Viterbo-Sheaves} we have
\begin{thm} \label{Thm-Quantization}
To each $L\in \LL(T^*N)$  we can associate $\F_L \in D^b(N\times {\mathbb R} )$ such that 
\begin{enumerate} 
\item\label{1} $SS^\bullet(\F_L)=\widehat L$. 
\item\label{2} $\F_L$ is simple (cf. \cite[Def.~7.5.4]{K-S}), $\F_L=0$ near $N\times
  \{-\infty\}$ and $\F_L=k_{N\times\Real}$ near $N\times \{+\infty\}$.
\item\label{3} We have an isomorphism 
$$FH^\bullet(L_0,L_1;a,b)=H^*_{N\times [a,b[} \left (N\times \Real ; \hHom^{\cstar}(\F_{L_0},\F_{L_1})\right)\dispdot $$
\item \label{4}  $\F_L$ is unique satisfying properties (\ref{1}) and (\ref{2}). 
\item \label{5} There is a natural product map  $$
 \homst(\F_{L_1},\F_{L_2}) \otimes \homst(\F_{L_2},\F_{L_3}) \longrightarrow  \homst(\F_{L_1},\F_{L_3})
$$
inducing in cohomology a map
\begin{gather*}
H^*_{N\times [\lambda , +\infty [}(N\times \Real ; \homst(\F_{L_1},\F_{L_2})) \otimes H^*_{N\times [\mu , +\infty [}(N\times \Real ;  \homst(\F_{L_2},\F_{L_3}))
\\  \Big\downarrow \cup_{\cstar} \\ H^*_{N\times [\lambda + \mu , +\infty [}(N\times \Real ; \homst(\F_{L_1},\F_{L_3}))
\end{gather*}
that coincides through the above identifications with the triangle product in Floer cohomology.\end{enumerate} 
\end{thm} 
\begin{rems}\leavevmode
\begin{enumerate} 
\item  
Note that we have  $$\F_{T_cL}={T_c}_*(\F_L)$$
\item In case $L$ has a Generating Function Quadratic at infinity, $S(x,\xi)$ defined on the vector bundle $E\overset{\pi} \longrightarrow  N$,  we can take for $\F_L$ the complex $R(\pi\times \id_{\mathbb R} )_*(k_{U_S})$ where $U_S= \{ (x,\xi,t) \mid S(x,\xi) \leq t \}$.
\end{enumerate}  
\end{rems}

\subsection{Definition of spectral invariants}

Now we can give a sheafy definition of the spectral invariants introduced in~\cite{Viterbo-STAGGF}.  The next
formulation appears in~\cite{Vichery} (see also~\cite{Viterbo-gammas}).

In fact the invariants can be defined for objects $\F$ of $D_{\tau\geq 0}(N\times\Real)$ satisfying:
\begin{equation}\label{eq:hyp_faisc_infini}
  \F|_{N\times ]-\infty, -t_0[} \simeq 0,  \quad
  \F|_{N\times ]t_0,+\infty[} \simeq k_{N\times\Real}|_{N\times ]t_0,+\infty[}
  \quad \text{for some $t_0$.}
\end{equation}
We have seen that $\F_L$ satisfies~\eqref{eq:hyp_faisc_infini} for any $L\in \LL(T^*N)$.  For $\F_1$, $\F_2
\in D_{\tau\geq 0}(N\times\Real)$ satisfying~\eqref{eq:hyp_faisc_infini}, $\homst(\F_1,\F_2)$ is
$k_{N\times\Real}[1]$ near $N\times \{-\infty\}$ and $0$ near $N\times \{+\infty\}$.  Using \ref{eq:homcstar-Hom} and
$H^0_{N\times\{c\}}(N\times\Real; k_{N\times\Real}[1]) \simeq k$, we deduce a natural morphism $\F_1 \to
T_{c*}(\F_2)$, for any $c$ big enough, which restricts to $\id \colon k_{N\times\Real} \to k_{N\times\Real}$ near
$N\times \{+\infty\}$. More generally $H^*_{N\times\{c\}}(N\times\Real; k_{N\times\Real}[1]) \simeq
H^*(N;k_N)$ and any class $\alpha \in H^d(N;k_N)$ yields a morphism
$$
u(\alpha, \F_1,\F_2,c) \colon \F_1 \to T_{c*}(\F_2) [d],
$$
for $c$ large enough, which restricts to $\alpha \colon k_{N\times\Real} \to k_{N\times\Real}[d]$ near $N\times
\{+\infty\}$ (using $\Mor(k_N, k_N[d]) \simeq H^d(N;k_N)$). The cup product corresponds to the composition, and for $c$ (and $t_0$) large enough this is summarized by  the following diagram
$$\xymatrix{
H^*(N,k)\simeq R\Gamma(N,k_N)\ar[d]^{\simeq}\ar[r]^-{\simeq}& \RHom(\F_1,{T_c}_*\F_2) \ar[d]^-{\simeq} \\ \RHom(k_{N\times ]t_0, +\infty[}, k_{N\times ]t_0, +\infty[})\ar[r]^-{\simeq} &\RHom\left ({\F_1}_{\mid N\times ]t_0,+\infty[}, {T_c}_*{\F_2}_{\mid N\times ]t_0,+\infty[}\right )&
}
$$
\begin{defn}\label{def:spectral-inv}
  For $\F_1$, $\F_2 \in D_{\tau\geq 0}(N\times\Real)$ satisfying~\eqref{eq:hyp_faisc_infini} and $\alpha \in
  H^d(N;k_N)$ we set
  $$
  c(\alpha, \F_1,\F_2)
  = -\inf \{c ;\; \exists u\colon \F_1 \to T_{c*}(\F_2) [d],\; u|_{N \times ]t_0,+\infty[} = \alpha, \; \text{for $t_0\gg0$}  \}\dispdot
  $$
  (Here $t_0$ is big enough so that both $\F_1$ and $T_{c*}(\F_2)$ satisfy~\eqref{eq:hyp_faisc_infini}).  We also set
  $c_-(\F_1,\F_2) = c(1,\F_1,\F_2)$ and if $N$ is oriented with fundamental class
$\mu_N$, $c_+(\F_1,\F_2) = c(\mu_N, \F_1,\F_2)$. Finally we set  $\gamma (\F_1,\F_2) = c_+(\F_1,\F_2) -
  c_-(\F_1, \F_2)$.

  For $L_1, L_2 \in \LL(T^*N)$ we set $c(\alpha, L_1,L_2) = c(\alpha, \F_{L_1},\F_{L_2})$ and we define
  $c_\pm(L_1,L_2)$, $\gamma(L_1,L_2)$ in the same way. Finally we set $c_\pm(L) = c_\pm(0_N,L)$.
\end{defn}
It is not difficult to see that $$c_-(\F_1,\F_2) = \inf \{c(\alpha,\F_1,\F_2);\; \alpha \in H^*(N;k_N)\}$$ and  $$c_+(\F_1,\F_2) = \sup
\{c(\alpha,\F_1,\F_2);\; \alpha \in H^*(N;k_N)\}$$
Indeed, a class $\alpha$ induces by multiplication $v \colon \F_{2} \to \F_2 [d],\; v|_{N \times ]t_0,+\infty[} =
\alpha$; now if we have $u \colon \F_{1} \to T_{c*}(\F_2) [0],\; u|_{N \times ]t_0,+\infty[} = 1$, then $v\circ
u\colon \F_{1} \to T_{c*}(\F_2) [d]$ satisfies $v\circ u|_{N \times ]t_0,+\infty[} = \alpha$ so $c(1,\F_{1},\F_{2}) \leq c(\alpha,
\F_{1},\F_{2})$.
 
Using the construction of the sheaf $\F_{L}$ by using Floer cohomology as in \cite{Viterbo-Sheaves} it is easy to see that 
the present definition of $c(\alpha, \F_{L_{1}}, \F_{L_{2}})$ coincides with the definition of $c(\alpha, L_{1}, L_{2})$ using Floer cohomology.  
In particular, we have that $c_+(\F_{L_1},\F_{L_2}) = -c_-(\F_{L_2},\F_{L_1})$ and this can be extended to the general case of $\F_1,\F_2$ using \cite{Viterbo-Sheaves}. 

We have given the definition using the right hand side of~\eqref{eq:homcstar-Hom}.  We can also rewrite it with the
left hand side as follows. Let us set $\H = Rt_* \homst(\F, \G))$, where $t \colon N\times \Real \to \Real$
is the projection.  Then $\H$ contains all the information about the spectral invariants. More
precisely~\eqref{eq:homcstar-Hom} is isomorphic to $\rsect_{\{-c\}}(\H)$ and $\H$ has its singular support in $\{\tau\geq
0\}$ so we have the morphisms
\begin{equation}\label{eq:invspec-variante}
\rsect(N;k_N) \simeq \rsect_{\{-t_0\}}(\Real; \H) \isoto \rsect_{[-t_0,-c]}(\Real; \H) \from \rsect_{\{-c\}}(\Real; \H)
\end{equation}
for $t_0\gg0$ and any $c\leq t_0$. To find the spectral invariants we look for the infimum on the $c$ such that a
given class is in the image of the composition of this morphisms. 

\begin{rem} We can see that the infimum in Definition \ref{def:spectral-inv} is actually a
minimum, at least when $\F$ is of the form $\F_L$. Indeed, this follows either by identifying the spectral invariants of the Definition to those obtained by using Floer cohomology, or by appealing to the theory of persistence modules and barcodes (see Appendix \ref{Appendix-barcodes-Floer} Proposition \ref{Prop-A1}).
\end{rem} 
  We notice the following immediate consequence of the definitions.
\begin{prop}\label{Prop-5.4}
  There exists a morphism $u: \F_{L_1} \longrightarrow \F_{L_2}$ such that $u=\Id$ at $+\infty$ if and only if
  $c_-(L_1,L_2)\geq 0$. In particular there is a morphism $k_{N \times [0, +\infty[} \longrightarrow \F$ if and only
  if $c_-(L)\geq 0$.
\end{prop}

\subsection{The case of Hamiltonian maps }\label{Subsection-5.2}

If $\varphi$ is a Hamiltonian map of $T^*\Real^n$ with compact support, we define the spectral invariants of
$\varphi$ as those of its graph, say $\Gamma_\varphi \subset \overline{T^*\Real^{n}}\times T^*\Real^{n}$,
compactified in the following way: we choose a symplectomorphism $\psi\colon \overline{T^*\Real^{n}}\times
T^*\Real^{n} \isoto T^*\Gamma_{\id}$ such that $\psi(\Gamma_{\id}) = 0_{\Gamma_{\id}}$; then $\psi(\Gamma_\varphi)$
coincides with the zero-section outside a compact set and we can compactify it as $\Gamma_\varphi^c$ in the cotangent
bundle of the sphere. In particular $\Gamma_\id^c = 0_{S^{2n}}$ and we set
\CorrectionRed{
\begin{defn} \label{Def-5.7}
The invariants of $\varphi$, $c_+(\varphi)$, $c_-(\varphi)$ are defined as those of $(\Gamma_\varphi^c,0_{S^{2n}})$.
\end{defn}
}
  This is well-defined since another choice
of $\psi$ would give a pair $(\Gamma_\varphi'^c,0_{S^{2n}})$ symplectomorphic to $(\Gamma_\varphi^c,0_{S^{2n}})$.
 
On the other hand we have seen in~\S\ref{sec:qhi} that $\varphi$ has an associated sheaf $\K_\varphi \in
D(\Real^{2n+1})$. In the next paragraph we explain how we can turn the above definition of $c_+(\varphi)$,
$c_-(\varphi)$ into Definition~\ref{def:inv-spec-pour-Hamiltonian} which uses $\K_\varphi$, $\K_\id$ and not the
intermediate symplectomorphism $\psi$.

\subsubsection*{Support in a Darboux chart}\label{Subsection-Floer-Darboux-chart}
Let now $S_\varphi(q,P,\xi)$ be a generating function quadratic at infinity for $\Gamma_\varphi$ over $\Delta_{ {\mathbb
    R}^{2n}}$, that is,
$$
\left \{
\begin{array}{ll}
p-P= \frac{\partial S_\varphi}{\partial q}(q,P;\xi) \\
Q-q= \frac{\partial S_\varphi}{\partial P}(q,P;\xi)\\
\frac{\partial S_\varphi}{\partial \xi}(q,P;\xi)=0\\
\end{array}
\right.
$$
Since $\varphi$ is compactly supported, so is $S_\varphi$ in the $q,P$ variables:  for $ \vert q \vert + \vert P \vert $ large
enough, $S_\varphi(q,P;\xi)=B(\xi)$ where $B$ is a non-degenerate quadratic form. We can then extend $S_\varphi$ to a generating function quadratic at infinity on a bundle over $S^{2n}$ (see
\cite{Viterbo-STAGGF}), and $c_\pm(\varphi)$ are defined as
\begin{align*}
c_+(\varphi)&= \inf \{ t \mid T\cup \mu_{S^{2n}} \neq 0\; \text{in}\; H^*(S^t,S^{-\infty})\} , \\
c_-(\varphi)&= \inf \{ t \mid T\cup 1 \neq 0\; \text{in}\; H^*(S^t,S^{-\infty})\},
\end{align*}
where $T$ is the Thom class associated to the negative bundle of $B$ and $S^{-\infty}$ means $S^{-c}$ for $c$ large
enough.  Note that these coincide with the spectral invariants defined using Floer cohomology (see \cite{Viterbo-FCFH2}).

Now if $\varphi$ is defined on $T^*N$ but has support in a Darboux chart, it can be considered either as a compact
supported Hamiltonian map on $T^*N$ or on ${\mathbb R}^{2n}$ and we can define the spectral invariants in two ways
(using Floer cohomology).  By an easy  argument (see Appendix \ref{Appendix-local-Floer}) we can see that these two definitions of
$c_\pm(\varphi)$ do coincide.

\subsubsection*{Sheaf definition of spectral invariants for a Hamiltonian map}
Now we can give another formulation of the definition of $c_\pm(\varphi)$.  Let $N$ be a compact manifold and let
$\varphi$ be a Hamiltonian map of $T^*N$ with compact support. We consider its homogenized graph $\widehat
\Lambda_\varphi \subset \overline{T^*N}\times T^*N \times T^* {\mathbb R} $ as in~\S\ref{sec:qhi}.  We have
associated to $\varphi$ its quantization $\K_\varphi$, with singular support $\widehat\Lambda_\varphi$, and it
coincides with $\K_\id = k_{\Delta_N \times [0,+\infty[}$ outside a compact subset of $N^2\times\Real$.  Now
$\homst(\K_\id, \K_\varphi)$ has similar properties as $\homst(\F_{L_1},\F_{L_2})$, namely, it is
$k_{\Delta_N\times\Real}[1]$ near $N^2\times \{-\infty\}$ and $0$ near $N^2\times \{+\infty\}$.  As in the case of
compact Lagrangians we have $H^*_{N^2\times\{c\}}(N^2\times\Real; k_{\Delta_N\times\Real}[1]) \simeq H^*(N;k_N)$ and
any class $\alpha \in H^d(N;k_N)$ yields a morphism
$$
u(\alpha, \varphi ,c) \colon \K_\id \to T_{c*}(\K_\varphi) [d],
$$
for $c$ big enough, which restricts to $\alpha \colon k_{\Delta_N\times\Real} \to k_{\Delta_N\times\Real}[d]$ near
$N\times \{+\infty\}$ (using $\Hom(k_{\Delta_N}, k_{\Delta_N}[d]) \simeq H^d(N;k_N)$.

\begin{defn}\label{def:inv-spec-pour-Hamiltonian}
  For a compactly supported Hamiltonian map $\varphi$ of $T^*N$ and $\alpha \in H^d(N;k_N)$ we set
  $$
  c(\alpha, \varphi)
  = -\inf \{c ;\; \exists u\colon \K_\id \to T_{c*}(\K_\varphi) [d],\; u|_{N^2 \times ]t_0,+\infty[} = \alpha,  \; \text{for $t_0\gg0$} \} \dispdot
  $$
  (Here $t_0$ is large enough so that over $N^2 \times ]t_0,+\infty[$ we have $T_{c*}(\K_\varphi) \simeq
  k_{N\times\Real}$.)  We also set $c_-(\varphi) = c(1,\varphi)$, \CorrectionRed{$c_+(\varphi) = -c_-( \varphi^{-1})$} and
  $\gamma(\varphi) = c_+(\varphi) - c_-(\varphi)$.  If $\psi$ is another compactly supported Hamiltonian map, we set
  $\gamma(\varphi, \psi) = \gamma(\varphi \psi^{-1})$.
\end{defn}

We have $c_-(\varphi)=\inf \{c(\varphi,\alpha);\; \alpha \in H^\star(N;k_N)\}$.
Since $\Gamma(\varphi)$ coincides with the zero section outside a compact set, normalizing the primitive of $\lambda$ on $\Gamma(\varphi)$ to be zero outside a compact set, we have
$c_-(\varphi)\leq 0 \leq c_+(\varphi)$ as in \cite{Viterbo-STAGGF}.

\subsubsection*{Equivalence between  Definitions~\ref{Def-5.7} and \ref{def:inv-spec-pour-Hamiltonian}} 
We first remark that we don't really need to compactify $\Gamma_\id$ into the sphere to obtain the invariants.  Let
$D_1$ be the subcategory of $D_{\tau\geq 0}(\Gamma_\id \times \Real)$ of complexes isomorphic to $k_{\Gamma_\id
  \times [c,+\infty[}$ (for some $c$) and $D_2$ be the subcategory of $D_{\tau\geq 0}(S^{2n} \times \Real)$ of
complexes isomorphic to $k_{S^{2n} \times [c,+\infty[}$ near $\{\infty\} \times [c,+\infty[$.  Then the
compactification and the restriction map induce functors $\alpha \colon D_1 \to D_2$ and $\beta\colon D_2 \to D_1$
such that $\beta \circ \alpha = \id$, $\alpha(\tau_{a,b}(\F)) = \tau_{a,b}(\alpha(\F))$.  We deduce that
$c_-(\alpha(\F_1), \alpha(\F_2)) = c_-(\F_1,\F_2)$ (we remark that $c_-$ makes sense even in the non compact
case). In particular
\begin{equation}\label{eq:inv-spect-compactification}
  c_-(\varphi) = c_-(\F_{\Gamma_\varphi^c}, \F_{0_{S^{2n}}}) = c_-(\F_{\Gamma_\varphi}, \F_{0_{\Real^{2n}}}) \dispdot
\end{equation}
Now the symplectomorphism $\psi$ induces an equivalence of categories between $\K_\psi^\convstar \colon D_{\tau\geq
  0}(\Real^{2n+1}) \isoto D_{\tau\geq 0}(\Gamma_\id \times \Real)$ which commutes with the translation $T_{c*}$ in
the last variables.  We have $\K_\psi^\convstar( \K_\varphi) ) \simeq \F_{\Gamma_\varphi}$ and we obtain
from~\eqref{eq:inv-spect-compactification} that the spectral invariant $c_-(\varphi)$ we have defined above coincides
with the one in Definition~\ref{def:inv-spec-pour-Hamiltonian} below.  To obtain $c_+$ we use $c_+(\F_1,\F_2) =
-c_-(\F_2,\F_1)$.

The existence of $\K_\psi$ doesn't really follows from~\S\ref{sec:qhi} because $\psi$ is not compactly supported.
But the choice of $\psi$ is irrelevant and we can choose $\psi \colon \overline{T^*\Real^{n}} \times T^*\Real^{n} \to
T^*\Delta_{\Real^{2n}}$ 
 given by
$\psi(q_1,q_2,p_1,p_2) = (q_1,p_2,p_1+p_2,q_1-q_2)$ . In this case we can give a direct construction of $\K_\psi$ as
follows.  Let us define $\Sigma^+ \subset \Real^{4n+1}$ by $\Sigma^+ = \{(q_1,q_2,Q_1,Q_2,t) \mid Q_1=q_1,\; t \geq
(q_1-q_2)Q_2 \}$, whose boundary is the projection of the lift $\widehat \Lambda_\psi$ of the graph of $\psi$ to the
base.  Then we can check that $\K_\psi = k_{\Sigma^+}$ is the quantization of $\psi$.

This justifies that both definitions coincide for a Hamiltonian map of $T^*S^n$ obtained from a compactly supported
map of $T^*\Real^n$.  We saw in a previous  paragraph that the case of a Hamiltonian map of $T^*N$ with support in a
Darboux chart is reduced to this case.

Note the following analogue of Proposition \ref{Prop-5.4}
\begin{lem} \label{Lemma-5.5}
Let $\varphi$ be a compact supported  Hamiltonian map. Then there is a morphism   $  \K_\id \longrightarrow  \K_\varphi $ in $D^b(N\times N \times {\mathbb R} )$ equal to $\Id$ at $+\infty$ if and only if $c_-(\varphi)=0$.
Analogously there is a morphism $ \K_\varphi \longrightarrow \K_\id $ in $D^b(N\times N \times {\mathbb R} )$ if and only if $c_+(\varphi)=0$.
\end{lem} 
\begin{proof} 
\CorrectionRed{
The first statement follows from Proposition \ref{Prop-5.4}. For the second one, we have by assumption a morphism
$\K_\id   \longrightarrow \K_{\varphi^{-1}} $. Composing by $\K_\varphi$ we get a morphism $ \K_\varphi \longrightarrow \K_\id $. Analogously, if we have such a morphism, composing by $\K_{\varphi^{-1}}$ we get $c_+(\varphi)=0$.
}
\end{proof} 
 \subsection{Defining \texorpdfstring{$\gamma$}{gamma}-coisotropic subsets} 
 For an open set $U$ in the symplectic manifold $(M, \omega)$ we denote by $\DHam_c(U)$ the set of time one flows of Hamiltonian with compact support contained in $U$. 
 
 \begin{defn} [see definition 6.1 in \cite{Viterbo-gammas}]
   Let  $V$  be a set in $(M, \omega)$. We say that $V$ is {\bf $\gamma$-coisotropic at $x\in V$} if
   \CorrectionBlue{ there exists $\eps>0$ such that, for all $\eta$ with $0< \eta< \eps$, there is a $\delta>0$ such
     that, for all $\varphi$ in $\DHam_{c}(B(x, \eps))$ satisfying $\varphi(V) \cap B(x, \eta)=\emptyset$, we
     have $\gamma (\varphi)> \delta$. }

We shall say that $V$ is {\bf $\gamma$-coisotropic} if it is closed, non-empty and $\gamma$-coisotropic at each $x\in V$. 
 \end{defn} 
 We refer to  \cite{Viterbo-gammas}, in particular section 6,  for properties and examples of $\gamma$-coisotropic subsets. In particular we remind the reader of the following result from 
 \cite{Viterbo-gammas}.
 \begin{prop} 
 Let  $\varphi$ be a homeomorphism preserving $\gamma$, for example a $C^0$-limit of Hamiltonian maps (usually called Hamiltonian homeomorphisms) and $V$ a $\gamma$-coisotropic set. Then, 
 the image by $\varphi$ of $V$ is a $\gamma$-coisotropic set. 
 \end{prop} 
 
Note that the completion of $\LL (T^*N)$ for $\gamma$  is denoted by $\widehat{\LL} (T^*N)$. In \cite{Viterbo-gammas}
the $\gamma$-support of an element $L$  in $\widehat{\LL} (T^*N)$ was defined, and it was proved that it must be a $\gamma$-coisotropic subset.

\section{The  \texorpdfstring{$\gamma_g$}{gamma-g}-metric on sheaves}
\label{sec:Gtopology}

We shall define the analogue of the $\gamma$ distance between sheaves.  It is also a sheaf analogue of the barcode
distance (see \cite{K-S-distance} or \cite{Asano-Ike}) and relies on the morphism $\tau$ introduced by Tamarkin
(see~\cite{Tamarkin} or~\cite[Prop. 4.8]{G-K-S}) that we have already encountered in~\eqref{eq:tau-version1}.  

To
define $\tau$ in general we recall that the subcategory $D_{\tau\geq 0}(N\times\Real)$ of $D(N\times\Real)$ 
can be characterized as follows (see~\cite[Prop.~5.2.3, 3.5.4]{K-S}).  
Let $P$ be the endofunctor of $D(N\times\Real)$
defined by $$P(\F) = R s_*(\F\boxtimes k_{[0,+\infty[})$$ with $s\colon N\times\Real^2 \to N\times\Real$, $(x,t_1,t_2)
\mapsto (x,t_1+t_2)$ (this is almost $\F \cstar k_{N \times [0,+\infty[}$ up to replacing from $Rs_!$ to $Rs_*$).  Then the natural morphism 
$k_{[0,+\infty[} \to k_{\{0\}}$ induces a morphism $P(\F) \to \F$ and $D_{\tau\geq 0}(N\times\Real)$ is  the set of
$\F$ such that $P(\F) \to \F$ is an isomorphism.  Now we have in general $R s_*(\F\boxtimes k_{\{c\}}) \simeq T_{c*}(\F)$,
where $T_c(x,t) = (x,t+c)$ and we obtain in the same way, using the morphism $k_{[0,+\infty[} \to k_{\{c\}}$, a
morphism $P(\F) \to T_{c*}(\F)$ for any $c\geq 0$.  In particular we have a morphism of endofunctors defined on
$D_{\tau\geq 0}(N\times\Real)$ for any $c\geq 0$
$$
\tau_{0,c}(\F) \colon \F \to T_{c*}(\F) \dispdot
$$
We shall call this morphism the Tamarkin morphism.

The case of $D_{\tau\geq 0}(N\times\Real)$ is in fact sufficient for this paper but we can consider a slightly more
general case without modifying the proofs. Let $g \colon T^*N \setminus 0_N \to  {\mathbb R} $ be a
$1$-homogeneous autonomous Hamiltonian function and let $D_{g\geq 0}(N)$ be the full subcategory of $D(N)$ formed by
the $\F$ with $SS^\bullet(\F) \subset \{g \geq 0\}$.  We let $\varphi^t = \varphi_g^t \colon T^*N\setminus 0_N \to
T^*N \setminus 0_N$ be the corresponding homogeneous flow.  We set for short $\fai{\varphi}^a(\F) = \F \circ
K_{\varphi^a}$.

We consider the lifting of the graph of the whole isotopy $\varphi^t$ to Lagrangian in $T^*(N \times N \times {\mathbb R} )$  given by
$$
\left\{(q,p, Q_{t}(q,p),P_{t}(q,p), t, g(q,p))\right\} ,
$$
where $\varphi^{t}(q,p)=(Q_{t}(q,p),P_{t}(q,p))$ and the associated kernel $K_\varphi \in D(N^2\times\Real)$. Then
$K_\varphi(\F) \in D(N\times\Real)$ is such that $K_\varphi(\F)|_{N \times \{a\}} = \fai{\varphi}^a(\F)$.  Now, if
$\F \in D_{g\geq 0}(N)$, then $K_\varphi(\F) \in D_{\tau\geq 0}(N\times\Real)$.  Indeed a point of
$SS^\bullet(K_\varphi(\F))$ is written $(q,p, Q,P, t,\tau)$ with $(q,p) \in SS^\bullet(\F)$ and $\tau = g(q,p)$.  We
thus have the Tamarkin morphism $\tau_{0,c}(K_\varphi(\F)) \colon K_\varphi(\F) \to T_{c*}(K_\varphi(\F))$ for $c\geq 0$. Restricting to
$N\times\{a\}$ and setting $b=a+c$ we obtain
\begin{equation}
  \label{eq:morphismetau}
  \tau_{a,b}(\F) \colon \fai{\varphi}^a(\F) \to \fai{\varphi}^b(\F),
  \qquad \text{for $a\leq b$.}
\end{equation}
The morphisms $\tau_{a,b}$ are invariant by $\fai{\varphi}^t$, for any $t$, and compatible with composition:
\begin{equation}\label{eq:compatibilite-tau}
  \begin{aligned}
  \fai{\varphi}^t(\tau_{a,b}(\F) )& = \tau_{a+t,b+t}(\fai{\varphi}^t(\F)), \\
  \tau_{a,c}(\F) &= \tau_{b,c}(\F) \circ  \tau_{a,b}(\F), \quad \text{ for } a\leq b\leq c \dispdot
  \end{aligned}
\end{equation}
\begin{examples} \label{Example-6.1 } 
\leavevmode
\begin{enumerate} 
\item Let $N= {\mathbb R} $ and $g$ be the norm given by the standard metric, so
  $g(t,\tau)= |\tau|$. Then $\varphi^a(t,\tau)=(t+a\frac{\tau}{ \vert \tau \vert}, \tau)$.  As a result
  $K_{\varphi^a}(k_{[x,y[}) = k_{[x+a, y+a[}$ and $SS(K_{\varphi^a})=\left \{(t, t+a\frac{\tau}{ \vert \tau \vert},\tau, \tau)
  \right \}\subset T^*( {\mathbb R}^{2})$.
\item 
  In the case of $N\times\Real$ and $g=\tau$ we recover the case we introduced in the first paragraph (and originally due to  Tamarkin).
\item \label{Example-6.1-3} 
  If $g$ is non negative we have $D_{g\geq 0}(N) = D(N)$. Typically $g$ is the norm given by a Riemannian metric and
  $\varphi_g$ is the normalized geodesic flow.  In this case, denoting by $B(x,r)$ the open ball at $x$ of radius
  $r$, we find
  $$
  \fai{\varphi}^a(k_{B(x,r)}) \simeq
  \begin{cases}
    k_{B(x,r+a)} & \text{for $-r < a < r_{inj} -r$,} \\
     k_{\overline{B(x,-r-a)}}[-\dim N]  \quad & \text{for $-r_{inj}-r < a\leq -r$,}
   \end{cases}
   $$
   and 
    $$
  \fai{\varphi}^a(k_{\overline{B(x,r)}}) \simeq
  \begin{cases}
    k_{\overline{B(x,r-a)}} & \text{for $ r-r_{inj} \leq a \leq r $,} \\
     k_{{B(x,a-r)}}[\dim N]  \quad & \text{for $r < a <r+r_{inj}$,}
   \end{cases}
   $$
  where $r_{inj}$ is the injectivity radius. 
  \end{enumerate} 
\end{examples}

\subsection{The \texorpdfstring{$\gamma_g$-topology}{gamma-{g}-topology}}

The following definition is the analogue of the spectral metric (in the case $g=\tau$) (and a sheaf analogue of the barcode distance). It is introduced in \cite{K-S-distance} or
\cite{Asano-Ike} (up to a small modification).
\begin{defn}\label{def:dist-gammag}
  Let $g \colon T^*N \setminus 0_N \to {\mathbb R} $ be a $1$-homogeneous autonomous Hamiltonian function and $\varphi=\varphi_g^t$ its  flow.
  For $\F,\G \in D_{g\geq 0}(N)$, we let $\gamma_g(\F, \G) \in [0,\infty]$ be the infimum of the $c\geq 0$ for which
  there exist $a,b \geq 0$ such that $a+b=c$ and such that there are morphisms
  $$
  u \colon \F \to \fai{\varphi}^a(\G),
  \qquad
  v \colon \G \to \fai{\varphi}^b(\F)
  $$
  satisfying $\fai{\varphi}^a(v) \circ u = \tau_{0,a+b}(\F)$ and
  $\fai{\varphi}^b(u) \circ v = \tau_{0,a+b}(\G)$.
\end{defn}
The authors thank N.~Vichery for the next remark.
\begin{rem}\label{rem:gamma_g=gamma_tau}
  The case $D_{g\geq 0}(N)$ is not really more general than the case $D_{\tau\geq 0}(N\times\Real)$. Indeed the
  functor $K_\varphi(-) \colon D(N) \to D(N\times\Real)$ is fully faithful and we have $T_{a*} (K_\varphi(-))
  \simeq K_\varphi(\fai{\varphi}^a(-))$.  Hence, for $\F,\G \in D_{g\geq 0}(N)$ we have $\gamma_g(\F, \G) =
  \gamma_\tau(\K_\varphi(\F), \K_\varphi(\G))$.
\end{rem}

The definition of $\gamma_{g}$ in~\cite{K-S-distance} imposes $a=b$ which leads to a bigger distance, say $\gamma'_g$, with
$\gamma_g \leq \gamma'_g \leq 2 \gamma_g$.  The definition in~\cite{Asano-Ike} only asks $\fai{\varphi}^a(v) \circ u
= \tau_{0,a+b}(\F)$ and $\fai{\varphi}^b(u') \circ v' = \tau_{0,a+b}(\G)$, for two possibly different morphisms $u'$,
$v'$ (see the arXiv version for another definition). This leads to a smaller distance but we don't know if it is equivalent to $\gamma_g$.
Lemma~\ref{lem:comparaison-distances} below says that $\gamma_g$ is equivalent to 
$\inf\{\gamma_g(\C(u),0)$; $u\colon \F \to \G\}$, where $\C(u)$ denotes the cone of $u$.

\begin{lem}\label{lem:proprietes-gammag}
  For $\F,\G,\H \in D_{g\geq 0}(N)$ we have
\begin{enumerate} 
  \item \label{lem:proprietes-gammag-i} $\gamma_g(\F, \G) = \gamma_g(\fai{\varphi}^t(\F), \fai{\varphi}^t(\G)) = \gamma_g(\G,\F)$.

  \item  \label{lem:proprietes-gammag-ii}  $\gamma_g(\F, \H) \leq \gamma_g(\F, \G) + \gamma_g(\G, \H)$.

  \item  \label{lem:proprietes-gammag-iii}  $\gamma_g(0,\F) = \inf \{t\geq 0$; $\tau_{0,t}(\F) =0\}$.

  \item  \label{lem:proprietes-gammag-iv}  If $\F$ satisfies $\gamma_g(\F,0) = 0$, then $\F \simeq 0$.
  
  \item \label{lem:proprietes-gammag-v} For any distinguished triangle $\F \to[u] \G \to[v] \H \smash{\to[+1]}$, we have $$\gamma_g(0, \G) \leq
    \gamma_g(0,\F) + \gamma_g(0, \H)$$

  \item  \label{lem:proprietes-gammag-vi}  For composable morphisms $u$, $v$ with cones $\C(u)$, $\C(v)$, $\C(v\circ u)$, we have
    $\gamma_g(0,\C(v\circ u)) \leq \gamma_g(0,\C(u)) + \gamma_g(0,\C(v))$.
\end{enumerate} 
\end{lem} 
\begin{proof} 
  (\ref{lem:proprietes-gammag-i}) follows from~\eqref{eq:compatibilite-tau} and~(\ref{lem:proprietes-gammag-iii}) follows from the definition of $\gamma_g$.  The triangular
  inequality~(\ref{lem:proprietes-gammag-ii}) is proved in~\cite[Prop. 4.11]{Asano-Ike}.

  \smallskip Let us prove~(\ref{lem:proprietes-gammag-iv}).  By Remark~\ref{rem:gamma_g=gamma_tau} we can assume that $N = M\times\Real$ and $g =
  \tau$.  Let $x_0 \in M$ be given and let $i\colon \{x_0\}\times\Real \to M\times\Real$ be the inclusion.  We have
  $\tau_{a,b}(i^{-1}\F) = i^{-1}(\tau_{a,b}(\F))$ and it follows that $\gamma_g(i^{-1}\F,i^{-1}\G) \leq
  \gamma_g(\F,\G)$. If $\F\not= 0$, there exists $x_0$ such that $i^{-1}(\F) \not=0$.  Hence we can assume that
  $N=\Real$. We can also take $g = |\tau|$ (it is differentiable on $T^*\Real\setminus 0_\Real$) instead
  of $g = \tau$. The corresponding $\tau$ morphisms, say $\tau^g_{a,b}$, coincide with $\tau_{a,b}$ on $D_{\tau\geq
    0}(\Real)$ but they are defined on $D(\Real)$. Up to shifting $\F$ in degrees and translating, we can assume that
  $H^0\F_0 \not= 0$. Hence there exists $u \in H^0(\mathopen]-\varepsilon, \varepsilon[; \F)$ with $u_0 \not=0$.
  Interpreting $u$ as a morphism and choosing a map $v: \F_0 \longrightarrow k$ non-vanishing on the germ defined by $u$, we get that the composition of $v$ and $u$;   $k_{]-\varepsilon,
    \varepsilon[} \to[u] \F \to[v] k_{\{0\}}$  is the canonical morphism. Now the morphisms
  $\tau^g_{0,\delta}(-)$, for $0<\delta< \varepsilon$ give the commutative diagram (using Example~\ref{Example-6.1 } (\ref{Example-6.1-3}) ) where the vertical arrows are given by $\tau^g_{0,\delta}$
  $$
  \xymatrix{
    k_{]-\varepsilon, \varepsilon[} \ar[r]^u \ar[d] & \F \ar[r]^v \ar[d] & k_{\{0\}} \ar[d]_w \\
    k_{]-\varepsilon+\delta, \varepsilon-\delta[} \ar[r] & T_{\delta *}(\F) \ar[r]  & k_{]-\delta,\delta[}[1] \dispdot
  }
  $$
  \CorrectionBlue{ We check that $w\circ v \circ u \not=0$; this then implies that $\tau^g_{0,\delta}(\F) \not=0$, as
    required. Since $w\not=0$, we have a distinguished triangle $k_{]-\delta,\delta[} \to k_{]-\delta,0]}
    \oplus k_{[0,\delta[} \to k_{\{0\}} \to[w] k_{]-\delta,\delta[}[1]$ (in other words $w$ gives a non trivial
    Yoneda extension).  If $w \circ (v\circ u)$ were $0$, then $v\circ u$ would factorize through
    $k_{]-\delta,0]} \oplus k_{[0,\delta[}$. But $\Hom(k_{]-\varepsilon, \varepsilon[}, k_{]-\delta,0]}) =
    \Hom(k_{]-\varepsilon, \varepsilon[}, k_{[0,\delta[}) = 0$ because $\varepsilon>\delta$ and we cannot have
    such a factorization.}    
  
\smallskip
Let us prove~(\ref{lem:proprietes-gammag-v}) (a similar but more detailed argument will be given in the proof
of Lemma~\ref{lem:comparaison-distances}-(i)). If $\tau_{0,a}(\F) =0$, then $\tau_{0,a}(\G) \circ u =
\fai{\varphi}^a(u) \circ \tau_{0,a}(\F) = 0$ and $\tau_{0,a}(\G)$ factorizes through $v$.  In the same way, if
$\tau_{0,b}(\H) =0$, then $\tau_{0,b}(\G)$ factorizes through \CorrectionBlue{$\fai{\varphi}^b(u)$}. Hence
$\tau_{0,a+b}(\G)$ factorizes via $\fai{\varphi}^b(v) \circ \fai{\varphi}^b(u) =0$. We conclude
with~(\ref{lem:proprietes-gammag-iii}).

  We deduce~(\ref{lem:proprietes-gammag-vi}) from~(\ref{lem:proprietes-gammag-v}) since, by the octaedron axiom, the three cones appear in a distinguished triangle.
\end{proof} 

If $N$ is real analytic and $\F,\G$ are constructible and satisfy $\gamma_g(\F,\G) = 0$, then $\F \simeq \G$
(see~\cite{Petit-Schapira-Waas} and also Proposition~\ref{prop:unicite-limite} for limits of constructible sheaves).  However $\gamma_g$ is only a pseudo-metric as shown by Berkouk and Ginot (see \cite{Berkouk-Ginot}, proposition 6.9) and the following example.
\begin{example}
  We define two sheaves on the real line $\F = \bigoplus_{x\in \Q}k_{[x,\infty[}$ and
  $\G = \bigoplus_{x\in \Q^\times}k_{[x,\infty[}$, where
  $\Q^\times = \Q \setminus \{0\}$.  We remark that
  $\fai{\varphi}^t(\F) \simeq T_{t*}(\F)$, where $T_t(x) = x +t$.  For a given
  $\varepsilon >0$ we choose a bijection $f \colon \Q \to \Q^\times$ such that
  $x\leq f(x) \leq x+\varepsilon$, for all $x\in \Q$, and we define
  $u = \bigoplus_{x\in \Q} u_x \colon \F \to \G$,
  $v = \bigoplus_{x\in \Q^\times} v_x\colon \G \to T_{\varepsilon *}\F$ where
  $u_x \colon k_{[x,\infty[} \to k_{[f(x),\infty[}$ and
  $v_x \colon k_{[x,\infty[} \to k_{[f^{-1}(x)+\varepsilon,\infty[}$ are the natural
  morphisms.  Using $u,v$ we see that $\gamma_g(\F,\G) \leq \varepsilon$. Hence
  $\gamma_g(\F,\G) =0$.
\end{example}
\begin{defn} 
  We denote by $\gamma_g$ the topology associated to the pseudo-distance
  $\gamma_g$. In other words a sequence $(\F_j)_{j\geq 1}$ in $D_{g\geq 0}(N)$ $\gamma_g$-converges to $\F$
  if and only if $\lim_j \gamma_g(\F_j, \F)=0$. We may write abusively $\gamma_g-\lim \F_j = \F$
  although $\gamma_g-\lim \F_j$ is not well-defined.
\end{defn} 
\begin{rem}
In view of the Lemma \ref{lem:comparaison-distances}, a sequence $(\F_n)_{n\in\N}$ converges to $\F$ if and only if
there exist a sequence of positive numbers $(\varepsilon_n)_{n\in\N}$ converging to
$0$ and morphisms $u_n \colon \F_n \to \fai{\varphi}^{\varepsilon_n}(\F)$ such that
$\gamma_g(0, \C(u_n))$ converges to $0$.
\end{rem}

Part~(\ref{lem:comparaison-distances-ii}) in the next lemma is analogue to~\cite[Lem.~4.14]{Asano-Ike}.

\begin{lem}\label{lem:comparaison-distances}
  Let $u\colon \F \to \F'$ be a morphism in $D_{g\geq 0}(N)$ and let $\C$ be the cone of $u$.
\begin{enumerate} 
\item \label{lem:comparaison-distances-i} If there exists $v \colon \F' \to \fai{\varphi}^\varepsilon(\F)$ such that
  $v\circ u=\tau_{0,\varepsilon}$ and $\fai{\varphi}^\varepsilon(u) \circ v=\tau_{0,2\varepsilon}$, then $\gamma_g(\C,0) \leq 2\varepsilon$.
\item  \label{lem:comparaison-distances-ii} If $\gamma_g(\C,0) < \varepsilon$, then there exists
  $v \colon \F' \to \fai{\varphi}^{2\varepsilon}(\F)$ such that $v\circ u= \tau_{0,2\varepsilon}$ and
  $\fai{\varphi}^{2\varepsilon}(u) \circ v=\tau_{0,4\varepsilon}$.
  \end{enumerate} 
\end{lem}
\begin{proof}
  (\ref{lem:comparaison-distances-i}) We consider the morphism of triangles given by the functorial
  morphism~\eqref{eq:morphismetau},
  $$
  \xymatrix{
    \F \ar[r]^u \ar[d] &  \F' \ar[r]^\alpha \ar[d] \ar@{.>}[dl]_v \ar[dr]^f
    &   \C \ar[r]^\beta \ar[d] \ar[dr]^g &  \F[1] \ar[d]  \\
    \fai{\varphi}^\varepsilon(\F) \ar[r] &  \fai{\varphi}^\varepsilon(\F') \ar[r]
    &   \fai{\varphi}^\varepsilon(\C) \ar[r] &  \fai{\varphi}^\varepsilon(\F)[1]   \dispdot
  }
  $$
  Here all small triangles commute (which defines $f$ and $g$). We have
  $f = \fai{\varphi}^\varepsilon(\alpha) \circ \tau_{0,\varepsilon}(\F') =
  \fai{\varphi}^\varepsilon(\alpha \circ u) \circ v = 0$. In the same way $g$ also
  vanishes.  Since $\tau_{0,\varepsilon}(\C) \circ \alpha = f$ vanishes,
  $\tau_{0,\varepsilon}(\C)$ factorizes through $\beta$ as
  $\tau_{0,\varepsilon}(\C) = h \circ \beta$.  Hence
  $\tau_{0,2\varepsilon}(\C) = \fai{\varphi}^\varepsilon(\tau_{0,\varepsilon}(\C))
  \circ \tau_{0,\varepsilon}(\C) = \fai{\varphi}^\varepsilon(h) \circ
  \fai{\varphi}^\varepsilon(\beta) \circ \tau_{0,\varepsilon}(\C) =
  \fai{\varphi}^\varepsilon(h) \circ g$ vanishes, which means
  $\gamma_g(\C,0) \leq 2\varepsilon$.

  \medskip\noindent (\ref{lem:comparaison-distances-ii}) Since $\tau_{0,\varepsilon}(\C) =0$, both morphisms $f$ and
  $g$ vanish.  Since $\fai{\varphi}^\varepsilon(\alpha) \circ \tau_{0,\varepsilon}(\F') =f$ vanishes,
  $\tau_{0,\varepsilon}(\F')$ factorizes as $\tau_{0,\varepsilon}(\F') = \fai{\varphi}^\varepsilon(u) \circ w$ for
  some $w$.

  In the same way, the vanishing of $g$ gives a factorization $\tau_{0,\varepsilon}(\F) = w' \circ u$ but we cannot
  say that $w=w'$.

  We set $v = \fai{\varphi}^\varepsilon(w') \circ \tau_{0,\varepsilon}(\F') \colon \F' \to
  \fai{\varphi}^{2\varepsilon}(\F)$.  In the following computation we omit abusively to write
  $\fai{\varphi}^\varepsilon$ and write $\tau$ for whatever $\tau_{-,-}(-)$.  A diagram chase on the diagram below
  gives $v \circ u = w' \circ \tau \circ u = w' \circ u \circ \tau = \tau \circ \tau = \tau$ and $u \circ v = u \circ
  w' \circ \tau = u \circ w' \circ u \circ w =u \circ \tau \circ w = \tau \circ u \circ w = \tau \circ \tau = \tau$,
  as required.
  $$
  \xymatrix{
    \F \ar[r]^u \ar[d] &  \F' \ar[d] \ar[dl]_{w}^\circlearrowleft  \\
    \fai{\varphi}^\varepsilon(\F) \ar[r]^u \ar[d]
    &  \fai{\varphi}^\varepsilon(\F') \ar[d] \ar[dl]^{w'}_\circlearrowright  \\
 \fai{\varphi}^{2\varepsilon}(\F) \ar[r]_u &  \fai{\varphi}^{2\varepsilon}(\F')
}
$$
\end{proof}

\subsection{Link with spectral invariants}

In this paragraph we work on $N\times\Real$ and choose $g=\tau$.  So we work with the category
$D_{\tau\geq0}(N\times\Real)$ as in~\S\ref{sec:specinv}.  We prove Propositions~\ref{prop:gamma-et-gammag1}
and~\ref{prop:gamma-et-gammag2} which explain the relations between the $\gamma_\tau$-distance and the spectral
invariants.  

We first give a useful variation on the Morse theorem for sheaves.  The following result is mentioned
in~\cite[\S6.1]{nadler2016noncharacteristic}.

\begin{prop}\label{prop:homFtGt_indpdtt}
  Let $M$ be a manifold and $I$ an open interval of $\Real$.  Let $\F \in D(M), \G \in D(M \times I)$ and set $\G_t=
  \G|_{M\times \{t\}}$ for $t \in I$.  We assume
  \begin{enumerate}
  \item\label{prop:homFtGt_indpdtt-i} there exists a compact set $C$ of $M$ such that $\G|_{(M\setminus C)\times I}$ is of the form $\G'
    \boxtimes k_I$ for some $\G' \in D(M\setminus C)$,
  \item\label{prop:homFtGt_indpdtt-ii}  $SS^\bullet(\G) \cap (0_M \times T^*I) = \emptyset$
  \item \label{prop:homFtGt_indpdtt-iii} $(SS^\bullet(\F) \times T^*I) \cap SS^\bullet(\G) = \emptyset$. 
    \end{enumerate}
  Then $\RHom(\F, \G_t)$ does not depend on $t\in I$.
\end{prop}
\begin{proof}
  We let $i_t \colon M \times \{t\} \to M\times I$, $t\in I$, be the inclusion.  We set $\H = \rhom(\F\boxtimes
  k_I,\G)$.  By~\cite[Prop.~5.4.14]{K-S},  condition~(\ref{prop:homFtGt_indpdtt-iii}) gives $SS(\H) \subset \Lambda := (SS(\F) \times 0_I)^a +
  SS(\G)$.  We can see moreover that $\Lambda$ is also non-characteristic for all $i_t$ (that is, $\Lambda \cap (0_M
  \times T^*I) \subset 0_{M\times I}$).  Using condition (\ref{prop:homFtGt_indpdtt-ii}) $i_t^{-1}\G \simeq i_t^!\G [1]$ and $i_t^{-1}\H \simeq i_t^!\H [1]$
  by~\cite[Prop.~5.4.13]{K-S} and we deduce $\rhom(\F,\G_t) \simeq i_t^{-1}\H$ by the adjunction formula
  $i_t^!\rhom(-,-) \simeq \rhom(i_t^{-1}(-), i_t^!(-))$. Hence $\RHom(\F, \G_t) \simeq \rsect(M; i_t^{-1}\H)$.
  By the condition (\ref{prop:homFtGt_indpdtt-i}) $\H|_{(M\setminus C)\times I}$ is of the form $\H' \boxtimes k_I$, for some $\H' \in
  D(M\setminus C)$, and we can use the base change formula and the bound for the singular support of a direct image.
  Hence $\rsect(M; i_t^{-1}\H) \simeq (Rq_*(\H))_t$, where $q \colon M \times I \to I$ is the projection and it is
  enough to check that $Rq_*(\H)$ is constant.  Since $\Lambda \cap (0_M \times T^*I) \subset 0_{M\times I}$, we have
  $SS(Rq_*(\H)) \subset 0_{I}$ and the result follows.
\end{proof}
\begin{cor}\label{cor:restriction-a-infini}
  Let $\F \in D_{\tau\geq 0}(N\times\Real)$ and $t_0$
  satisfy~\eqref{eq:hyp_faisc_infini}.  We assume that $SS(\F) \cap SS(T_{c*}(\F)) \subset 0_{N\times\Real}$ for any
  $c\not=0$.  Then the restriction map
  $$
  \RHom(\F, T_{c*}(\F)) \to \RHom(\F|_U, T_{c*}(\F)|_U) \simeq \RHom(k_U, k_U) \simeq \rsect(N; k_N)
  $$
  is an isomorphism for any $c> 0$, where $U = N\times ]c+t_0,+\infty[$.  In particular, if $u \colon \F \to
  T_{c*}(\F)$ is $\id$ near $+\infty$, then $u = \tau_{0,c}(\F)$.
\end{cor}
When $SS^\bullet(\F)$ is the cone over a Legendrian, the hypothesis means that there is no Reeb chord.  In particular
the hypothesis is satisfied for the sheaves $\F_L$, $L \in \LL(T^*N)$.
\begin{proof}
  We are going to apply Proposition \ref{prop:homFtGt_indpdtt} with $M=N\times \Real$,  $I = ]0,d]$ for $d>0$ and  $\G = \delta^{-1}(\F)$, where $\delta \colon N\times\Real^2 \to N\times\Real$ is the map $$(x,t_1,t_2)
  \mapsto (x,t_1-t_2)$$ so that $\G|_{N\times\Real\times \{c\}} \simeq T_{c*}(\F)$.  
  
  For $a<b$, $\G$ is constant on $N
  \times \mathopen]-\infty,a-t_0\mathclose[ \times [a,b]$ and $N \times \mathopen]b+t_0, +\infty\mathclose[ \times [a,b]$. Hence $\G$ satisfies~(\ref{prop:homFtGt_indpdtt-i}) in
  Proposition~\ref{prop:homFtGt_indpdtt}.  It is not difficult
  to check~(\ref{prop:homFtGt_indpdtt-ii}) since its singular support is contained in the conormal of $t_1=t_2$ that is $\tau_1+\tau_2=0$. For ~(\ref{prop:homFtGt_indpdtt-iii}) it follows from the hypothesis on $SS(\F)$.  The proposition gives $\RHom(\F, T_{c*}(\F))
  \simeq \RHom(\F, T_{d*}(\F))$, for any $0< c \leq d$.
\end{proof}

\begin{prop}\label{prop:gamma-et-gammag1}
  Let $L_1, L_2 \in \LL(T^*N)$.  Then
  $$
\gamma_\tau(\F_{L_1}, \F_{L_2}) =
\begin{cases}
  c_+(L_1,L_2) & \text{if $c_-(L_1,L_2) >0$,} \\
  -c_-(L_1,L_2) & \text{if $c_+(L_1,L_2) <0$,} \\
  c_+(L_1,L_2) -  c_-(L_1,L_2) \quad & \text{else.}
\end{cases}
$$  
\end{prop}
\begin{proof}
  We recall that $\F_{L_i} \simeq k_{N\times\Real}$ near $+\infty$ and that $-c_-(L_1,L_2)$ is the infimum on the $a$
  such that there exists $u\colon \F_{L_1} \to T_{a*}(\F_{L_2})$ restricting to $\id_{k_{N\times\Real}}$ near
  $+\infty$.  Similarly $c_+(L_1,L_2)$ is the infimum on the $b$ such that there exists $v\colon \F_{L_2} \to
  T_{b*}(\F_{L_1})$ restricting to $\id_{k_{N\times\Real}}$ near $+\infty$.  We then have $T_{a*}(v) \circ u =
  \tau_{0,a+b}(\F_{L_1})$ and $T_{b*}(u) \circ v = \tau_{0,a+b}(\F_{L_2})$ by Corollary~\ref{cor:restriction-a-infini}.
  Hence $u$, $v$ satisfy Definition~\ref{def:dist-gammag}, except that here $a,b$ may be negative.  Assuming moreover
  $a,b\geq 0$ we find that $\gamma_\tau(\F_{L_1}, \F_{L_2})$ is less than the right hand side of the equality in the
  statement.
  
  Conversely, if $\gamma_\tau(\F_{L_1}, \F_{L_2}) <c$, there exists $a,b,u,v$ as above with $a+b = c$, $a,b\geq0$ but
  $u$, $v$ only satisfy the condition on the composition.  Near $+\infty$ we only know that they are $u = \lambda_u
  \id_{k_{N\times\Real}}$, $v = \lambda_v \id_{k_{N\times\Real}}$ (indeed $\Hom(k_{N\times\Real}, k_{N\times\Real})
  \simeq k$).  Since $T_{a*}(v) \circ u = \tau_{0,a+b}(\F_{L_1})$, we have $\lambda_u \lambda_v =1$.  Hence
  $\lambda_u$, $\lambda_v$ are not zero and, replacing $u$, $v$ by a non zero multiple, we can assume that
  $u=v=\id_{k_{N\times\Real}}$ near $+\infty$.  We deduce $a \geq \max\{-c_-(L_1,L_2),0\}$, $b \geq
  \max\{c_+(L_1,L_2),0\}$. 
\end{proof}
\CorrectionRed{
\begin{rem} \label{Rem-6.13}
Note that in \cite{Viterbo-gammas} the metric on $\LL(T^*N)$ is  defined by $c(L_1,L_2)= \max\{0, c_+(L_1,L_2)\}- \min\{0,  c_-(L_1,L_2)\} $. We thus get 
 \end{rem} 
\begin{prop}\label{Prop-6.13} The map $L \mapsto \F_L$ is then an isometric embedding. 
\end{prop} 
}

The proof of Proposition~\ref{prop:gamma-et-gammag1} works exactly the same way for Hamiltonian maps, replacing the
condition ``$\F_{L_i}\simeq k_{N\times\Real}$ near $+\infty$'' by ``$\K_\varphi \simeq k_{\Delta_N \times
  [0,+\infty[}$ near $+\infty$''. Since we know that $c_-(\varphi)\leq 0 \leq c_+(\varphi)$ by \cite{Viterbo-STAGGF}
we obtain the following statement:

\begin{prop}\label{prop:gamma-et-gammag2}
  Let $\varphi_1$, $\varphi_2$ be compactly supported in a \CorrectionRed{Darboux chart} Hamiltonian maps of $T^*N$.  Then $\gamma_\tau(\K_{\varphi_1},
  \K_{\varphi_2}) = \gamma(\varphi_1, \varphi_2)$.
\end{prop}

We end this section by the useful result that $\convstar$ is Lipschitz for the $\gamma_\tau$-distance.

\begin{lem}\label{lem:fonctorialite-tau}
  Let $M$, $N$, $P$ be manifolds and $\K_1 \in D_{\tau\geq0}(M\times N\times\Real)$, $\K_2 \in D_{\tau\geq0}(N\times
  P\times\Real)$.  Then, for any $a\in\Real$, we have natural isomorphisms $T_{a*}(\K_1 \convstar \K_2) \simeq
  T_{a*}(\K_1) \convstar \K_2 \simeq \K_1 \convstar T_{a*}(\K_2)$ and, through these isomorphisms, we have the equalities,
  for any $a\leq b$, $$\tau_{a,b}(\K_1 \convstar \K_2) = \tau_{a,b}(\K_1) \convstar \K_2 = \K_1 \convstar \tau_{a,b}(\K_2)$$
\end{lem}
\begin{rem} 
The  notation $ \tau_{a,b}(\K_1) \convstar \K_2$ means that we apply the functor $\convstar \K_2$ to the morphism
$\tau_{a,b}(\K_1)$. This means that $( \tau_{a,b}(\K_1) \convstar \K_2)$ is a morphism $$\tau_{a,b}(\K_1) \convstar \K_2: {T_a}_*(\K_1)\convstar \K_2 \longrightarrow {T_b}_*(\K_1)\convstar \K_2 $$
\end{rem} 
\begin{proof}
  In the proof the manifolds $M, N, P$ play an auxiliary role and we can consider that our sheaves live on $\Real$ to
  simplify the notations.  Hence $\K_1\convstar \K_2 = Rs_!(\K_1\boxtimes \K_2)$ where $s\colon \Real^2 \to \Real$ is the
  sum.  The isomorphisms then follow from the equalities $T_a \circ s = s \circ (T_a \times \id_\Real) = s \circ
  (\id_\Real \times T_a)$.  However to prove the equalities of morphisms we need to be more precise.

  Let $(t_1,t_2,\tau_1,\tau_2)$ be the coordinates on $T^*\Real^2$. Then the singular support of $\K_1\boxtimes \K_2$ is
  contained in $\{\tau_1\geq 0\} \times \{\tau_2\geq 0\}$ and, setting $g_u = u \tau_1 + (1-u)\tau_2$, $\K_1\boxtimes
  \K_2$ belongs to $D_{g_u\geq 0}(\Real^2)$ for any $u\in [0,1]$.  Let us write $\tau_{a,b}^u$ for the morphism
  $\tau_{a,b}$ induced by $g_u$.  Then $\tau_{a,b}^1$ gives $\tau_{a,b}(\K_1) \convstar \K_2$ by applying $Rs_!$ and
  $\tau_{a,b}^0$ gives $\K_1 \convstar \tau_{a,b}(\K_2)$. It remains to see that $\tau_{a,b}^u$ gives the same morphism
  for all $u$.  We define $\K = \K_1 \boxtimes \K_2 \boxtimes k_{[0,1]} \in D(\Real^2\times [0,1])$ and $g = u \tau_1
  + (1- u)\tau_2$.  Then $K \in D_{g\geq 0}(\Real^2\times[0,1])$ (here we are cheating a little bit because the
  singular support is not defined on a manifold with boundary -- however we can replace $[0,1]$ by
  $\mathopen]-\varepsilon,1+\varepsilon[$ and $u$ by a smooth function with values in $[0,1]$ which is $0$ on
  $\mathopen]-\varepsilon,0]$ and $1$ on $[1,1+\varepsilon[$).  We thus obtain a morphism $\tau_{a,b}^g \colon
  T'_{a*}(\K) \to T'_{b*}(\K)$ where $T'_c(t_1,t_2,u) = (t_1+uc, t_2+(1-u)c,u)$.  Let $s'\colon \Real^2\times[0,1]
  \to \Real\times [0,1]$ be the sum of the first two variables.  Then $s' \circ T'_c = T_c \times \id_{[0,1]}$ and
  $\tau_{a,b}^g$ induces $\tau'_{a,b} \colon (T_{a*}(\K_1 \convstar \K_2)) \boxtimes k_{[0,1]} \to (T_{b*}(\K_1 \convstar
  \K_2)) \boxtimes k_{[0,1]}$ which restricts to $\tau_{a,b}^s$ along $\Real\times\{s\}$.  Now restricting along
  $\Real\times\{s\}$ always induces the same isomorphism $\Hom(\F\boxtimes k_{[0,1]}, \G\boxtimes k_{[0,1]}) \isoto
  \Hom(\F,\G)$, for all $\F$, $\G$, namely, the inverse of the isomorphism induced by the pull-back along the
  projection to $[0,1]$. This concludes the proof. 
\end{proof}

We deduce an analogue of~\cite{Petit-Schapira}, theorem 2.4.1, in our situation.
\begin{lem}\label{lem:distance-composition}
  Let $\K_1, \K_2 \in D_{\tau\geq0}(M\times N\times\Real)$ and $\F \in D_{\tau\geq0}(M\times\Real)$. Then $\K_j\convstar
  \F \in D_{\tau\geq0}(N\times\Real)$ and we have
$$ \gamma_\tau (\K_1^\convstar( \F), \K_2^\convstar(\F)) \leq \gamma_\tau(\K_1,\K_2)\dispdot $$ 
\end{lem}
\begin{proof}
  We have $\gamma_\tau(\K_1,\K_2) < c$, for a given $c$, if and only if there exist $a,b\geq 0$ with $a+b=c$ and
  morphisms $u\colon \K_1 \to T_{a*}(\K_2)$, $v\colon \K_2 \to T_{b*}(\K_1)$ satisfying the conditions in
  Definition~\ref{def:dist-gammag}.  The morphisms induce $u \circ \id_\F$, $v \circ \id_\F$ which also satisfies the
  conditions of the definition by Lemma~\ref{lem:fonctorialite-tau}.  
\end{proof}

\subsection{\texorpdfstring{$\gamma_g$-limits}{gamma-{g}-limits} and colimits}\label{subsection-6.3}
In this paragraph we shall prove that any sequence $(\F_n)_{n\in\N}$ which $\gamma_g$-converges to $\F$ can be turned
into an inductive system whose homotopy colimit is $\F$.  We recall the notion of homotopy colimit in $D(N)$
(see~\cite[Def.~2.1]{Bokstedt-Neeman})
\begin{defn} Let $(\F_n, f_n)_{n\in \N}$, with $f_n\colon \F_n \to \F_{n+1}$, be an
inductive system in $D(N)$; then $\hocolim_n \F_n$ is the cone of the morphism $$(\id -f) \colon \bigoplus_{n\geq 0}
\F_n \to \bigoplus_{n\geq 0} \F_n$$ where $f = \bigoplus_{n\geq 0} f_n$.  
\end{defn} 
We thus have a distinguished
triangle
\begin{equation}\label{eq:def_hocolim}
\bigoplus_{n\geq 0} \F_n \to[\id-f] \bigoplus_{n\geq 0} \F_n
\to[p] \hocolim_n \F_n \to[+1] \dispdot
\end{equation}
However, like any cone, $\hocolim_n \F_n$ is only defined up to (non unique) isomorphism and the morphism $p$ itself
is not uniquely defined.  For a given triangle~\eqref{eq:def_hocolim} and for any $n_0\in \N$, we have a morphism
$i_{n_0} \colon \F_{n_0} \to \hocolim_n \F_n$ defined as $i_{n_0} = p \circ j_{n_0}$ where $j_{n_0}$ is the obvious
morphism to the sum.  For another choice of triangle, with $p'$ instead of $p$, there exists an automorphism $\psi$
of $\hocolim_n \F_n$ such that $p' = \psi\circ p$. Then $i'_{n_0} = p' \circ j_{n_0}$ is ``conjugate'' to $i_{n_0}$
in the sense $i'_{n_0} = \psi \circ i_{n_0}$.

We recall the ``nine diagram'' from \cite[Prop.~1.1.11]{BBD} (see~\cite[Ex.~10.6]{K-S06} where it is stated precisely 
which arrows are given and which ones are obtained),  a variant of the octaedron axiom. 
We assume to be given a commutative square as on the top left corner of the diagram below and four distinguished triangles (two vertical and two horizontal) given by the unmarked solid arrows. The arrows marked  $[1]$ are just translations of already defined solid arrows. \newline
The proposition claims that the dotted arrows can be chosen to make all squares commutative, except for the one with $Z^{2}$ at its top left corner, which is anticommutative, and the new line and the new column are distinguished triangles.
$$
\xymatrix@R=7mm{
  X^0 \ar[r] \ar[d] &  X^1 \ar[r] \ar[d] &  X^2 \ar[r] \ar@{.>}[d] &X^{0}[1] \ar[d]_{[1]}  \\
  Y^0 \ar[r] \ar[d] &  Y^1 \ar[r] \ar[d] &  Y^2 \ar[r] \ar@{.>}[d]  &Y^{0}[1] \ar[d]_{[1]}  \\
  Z^0 \ar@{.>}[r] \ar[d] &  Z^1  \ar[d] \ar@{.>}[r] & Z^{2}  \ar@{.>}[r] \ar@{.>}[d] \ar@{}[rd]|-{\ominus}&Z^{0}[1]\ar[d]_{-[1]} \\
X^{0}[1] \ar[r]^{[1]}&X^{1}[1] \ar[r]^{[1]}&X^{2}[1] \ar[r]^{-[1]} &X^{0}[2]
 }
$$

\begin{rem} \label{Rem-nine-diagram}
Using the nine diagram built on the square

$$\xymatrix@R=5mm@C=10mm{
\bigoplus_{n\geq n_0} \F_n  \ar[r]^{\id-f} \ar[d] & \bigoplus_{n\geq n_0} \F_n \ar[d] \\
\bigoplus_{n\geq 0} \F_n  \ar[r]^{\id-f}  & \bigoplus_{n\geq 0} \F_n 
}$$
we can check that $\hocolim_{n\geq 0} \F_n \isoto \hocolim_{n\geq n_0} \F_n$ is indeed an isomorphism.
With the above notations, we take $X^0=X^1=\bigoplus_{n\geq n_0} \F_n , Y^0=Y^1= \bigoplus_{n\geq 0} \F_n, Z^0=Z^1=\bigoplus_{0\leq n\leq n_0-1} \F_n$. The map  from $Z^0$ to $Z^1$ is also equal to $\id-f$, but as is an upper triangular finite-dimensional matrix with $\Id$ on the diagonal, so it is invertible. This implies $Z^2=0$.  As a result the  vertical arrow from $X^2$ to $Y^2$ is an isomorphism (in $D(N)$). 
\end{rem} 

A more general statement of the following Lemma can be found in
Lemma~13.33.4 of \cite[\href{https://stacks.math.columbia.edu/tag/0A5K}{Section 0A5K}]{stacks-project}.
 It tells us, as we might expect, that the limit of a subsequence coincides with the limit of the original sequence. 

\begin{lem}\label{lem:sous-systeme-inductif}
  Let $(\F_n, f_n)_{n\in \N}$ be an inductive system in $D(N)$ and $\sigma \colon \N \to \N$ an increasing map.  Then
  the compositions of the $f_n$'s define an inductive system $(\F_{\sigma(n)}, f'_n)$ and there exists an isomorphism
  $j\colon \hocolim_n \F_{\sigma(n)} \isoto \hocolim_n \F_n$ such that $j \circ i'_{\sigma(n)} = i_{\sigma(n)}$,
  where $i'_{\sigma(n)}$ is the morphism $i_\bullet$ for $(\F_{\sigma(n)})$.
\end{lem}

A family of morphisms $u_n \colon \F_n \to \G$ such that $u_{n+1} \circ f_n = u_n$ induces a morphism $u\colon
\hocolim_n \F_n \to \G$ such that $u \circ i_n = u_n$ for all $n$.  Indeed we have $\bigoplus_n u_n \circ (\id-f) =
0$, hence $\bigoplus_n u_n$ factorizes as $\bigoplus_n u_n = u \circ p$. Again this morphism $u$ is not uniquely
defined.

\begin{lem}\label{lem:gammalim=hocolim0}
  Let $(\F_n, f_n)_{n\in \N}$, with $f_n\colon \F_n \to \F_{n+1}$, be an inductive system in $D(N)$.  We let
  $\C(f_n)$ be the cone of $f_n$ and we assume that the sum $\sum_n \gamma_g(\C(f_n),0)$ is finite.  Then the sequence
  $(\F_n)_{n\in \N}$ $\gamma_g$-converges to $\hocolim_k \F_k$.  More precisely, if we let $i_n \colon \F_n \to
  \hocolim_k \F_k$ be the morphism defined after~\eqref{eq:def_hocolim}, we have $\gamma_g(\C(i_n),0) \leq 2
  \sum_{k=n}^\infty \gamma_g(\C(f_k),0)$.
\end{lem}
\begin{proof}
  Let $n_0$ be given.  We set for short $\calS = \bigoplus_{n\geq n_0} \F_n$, $\C = \hocolim_n \F_n$.  We  checked
  that $\C$ is the cone of $u = \id - \bigoplus_{n\geq n_0} f_n \colon \calS \to \calS$.  We also consider the
  similar sum but associated to the constant sequence $\G_n=\F_{n_0}$ for all $n\geq n_0$  and  $\mathcal T_{} = \bigoplus_{n\geq n_0} \G_{n}$. We define $w
  \colon \mathcal T_{} \to \mathcal T_{}$  as $\id-s$ where $s_n: \G_n \longrightarrow \G_{n+1}$ is induced by $\Id_{\F_{n_0}}$.  The cone of $w$ is $\F_{n_0}$.  We let
  $f_n^m \colon \F_n \to \F_m$ be the composition of the $f_k$'s and we define $F = \bigoplus_{n\geq n_0} f_{n_0}^n
  \colon \T\to \calS$.  Then $F \circ w = u \circ F$ and the nine diagram recalled above gives the
  commutative (except for one anticommuting square) diagram of distinguished triangles
  $$
  \xymatrix{
    \mathcal T \ar[r]^{w} \ar[d]_F &  \mathcal T\ar[r]^{q} \ar[d]_F & \F_{n_0} \ar[d]^{j_{n_0}} \ar[r]^{+1} & \\
    \calS \ar[r]^u \ar[d] &  \calS \ar[r]^p \ar[d] & \C \ar[d] \ar[r]^{+1} & \\
    \calS' \ar[r] \ar[d]^{+1} &  \calS' \ar[r] \ar[d]^{+1} & \C(j_{n_0}) \ar[r]^{+1} \ar[d]^{+1}  &\\
&&&
}
$$
where $\calS' = \bigoplus_{n\geq n_0} \C(f_{n_0}^n)$,  $j_{n_0}$ is conjugate to $i_{n_0}$ (unfortunately we cannot say they are equal because of the ambiguity in the definition of the cone) and $\mathcal C(j_{n_0})$ is the cone of $j_{n_0}$.  

We set
$r  \colon \F_{n_0} \to \mathcal T$ be the inclusion of the first factor.  The morphisms $r$ and $w$ give
a splitting $\mathcal T \simeq \F_{n_0} \oplus \mathcal T$ and it follows that $q \circ r$ is an
isomorphism.  Since $i_{n_0}$ is ``conjugate'' to $p \circ F \circ r$ (see
after~\eqref{eq:def_hocolim}), it is also conjugate to $j_{n_0}$ and
$\C(j_{n_0}) \simeq \C(i_{n_0})$.

Now we claim that $\gamma_g(\C(f_n^m),0) \leq \sum_{k=n}^{m-1} \gamma_g(\C(f_k),0)$.
We prove this by induction on $m-n$.  When $m=n+1$ this is clear.  Since
$f_n^{m+1} = f_m \circ f_n^m$, we have, using Lemma \ref{lem:proprietes-gammag} ( \ref{lem:proprietes-gammag-vi}), 
$\gamma_g(\C(f_n^{m+1}),0) \leq \gamma_g(\C(f_n^m),0) + \gamma_g(\C(f_m),0)$.

In particular
$\gamma_g(\C(f_{n_0}^n),0) \leq \varepsilon_{n_0} :=\sum_{k=n_0}^\infty \gamma_g(\C(f_k),0)$, for all
$n\geq n_0$, and then $\gamma_g(\calS',0) \leq \varepsilon_{n_0}$.  By the triangle in the last
row of the above diagram and Lemma \ref{lem:proprietes-gammag} ( \ref{lem:proprietes-gammag-v}) we deduce $\gamma_g(\C(j_{n_0}),0) \leq 2 \varepsilon_{n_0}$, as required.
\end{proof}

\begin{prop}\label{prop:completudefaisceaux}
  Any Cauchy sequence in $D_{g\geq 0}(N)$ with respect to $\gamma_g$ is convergent.
\end{prop}
\begin{rem} 
The reader should be careful : the distance $\gamma_g$ is degenerate, so the limit is by no means unique. 
However we shall see that the restriction of $\gamma_g$ to the set of limits of constructible sheaves (i.e. in $D_{lc}(N)$)  the distance is non-degenerate (see Appendix \ref{Appendix-B}).
\end{rem} 
\begin{proof}
  Let $(\F_n)_{n\in \N}$ be a Cauchy sequence.  Up to taking a subsequence we can
  assume that $\gamma_g(\F_n,\F_m) \leq 2^{-n}$ for any $n\leq m$.  In particular, by
  Lemma~\ref{lem:comparaison-distances}, there exist $f_n \colon \F_n \to \fai{\varphi}^{2^{-n}}(\F_{n+1})$
  such that $\gamma_g(\C(f_n),0) \leq 2^{-n+1}$. We then set $ \varepsilon_n= \sum_{k\geq n} 2^{-k}$ and $\F'_n=\fai{\varphi}^{- \varepsilon_n} (\F_n)$ so $f_n$ induces a map $f'_n: \F'_n \longrightarrow \F'_{n+1}$ and according to   Lemma~\ref{lem:gammalim=hocolim0}, 
  $\F'_n$ $\gamma_g$ converges to $\hocolim \F'_n$. Since $\gamma_g(\F_n, \F'_n) \leq \varepsilon_n$ we have that $(\F_n)_{n\geq 1}$ has the same limit. 
\end{proof}

We just proved that to a Cauchy sequence we can associate a limit that is a homotopy colimit. In the next Proposition
we prove that conversely any $\gamma_g$-limit of a sequence $(\F_n)_{n\in\N}$ is a homotopy colimit of some (slightly
modified) subsequence.  We first notice an easy consequence of Lemma~\ref{lem:gammalim=hocolim0}.

\begin{lem}\label{lem:F=holim-translates}
  Let $\F\in D(N)$ and let $\varepsilon_n>0$ be a sequence decreasing to $0$.  Then we have an inductive system
  $(\fai{\varphi}^{-\varepsilon_n}(\F), \tau_{-\varepsilon_n, -\varepsilon_{n+1}}(\F))$ and the maps
  $\tau_{-\varepsilon_n, 0}(\F)$ induce an isomorphism $\hocolim \fai{\varphi}^{-\varepsilon_n}(\F) \isoto \F$.
\end{lem}
\begin{proof}
  Let $\tau\colon \hocolim \fai{\varphi}^{-\varepsilon_n}(\F) \to \F$ be the morphism induced by the maps
  $\tau_{-\varepsilon_n, 0}(\F)$.  With the notation $i_n$ of Lemma~\ref{lem:gammalim=hocolim0} we have $\tau \circ
  i_n = \tau_{-\varepsilon_n, 0}(\F)$. By the octaedron axiom there exists a distinguished triangle relating the
  cones of these maps and Lemma~\ref{lem:proprietes-gammag} implies that $\gamma_g(0,\C(\tau)) \leq
  \gamma_g(0,\C(i_n)) + \gamma_g(0,\C(\tau_{-\varepsilon_n, 0}(\F)))$.  By Lemma~\ref{lem:comparaison-distances}
  $\gamma_g(0,\C(\tau_{-\varepsilon_n, -\varepsilon_{n+1}}(\F))) \leq 2(\varepsilon_n - \varepsilon_{n+1})$ and we
  can apply Lemma~\ref{lem:gammalim=hocolim0}. Hence $\gamma_g(0,\C(i_n))$ goes to $0$ and $\gamma_g(0,\C(\tau))$ is
  as small as desired. So $\C(\tau) \simeq 0$ and $\tau$ is an isomorphism.
\end{proof}

\begin{prop}\label{prop:gammalimit=colimit}
  Let $(\F_n)_{n\in \N}$ be a sequence in $D(N)$ which $\gamma_g$-converges to $\F$.  Then, up to taking a
  subsequence, there exist a sequence of positive numbers $(\varepsilon_n)_{n\in \N}$ converging to $0$ and morphisms
  $f_n \colon \fai{\varphi}^{-\varepsilon_n}(\F_n) \to \fai{\varphi}^{-\varepsilon_{n+1}}(\F_{n+1})$, $u_n \colon
  \fai{\varphi}^{-\varepsilon_n}(\F_n) \to \F$ such that $u_{n+1} \circ f_n = u_n$, for all $n$, and the morphism
  $\hocolim \fai{\varphi}^{-\varepsilon_n}(\F_n) \to \F$ induced by the $u_n$'s (see after~\eqref{eq:def_hocolim}) is
  an isomorphism.
\end{prop}
\begin{proof}
  We set $\alpha_n = 2^{-n}$.  Up to taking a subsequence we can assume that $\gamma_g(\F_n,\F) \leq \alpha_{n+1}$
  for each $n$. This implies the existence of $0\leq \beta_n \leq \alpha_{n+1}$ and morphisms $v_n \colon \F \to
  \fai{\varphi}^{\beta_n}(\F_n)$, $w_n\colon \fai{\varphi}^{\beta_n}(\F_n) \to
  \fai{\varphi}^{\alpha_{n+1}}(\F)$ such that $w_n\circ v_n = \tau_{0,\alpha_{n+1}}(\F)$.  Translating these
  morphisms by $\fai{\varphi}^\bullet(-)$ and alternating the $\F_n$ and the $\fai{\varphi}^{-\alpha_n}(\F)$ we
  obtain an inductive system
  \begin{multline*}
\cdots\to    \fai{\varphi}^{-\alpha_n}(\F) \to[\fai{\varphi}^{-\alpha_n}(v_n)] \fai{\varphi}^{-\varepsilon_n}(\F_n) 
    \to[\fai{\varphi}^{-\alpha_n}(w_n)] \fai{\varphi}^{-\alpha_{n+1}}(\F) \\
    \to[\fai{\varphi}^{-\alpha_{n+1}}(v_{n+1})] \fai{\varphi}^{-\varepsilon_{n+1}}(\F_{n+1})
     \to[\fai{\varphi}^{-\alpha_{n+1}}(w_{n+1})] \fai{\varphi}^{-\alpha_{n+2}}(\F) \to\cdots
  \end{multline*}
  where $\varepsilon_n = \alpha_n-\beta_n$.  We define $f_n = \fai{\varphi}^{-\alpha_{n+1}}(v_{n+1}) \circ
  \fai{\varphi}^{-\alpha_n}(w_n)$ and $u_n = \tau_{-\alpha_{n+1},0}(\F) \circ \fai{\varphi}^{-\alpha_n}(w_n)$.
\CorrectionBlue{  By Lemma~\ref{lem:sous-systeme-inductif} we have $\hocolim \fai{\varphi}^{-\varepsilon_n}(\F_n) \simeq
  \hocolim \fai{\varphi}^{-\alpha_n}(\F)$ and by Lemma~\ref{lem:F=holim-translates} we have $\hocolim
  \fai{\varphi}^{-\alpha_n}(\F) \isoto \F$, which proves the result. }
\end{proof}

We now use the previous results to bound the microsupport of a $\gamma_g$-limit.  For a sequence
$(X_{j})_{j\geq 0}$ of subsets of a topological space $Z$, its topological upper and lower limits are defined
as
\begin{align*}
  \limsup_{j} X_{j}&=\bigcap_n \overline{\bigcup_{j\geq n}X_{j}} \\
  &=\left\{x\in Z \mid \exists (x_j)_{j\geq 1}, \;x_j\in X_j\;  \text{for infinitely many}\; j, \; \lim_j x_j=x\right \} \\
    \liminf_j X_j &= \left\{x\in Z \mid \exists (x_j)_{j\geq 1}, \; x_j\in X_j,\; \lim_j x_j=x\right \}
\end{align*}
Obviously $\liminf_j X_j\subset \limsup_j X_j$. 
We have
\begin{prop} \label{Prop-g-continuity}
  Let $N$ be a manifold and let $g \colon T^*N \setminus 0_N \to {\mathbb R} $ be a $1$-homogeneous autonomous
  Hamiltonian function.  If $(\F_j)_{j\geq 1}$ is a sequence such that $\gamma_g-\lim \F_j=\F$ then we have
  $$
  SS(\F) \subset \liminf_{j} SS(\F_j)
  $$
\end{prop}
\begin{proof}
\CorrectionBlue{  By Proposition~\ref{prop:gammalimit=colimit}, up to taking a subsequence, we can find some sequence
  $(\varepsilon_j)$ converging to $0$, such that $\hocolim \F'_j \isoto \F$, where $\F'_j =
  \fai{\varphi_g}^{\varepsilon_j}(\F_j)$.  We have $SS(\F'_j) = \varphi_g^{\varepsilon_j}(SS(\F_j))$ and
  $\varphi_g^{\varepsilon_j}$ tends to $\id$ in $C^0$ norm over compact subsets. Hence $\liminf_{j} SS(\F_j) =
  \liminf_{j} SS(\F'_j)$. }  
  The homotopy colimit is defined by a distinguished triangle where the other two terms are $\bigoplus_{n\geq
    n_0}\F'_n$ (we can indeed restrict to $\bigoplus_{n\geq n_0}$ by Remark \ref{Rem-nine-diagram}). By the
  triangle inequality for singular supports and Proposition \ref{Prop-2.3}, this implies that $SS(\F) \subset
  \limsup_{j} SS(\F'_j)$. Now let $p=(x,\xi)\not\in \liminf_j SS(\F'_j)$, we can find a subsequence
  $\sigma(j)$ such that $p\notin \limsup_j SS(\F'_{\sigma(j)})$; but $\F'_{\sigma(j)}$ has limit $\F$, so
  $p\notin SS(\F)$. As a result $SS(\F) \subset \liminf_{j} SS(\F'_j)$.
\end{proof}

\section{Proof of Theorem \ref{Thm-sheaf-gamma}}\label{sec:proofmainthm}
Let $\F \in D(N)$.  We want to prove that around a point $z_0\in SS(\F)$ there is no sequence
$(\varphi_j)_{j\geq 1}$ in $\DHam_c (B(z_{0}, \varepsilon))$ such that $\gamma-\lim_j\varphi_j=\Id$ and
$\varphi_j(SS(\F)) \cap B(z_{0},\eta)=\emptyset$ for some $\varepsilon>\eta>0$.  We argue by contradiction and
assume the existence of such a sequence $(\varphi_j)_{j\geq 1}$ and open ball $B(z_0,\eta)$.  Let
$\K_{\varphi_j}$ be the sheaf in $D^b(N\times N\times {\mathbb R} )$ inducing $\varphi_j$ as recalled in
\S\ref{sec:qhi}.  We denote by $T^*_{\tau>0}(N\times\Real)$ the subset of $T^*(N\times\Real)$ given by
$\tau>0$ and we set $\rho\colon T^*_{\tau>0}(N\times\Real) \to T^*N$, $(x,t,\xi,\tau) \mapsto (x,\xi/\tau)$.
Then $\rho(SS(\K^\convstar_{\varphi_j}( \F)) \cap T^*_{\tau>0}(N\times\Real)) = \varphi_j (SS(\F))$ does not
intersect $B(z_0, \eta)$. We want to prove that this implies that $SS(\F) \cap B(z_0,\eta)=\emptyset$.

We will use the following proposition.
\begin{prop} \label{Prop-quantization-continuity}
Let $(\varphi_j)_{j\geq 1}$ be a sequence in $\DHam_c (T^*N)$ such that $\gamma-\lim \varphi_j= \Id$ and $\K^\convstar_{\varphi_j}$ its quantization in $D^{b}(N\times N \times {\mathbb R} )$.  Then for all $\F \in D^b(N)$ we have 
$$\gamma_\tau-\lim \K^\convstar_{\varphi_j}(\F)= \K^\convstar_{\Id} ( \F)=\F \boxtimes k_{[0,+\infty[} \in D^b(N\times {\mathbb R} )$$
\end{prop} 
\begin{proof}
  This follows from Lemma~\ref{lem:distance-composition} and Proposition~\ref{prop:gamma-et-gammag2}.
\end{proof}

\begin{rem} 
Note that $\varphi_j(SS(\F))$ is not a homogeneous manifold, and  $SS(\K^\convstar_{\varphi_j} ( \F)) \subset T^*(N\times {\mathbb R} )$ represents the homogenization of  $\varphi_j(SS(\F))$. 
\end{rem} 

\CorrectionBlue{ We now claim that $SS(\F\boxtimes k_{[0,+\infty[}) \cap \{\tau>0\} = SS(\F) \times (\{0\}\times
  \{\tau>0\} )$.  Indeed we have the general bound $SS(\F\boxtimes k_{[0,+\infty[}) \subset SS(\F) \times
  SS(k_{[0,+\infty[})$ and a similar bound for $\F\boxtimes k_{]-\infty,0]}$.  Now we have the distinguished
  triangle $\F\boxtimes k_\Real \to \F\boxtimes (k_{]-\infty,0]} \oplus k_{[0,+\infty[}) \to \F\boxtimes
  k_{\{0\}} \to[+1]$ and the equalities $SS(\F\boxtimes k_{\{0\}}) = SS(\F) \times (\{0\} \times \Real)$,
  $SS(\F\boxtimes k_\Real) = SS(\F) \times (\Real\times \{0\})$, by~\cite[Prop. 5.4.4, 5.4.5]{K-S}.  The
  triangular inequality for the microsupport then gives the claim.  }

We return to the sequence $(\varphi_j)_{j\geq 1}$ and the ball $B(z_0,\eta)$.  Proposition
\ref{Prop-quantization-continuity} and Proposition \ref{Prop-g-continuity} (applied with $N\times\Real$
and $g=\tau$) give
$$
SS(\F) \times (\{0\}\times \{\tau>0\} ) \subset
SS(\K^\convstar_\id(\F)) \subset \liminf_jSS(\K^\convstar_{\varphi_j} (\F))\dispdot
$$
Applying the map $\rho$ we deduce $SS(\F) \subset \liminf_j \varphi_j(SS(\F))$.  In particular $SS(\F) \cap
B(z_0,\eta) =\emptyset$ which contradicts $z_0 \in SS(\F)$.


\section{ \texorpdfstring{$\gamma$}{Gamma}-coisotropic vs cone-coisotropic}\label{Section-8}

\subsection{A \texorpdfstring{$\gamma$}{Gamma}-coisotropic set is cone-coisotropic}
Our goal in this section is to prove Proposition \ref{Prop-gamma-to-cone}. 

Following Bouligand (see \cite{Bouligand}), one may define for a point $x$ of a  subset  $V$ in a smooth manifold,  two cones :
\begin{defn} Let $V$ be  a closed subset of the $C^{1}$-manifold $M$ and $x\in V$. We choose a $C^{1}$ chart of $M$ at $x$ and define 
\begin{enumerate} 
\item The {\bf paratingent cone} of $V$ at $x$ is 
  \begin{multline*}
    C^+(x,V)= \left \{ \lim_n c_n(x_n-y_n) \mid x_n, y_n\in V, c_n \in {\mathbb R},\; \lim_nx_n=\lim_n y_n=x, \right. \\
     \left. \lim_n c_n=+\infty \right\}
\end{multline*}
\item The {\bf contingent cone} of $V$ at $x$ is 
$$C^-(x,V)= \left \{ \lim_n c_n(x_n-x) \mid x_n \in V, c_n \in {\mathbb R}, \; \lim_nx_n=x,  \lim_n c_n=+\infty \right \}$$
\end{enumerate} 
\end{defn} 
Clearly $C^-(x,V) \subset C^+(x,V)\subset T_{x}V$. Note that $C^+(x,V)$ is invariant by $v \mapsto -v$, while it is not necessarily the case for $C^-(x,V)$. 

\CorrectionBlue{We shall say that a cone $C$ is closed (resp. open) if $C\cap S(x,1)$ is closed (resp. open). Note that a closed cone is closed in the usual sense, while an open cone is not open (but becomes open if we remove $x$). 
Then the contingent cone is closed. }
It is also easy to see that both cones are independent from the choice of the $C^{1}$ chart. We then have the following definition, which appears under the name ``involutivity'' in the book of  Kashiwara and Schapira 
\begin{defn} [Cone-coisotropic, see \cite{K-S}, definition 6.5.1 p. 271]
\label{Def-cone-coisotropic}
We shall say that $V$ is {\bf cone-coisotropic} if whenever a hyperplane $H$ is such that $C^+(x,V) \subset H$ then the symplectic orthogonal of $H$, $H^\omega$ is contained in $ C^-(x,V)$. 
\end{defn} 

\begin{figure}[ht]
 \begin{overpic}[width=6cm]{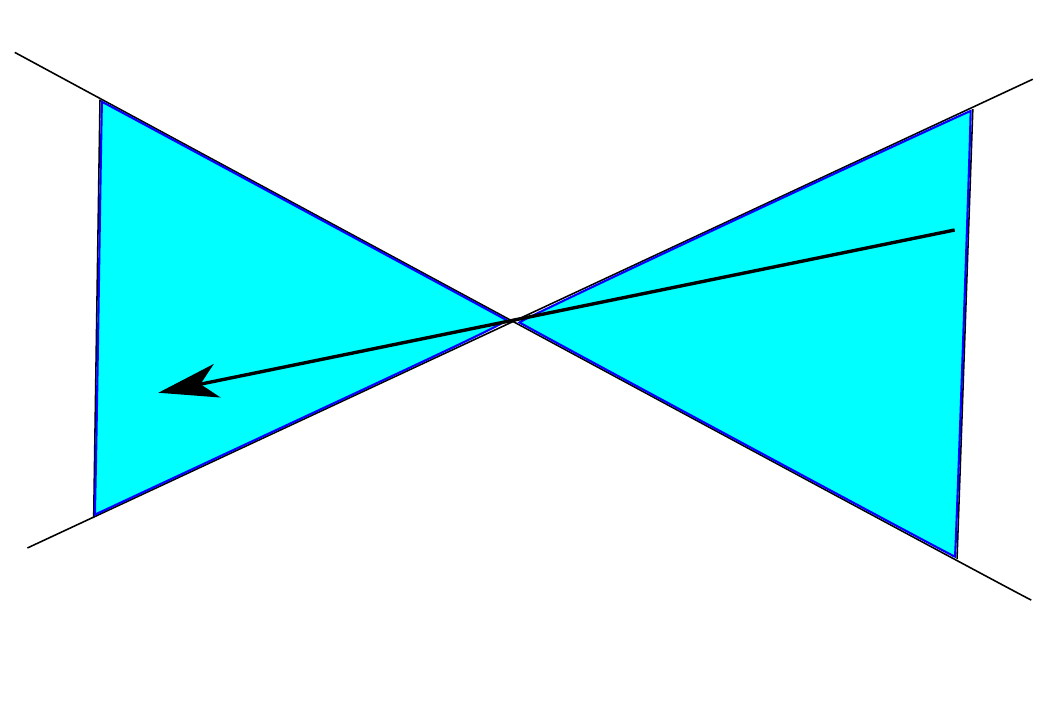}
 \put (75,35) {$\myBlue C$} 
  \put (12,45) {$\myBlue C$} 
   \put (10,35) {$ {H^\omega}$} 
 \end{overpic}
\begin{overpic}[width=6cm]{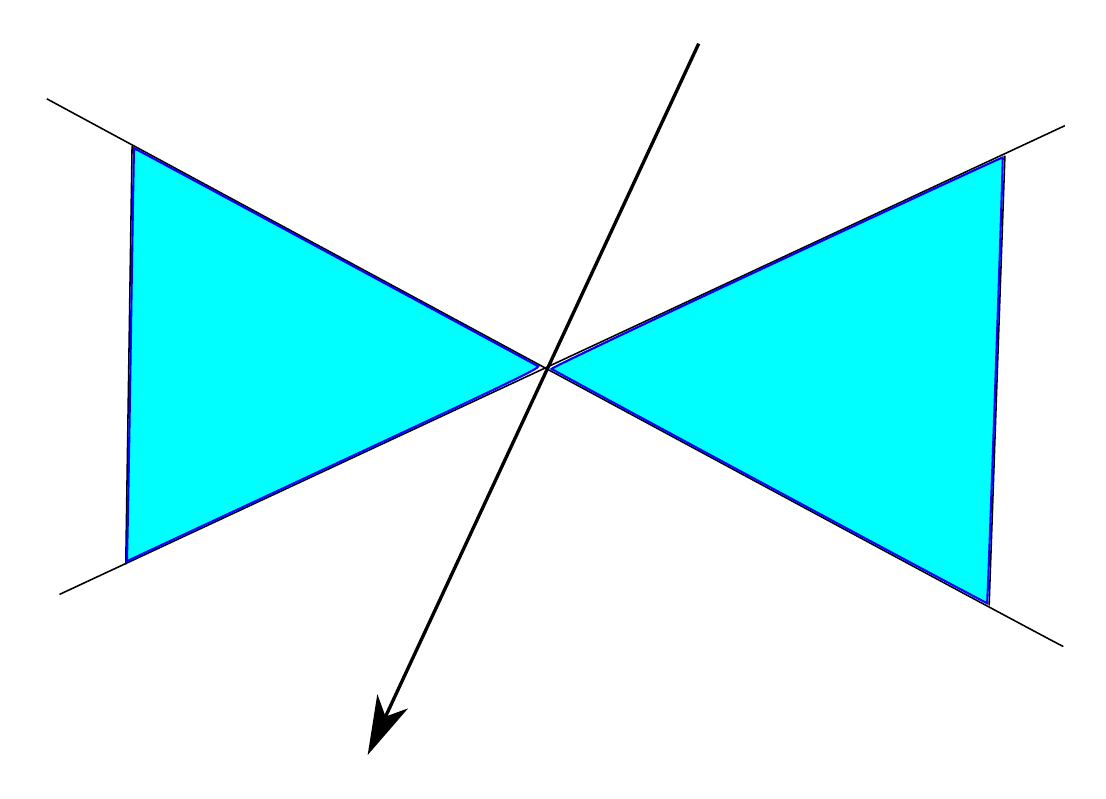}
 \put (75,35) {$\myBlue C$} 
  \put (12,35) {$\myBlue C$} 
   \put (27,10) {$ {H^\omega}$} 
 \end{overpic}
\caption{Coisotropic cone on the left, non-coisotropic on the right}
\label{fig-7}
\end{figure}
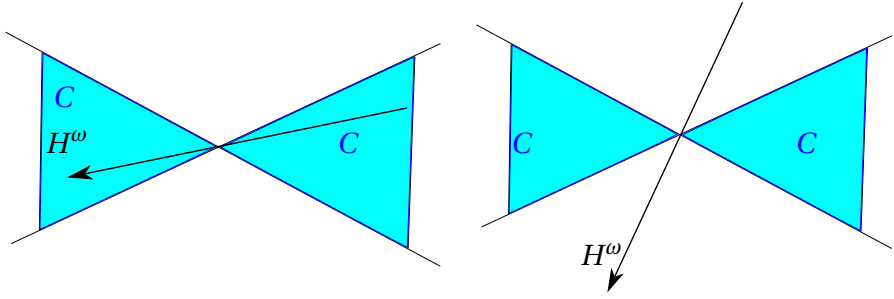

 \begin{proof} [Proof of Proposition \ref{Prop-gamma-to-cone}]
We want to prove that if $V$ is not cone-coisotropic at $z\in V$  then it is not $\gamma$-coisotropic at $z$. Using a chart, we may assume $V \subset {\mathbb R}^{d}$ and $z=0$. 
For $H$ a hyperplane we set  
$n_H$ to be  the normal to $H$ (for the usual scalar product). 
Let us assume that  $C^+(0,V)\subset H$ and $H^\omega \notin C^-(0,V)$. 
We must prove that for all $ \varepsilon >0$ there exists a ball $B(0, \eta)$ such that for all $\delta>0$ we can find  $\varphi \in \DHam_{c}(B(x, \varepsilon )$ such that  
$\varphi(V)\cap B(0,\eta)=\emptyset$ and $\gamma (\varphi)<\delta$.
We first prove
\begin{lem} \label{Lemma-8.3}
  Let $V$ be a closed set in $ {\mathbb R}^{d}$ containing $0$ and $H$ a hyperplane such that $C^+(0,V)\subset H$. Then there exists a continuous function $f_{H,V}$ defined in a neighbourhood of $0$ in $H$ such that $f_{H,V}(0)=0$ and $d_Gf_{H,V}(0)=0$, where $d_G$ is the G\^ateaux derivative, and for any  open cone
  $C_\rho ^-(0,V) := \{ \lambda x;\; \lambda \geq 0,\, x\in S^{d-1},\, d(x,C^-(0,V)) < \rho\}$,
  we have  for $\delta$ small enough
$$V\cap B(0, \delta ) \subset \left\{ x+f_{H,V}(x)n_H \mid x \in C_{ \rho}^-(0,V)   \right \}$$
\end{lem} 
\begin{proof} 
 For $x\in H$ let $V_x= V \cap (x+ {\mathbb R}n_H)$. We claim that for $x$ close enough to $0$,  $V_x$ contains at most one point. Indeed, if this was not the case, we would have sequences $(y_n)_{n\geq 1}, (w_n)_{n\geq 1}$ converging to $0$ such that $y_n-w_n=t_nn_H$ with $t_n\neq 0$. But then by definition of $C^+(0,V)$, we would have $n_H \in C^+(0,V)$ a contradiction. 
 
As a result, on the closed set $\Gamma \subset H$ defined as the set of $x$ such that $V_x$ is non empty,  we have a function $f$ such that 
$$x\in \Gamma \Longleftrightarrow x+f(x)n_{H} \in V \dispdot $$
Let us prove $f$ is continuous on $\Gamma$, at least in a neighbourhood of $0$. If this was not the case, there would be a sequence $(z^n)_{n\geq 1}$ converging to $0$ and for each $n$ two sequences $(x_k^n)_{k\geq 1},\, (y_k^n)_{k\geq 1}$ such that $\lim_k x_k^n=\lim y_k^n=z^n$ and $\lim_k f(x_k^n) \neq \lim_k f(y_k^n)$. 
 Setting $u_k^n=x_k^n+f(x_k^n)n_{H}, v_k^n =y_k^n+f(y_k^n )n_{H}$ we have 
$$ u_k^n-v_k^n= (x_k^n-y_k^n)+( f(x_k^n)-f(y_k^n))n_{H}$$ which can be normalized so that, for any $n$, it has a limit $(0,1)$ when $k\to +\infty$. As a result for $k=k(n)$ large enough, setting $u_n=u_{k(n)}^n, v_n=v_{k(n)}^n$,  there is a $\tau_n$ so that $$\Vert \tau_n ( u_n-v_n) - (0,1) \Vert < 1/3$$ We can then find a subsequence of $\tau_{n}(u_{n}-v_{n})$  converging  to a vector $w$ which does not belong to $H$. Since $w \in C^+(0,V)$ we have a contradiction. 
From now on we replace $\Gamma$ by $\Gamma \cap B(0,r)$ with $r$ small enough so that $f$ is now continuous on $\Gamma$.

\CorrectionBlue{ To show that $\Gamma \subset C_{\rho}^{-}(0,V)$ we first show 
 that $\displaystyle\lim_{ \substack{\vert x \vert \to 0\\  x\in \Gamma}} \frac{f(x)}{ \vert x \vert }=0$.
Indeed, if we had a sequence $(x_{k})_{k\geq 1}$ in $\Gamma$ converging to $0$ and such that $\displaystyle \lim_{k} \frac{f(x_{k})}{ \vert x_{k} \vert }\neq 0$, the sequence of vectors $ \frac{1}{ \vert x_{k}\vert }(x_{k}+ f(x_{k})n_{H})$ would have a subsequence with a limit which does not belong to $H$, so $C^{+}(0,V)\not\subset H$, a contradiction. 
To prove that $\Gamma \subset C_{\rho}^{-}(0,V)$ we argue by contradiction, and assume we have a sequence $x_{n}\in \Gamma$ such that $\lim_{n}x_{n}=0$ and 
$w=\lim_{n} \frac{x_{n}}{ \vert x_{n} \vert } \notin C^{-}(0,V)$. Then there exists a sequence $x_{n}+f(x_{n})n_{H}$ converging to $0$ and belonging to $V$. 
As a result  $\displaystyle \lim_{n} \frac{x_{n}+f(x_{n})n_{H}}{ \vert  x_{n} \vert }= \lim_{n} \frac{x_{n}}{ \vert  x_{n} \vert }=W \notin C^{-}(0,V)$ a contradiction.  
}

We now will extend $f$ to $f_{H,V}$ such that
\begin{enumerate} 
\item \label{8.1-1} $f_{H,V}(x)=f(x)$ for $x\in \Gamma$
\item \label{8.1-2} $f_{H,V}$ is continuous
\item \label{8.1-3} $d_G(f_{H,V}(0))=0$
\end{enumerate}  
 
   For this we can assume for simplicity that $r=1$ and write $x=\rho\cdot \theta$ with $0\leq \rho \leq 1, \theta \in S^{d-1}$. We write $g(s,\theta)=f(e^s\cdot \theta)$ with $s \in \mathopen]-\infty, 0]$ and identify $\Gamma$ with its image in $\mathopen]-\infty, 0] \times S^{d-1}$. Now let $g_{n}(s,\theta)$ be a continuous function defined on $ C_n= \mathopen]-n-1, -n+1\mathclose[ \times S^{d-1}$ for $n \in \mathbb N$,
   coinciding with $g$ on $\Gamma \cap C_n$ and having the same bound, i.e.
$$\sup_{} \left \{g_n(s,\theta) \mid (s,\theta)\in  C_n \right \}=\sup_{} \left \{g(s,\theta) \mid (s,\theta)\in  \Gamma \cap C_n \right \}$$
$$\inf_{} \left \{g_n(s,\theta) \mid (s,\theta)\in  C_n \right \}=\inf_{} \left \{g(s,\theta) \mid (s,\theta)\in  \Gamma \cap C_n \right \}$$
Such a function exists by Tietze's extension theorem. 

Now let $\chi_{n}$ be a partition of unity subordinated to the covering of $\mathopen]-\infty, 0] \times S^{d-1}$ by the $C_{n}$. 
We set 
$$G(s,\theta)=\sum_{n\geq 0}\chi_{n}(s)
g_{n}(s,\theta)$$
Note that this is well defined since whenever $g_{n}(s,\theta)$ is not defined, we have $\chi_{n}(s,\theta)=0$, so we set $\chi_{n}(s)g_{n}(s,\theta)=0$ outside of $C_{n}$.  And since whenever $g_{n}(s,\theta)$ is defined and $(s,\theta)\in \Gamma$ we have $g_{n}(s,\theta)=g(s,\theta)$, we see that $G$ is an extension of $g$. 
We now set $$f_{H,V}(e^{s}\cdot \theta)=G(s,\theta)$$
Properties (\ref{8.1-1}) and (\ref{8.1-2}) are then obvious. As for the third one, we notice that 
\begin{multline*}  \sup_{} \left \{f_{H,V}(x) \mid  \vert x \vert \leq e^{-n} \right \}=\sup_{} \left\{G(s,\theta) \mid (s,\theta) \mid s\leq -n \right \}  \leq \\
\sup_{k\geq n} \sup_{} \left \{g_{k}(s,\theta) \mid (s,\theta)\in  C_k \right \} 
\leq  \sup_{} \left \{g(s,\theta) \mid (s,\theta)\in  \Gamma, \, s\leq -n \right \}=\\ \sup_{} \left \{f(x) \mid x\in  \Gamma , \, \vert x \vert \leq e^{-n} \right \}
\end{multline*} 
We see that since $\displaystyle \lim_{ \substack{\vert x \vert \to 0 \\  x\in \Gamma}} \frac{f(x)}{ \vert x \vert }=0$  we have 
$\displaystyle \lim_{ x  \to 0} \frac{f_{H,V}(x)}{ \vert x \vert }=0$.

As a result
$$\lim_{t\to 0} \frac{1}{t}(f_{H,V}(th)-f_{H,V}(0))=0$$ and we may conclude that $d_{G}f_{H,V}(z)=0$.
\end{proof}

\begin{lem} \label{Lemma-8.4}
Let $f$ be a continuous function  on the hyperplane $H$ such that $d_Gf(0)=0$ and $C$ a closed cone in $H$ not containing $H^\omega$. Set 
$$\widetilde C= \left\{x+f(x)n_H \mid x \in C   \right \}$$
Then for $\eta$ small enough and any positive $ \delta $,  there is a Hamiltonian map $\varphi$ supported in $B(0,1)$, such that $\gamma (\varphi) \leq \delta $ and 
$$\varphi( \widetilde C)\cap B(0,\eta) = \emptyset $$
\end{lem} 
\begin{proof} \CorrectionBlue{
We shall assume w.l.o.g. that $H=\{p_n=0\}$, so that $\widetilde C=\{(q_n,\bar q, p_n, \bar p) \mid p_n=f(q_n,\bar q, \bar p)$ and $H^\omega$ is given by the direction $\frac{\partial} {\partial q_n}$. 
First let us mention that the case $f=0$ is trivial: just take a Hamiltonian $K_0(p_n)$ such that $K_0(0)=0$, $K_0'(0)=\frac{1}{2}$  and $K_0$ is $C^0$ small, $\eta$ being chosen so that  we have $$ B(0,\eta) \cap ( \frac{1}{2} \frac{\partial }{\partial q_n}+C)=\emptyset$$ 
We may then truncate $K_0$ outside the ball $B(0,1)$ so that we still have 
$$ B(0,\eta) \cap (\varphi_{K_0}^1(C))=\emptyset$$
}
 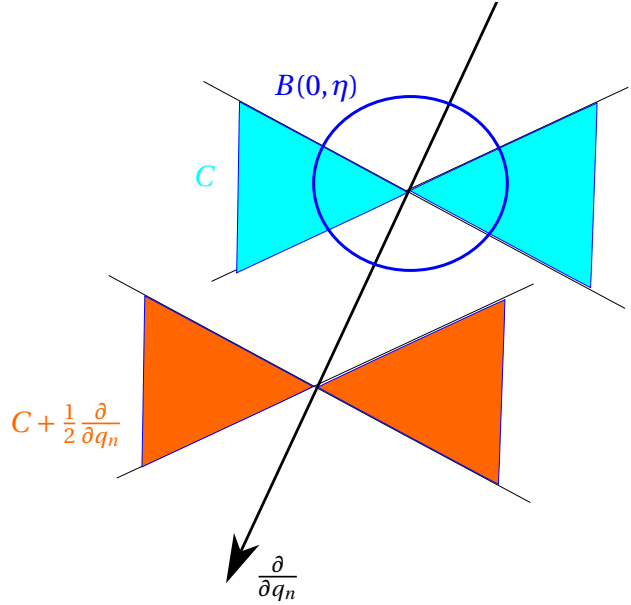
\begin{figure}[ht]
\center\begin{overpic}[width=8cm]{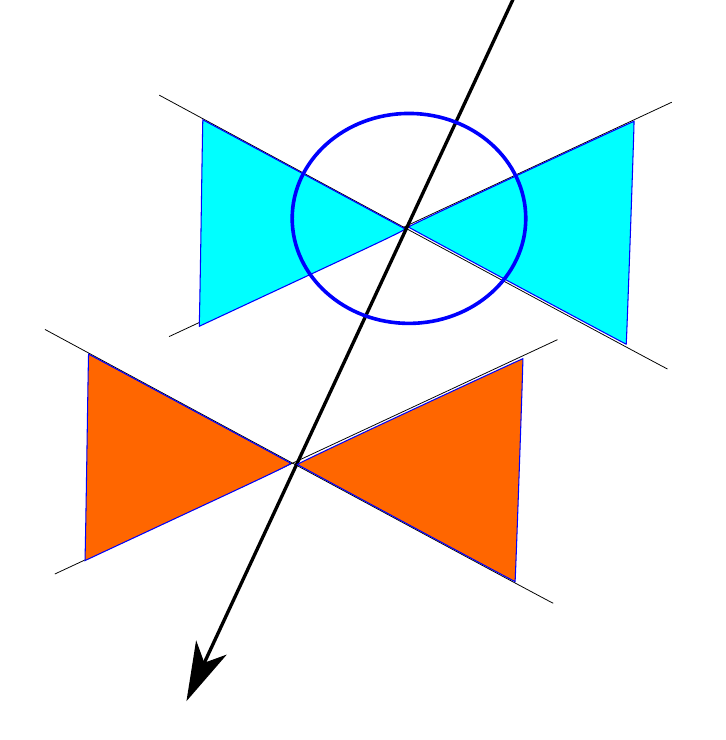}
\put (20,70){$\textcolor{Coneblue}{ C}$}
\put (-10,30){$\textcolor{Coneorange} {C+\frac{1}{2} \frac{\partial}{\partial q_{n}} }$}
\put (30,5){$\frac{\partial}{\partial q_{n}}$}
\put (33,85){$\myBlue {B(0, \eta)}$}
 \end{overpic}
\caption{ The cone $C$ and its translation by $\frac{1}{2} \frac{\partial }{\partial q_n}$}
\label{fig-8}
\end{figure}

\CorrectionBlue{
This extends to the case where $f$ is smooth.
Indeed, let us consider a hamiltonian $K(q_{n},p_{n}, \bar q, \bar p)$ to be a truncation of $$F(q_{n},p_{n},\bar q, \bar p)=\frac{1}{2}(p_{n} -f(q_{n},\bar q, \bar p))$$ 
$K(q_n,f(q_n, \bar q, \bar p),\bar q, \bar p)=0$ and $ \frac{\partial K}{\partial p_n} (q_n,f(q_n, \bar q, \bar p),\bar q, \bar p)=\frac{1}{2}$. So for any $(q_n,p_n,\bar q, \bar p)\in \widetilde C$, $\varphi_K^t(q_n,p_n,\bar q, \bar p) =(q_n+\frac{t}{2}, p_n(t), \bar q (t), \bar p(t))$ and we get $$\varphi_K^1(\widetilde C) \cap B(0, \eta)= \emptyset$$
 Since $K$ is constant on $\widetilde C$,
we can truncate the flow so that it is unchanged on the trajectory of points in $\varphi_K^t(\widetilde C)  \cap B(0,1)$ and $\gamma(\varphi_K) \leq \delta$. 
}

\CorrectionBlue{Note that if we replace $\widetilde C$ by the region
$$\widetilde C_{g, \delta}=\left \{(q_n, p_n, \bar q, \bar p) \mid (q_n,\bar q, \bar p) \in C, g(q_n, \bar q, \bar p) \leq
  p_n \leq g(q_n, \bar q, \bar p)+\delta \right \}$$ 
  we can choose $K$ to be the truncation of
  a function $F$ such that 
  \begin{enumerate} 
  \item $F(q_{n},g(q_{n},\bar q, \bar p), \bar q, \bar p)=0$
  \item $F(q_{n},g(q_{n},\bar q, \bar p)+\delta, \bar q, \bar p)= \frac{1}{2} \delta$ 
  \item $ \frac{\partial F}{\partial p_{n}} (p_{n}, q_{n}, \bar q, \bar p) =  \frac{1}{2} $ for $ g(q_{n},\bar q, \bar p) \leq p_{n}\leq g(q_{n},\bar q, \bar p)+\delta$
  \end{enumerate} 
   As a result the flow of $K$ preserves $\widetilde{C}_{g,\delta}$ and on $\widetilde{C}_{g,\delta}$ is given by $(q_{n},p_{n},\bar q, \bar p)\mapsto (q_{n}+\frac{1}{2} \frac{\partial}{\partial q_{n}},p_{n},\bar q, \bar p)$. Then  $\varphi_{K}^{1}(\widetilde{C}_{g,\delta})\cap B(0,\eta)=\emptyset$, since the projection of 
  $\varphi_{K}^{1}(\widetilde{C}_{g,\delta})$ on $H$ does not intersect $B(0,\eta)$. 
However
now we cannot have $\gamma(\varphi_K)$ arbitrarily small because $K$ has oscillation bounded from below by $2\delta$.  Indeed we must have $\frac{\partial K}{\partial p_n}=\frac{1}{2}$ on a segment in the $p_n$ direction of
length $\delta$. However we can assume $\gamma (\varphi_K) \leq \Vert K \Vert \leq \frac{1}{2} \delta+ \delta = \frac{3\delta}{2}  $.
}

\CorrectionBlue{
The general case is obtained as follows. We take a sequence $g_k$ such that $g_{k}\leq f\leq g_{k}+\frac{1}{k} $. We can then 
displace $\widetilde C_{g_k,\delta}$ hence $\widetilde C$ by a map with $\gamma (\varphi_k) \leq \frac{3}{2k} $. 
This  concludes our proof. 
}
\end{proof} 
\CorrectionBlue{
Let us now conclude the proof of the Proposition. According to Lemma \ref{Lemma-8.3} there is a  ball $B(0, \varepsilon )$ such that $V\cap B(0, \varepsilon )$ can be written as the graph of $f_{H,V}$ over $\Gamma\subset C_{\rho}^{-}(0,V) \subset H$. By assumption $C_{\rho}^{-}(0,V)$ does not contain $H^{\omega}$, so 
according to Lemma  \ref{Lemma-8.4}, where we replace $B(0,1)$ by $B(0,\varepsilon )$, we have for any $\delta>0$  a Hamiltonian map $\varphi$ supported in $B(0, \varepsilon )$, such that $\gamma(\varphi)<\delta$ and 
$\varphi(V)\cap B(0, \eta)=\emptyset$. This means that $V$ is not $\gamma$-coisotropic at $0$. 
}
\end{proof} 
\begin{rem} 
The proof actually shows that $V$ is locally rigid in the sense of Usher (see \cite{Usher}). So we have that 
$\gamma$-coisotropic implies locally rigid  which implies  cone-coisotropic.
\end{rem} 

\subsection{A cone-coisotropic set that is not \texorpdfstring{$\gamma$}{gamma}-coisotropic}
Let $K_{a}$ be the central Cantor set of ratio $a$. In other words we set \begin{gather*} F_{1}=[0,a]\cup [1-a,1]\\ F_{2}=[0,a^{2}]\cup [a(1-a),a]\cup [1-a,1-a+a^{2}]\cup [1- a^{2},1]\end{gather*} and $K_{a}=\bigcap_{k }F_{k} $.
\begin{prop} 
For $a<  2^{-2n}$, the set $X=K_{a}^{2n} \subset \Real^{2n}$ is cone-coisotropic but not $\gamma$-coisotropic. 
\end{prop}
\begin{proof}
  We claim that $C^+(z,X)$ is not contained in any hyperplane. Indeed for $x\in K=K_{a}$ the cone $C^+(x,K)$ is never empty
  since any point in $K$ is a cluster point of $K$. In fact $C^-(x,K)$ is either $ {\mathbb R}_-, {\mathbb R}_+$ or $
  {\mathbb R} $ depending whether $x$ is the limit of points of $K$ less than $x$, the limit of points of $K$ greater
  than $x$ or both, and $C^+(x,K) = \Real$.  Even though the paratingent cone of a product does not necessarily contain the product of the
  paratingent cones of the factors, in the present case $C^+(z,X)$ contains the rays $\{0^{j-1}\} \times C^+(x_j,K)
  \times \{0^{n-j}\}$ where $z=(x_1,..,x_{2n})$, and $0^k=(0,..., 0)\in {\mathbb R}^k$. Therefore $C^+(z,X)$ is not
  contained in any hyperplane and $X$ is trivially cone-coisotropic.

We now claim $X$ is not $\gamma$-coisotropic. Pick $z\in X$ and balls $B(z,\varepsilon)$, $B(z,\eta)$, $0<\eta< \varepsilon$.  Let $k$ big enough (to be made
  precise later). Since $K$ is contained in $F_k$ which consists of $2^{k}$ intervals of length $a^k$, our $X$ is
  contained in a family $\I$ of cubes of edge length $a^k$ with $|\I| = 2^{2nk}$.  The projection $(q,p) \mapsto
  (q_2,\dots,q_n,p_1,\dots,p_n)$ maps $\I$ to a family $\I'$ of cubes of $\Real^{2n-1}$ (again these cubes have edge
  length $a^k$ and $|\I'| = 2^{(2n-1)k}$).  We choose disjoints neighbourhoods $U_i$, $i\in \I'$, of the cubes in
  $\I'$.  For a given $i_0\in \I'$, let $C_j$, $j=1,\ldots,N$, be the cubes of $\I$ contained in $B(z,\varepsilon) \cap
  (\Real\times U_{i_0})$, ordered according to the $q_1$ variable. Let $N_1$ be the maximal index $j$ such that $C_{j}$ meets $B(z,\eta)$. 
  We push all cubes $C_2,\dots,C_{N_1}$ to the left (that is, decreasing their $q_1$
    coordinates), close to $C_1$, in the space between $B(z,\varepsilon)$ and $B(z,\eta)$ (if $k$ is big
  enough), as follows.  We move $C_{2}$ by a translation $T_2$ with direction $-\partial_{q_1}$, so that
  $T_2(C_{2})$ is close to $C_{1}$.  Then we move $C_{3}$ by a similar translation $T_3$, so that $T_3(C_{3})$
  is close to $T_2(C_{2})$.  We go on until $C_{N_1}$.  We have thus accumulated at most $2^k$ cubes of edge
  length $a^k$ near $C_{1}$. Since $a<1/2$, we have room to put all these cubes in $B(z,\varepsilon) \setminus
  B(z,\eta)$ for $k$ big enough.

  Now the translation $T_j$ can be realized by a Hamiltonian function of the form $h_j = h(p_1)\rho_j(q,p)$, where
  $h'=1$ and $\rho_j$ is a bump function vanishing outside $B(z,\varepsilon) \cap (\Real\times U_{i_0})$ and over
  $C_{1}\cup\dots\cup T_{j-1}(C_{j-1}) \cup C_{j+1} \cup\dots\cup C_{N_1}$ and equals to $1$ over the path swept by $C_j$.
  We have $||h_j|| \sim a^k$ and we apply the isotopy for a length of time less than $\varepsilon< 1$.  We have to do this
  less than $|\I|$ times. Since we chose $a < 2^{-2n}$, the composition of all these isotopies has norm as small
  as required.
  \begin{figure}[ht]
 \begin{overpic}[width=6cm]{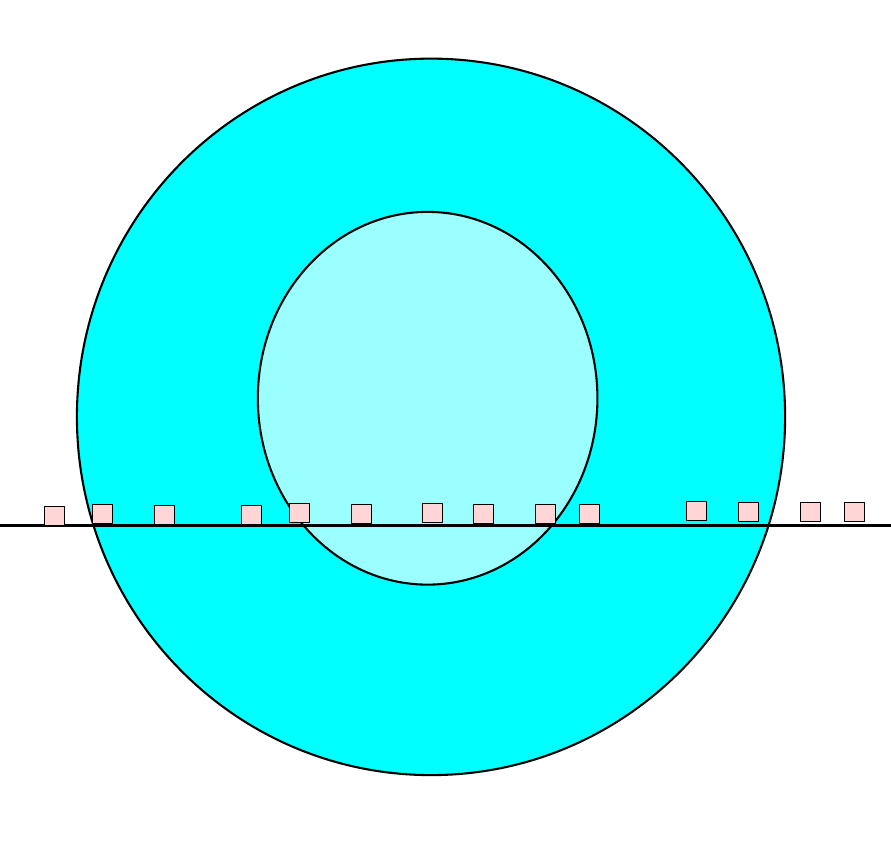}
 \end{overpic}
\begin{overpic}[width=6cm]{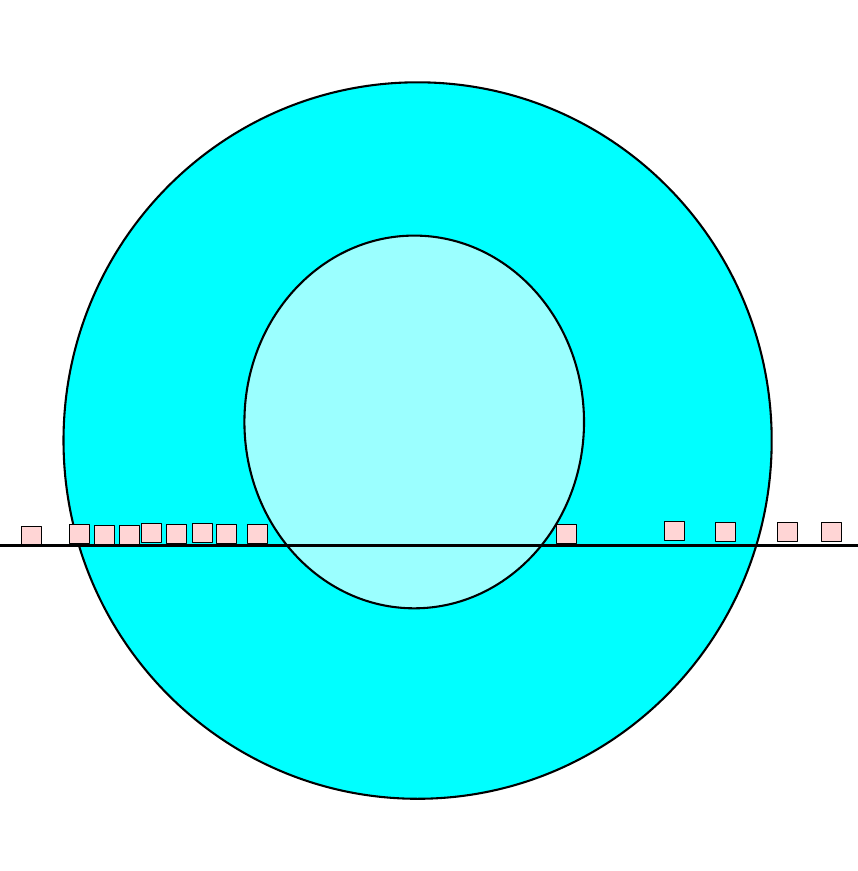}
 \end{overpic}
\caption{Moving the cubes..}
\label{fig-14}
\end{figure}

\end{proof}
\begin{rems} 
\begin{enumerate} 
\item
We can weaken the assumption $a< 2^{-2n}$ to $a<1/4$ at the expense of a more complicated proof. 
\item 
According to \cite{AGHIV} where we prove that $SS(\widehat Q(L))=\gamma-supp(L)$ and  \cite{MCA-VH-CV}, where we prove that the Birkhoff attractor is a $\gamma$-support, we may find a limit of constructibles sheaves such that $SS(\F)=\gamma-suppp(L)$ is an indecomposable continuum (i.e. a compact connected set that cannot be the non-trivial union of compact connected subsets). 
\end{enumerate} 
\end{rems} 

\begin{Question}
Find an example where the image of a smooth non-coiso\-tro\-pic submanifold $C$ by a symplectic homeomorphism, $\varphi$, is  such that $\varphi(C)$ is  cone-coisotropic (but will not be $\gamma$-coisotropic by \cite{Viterbo-gammas}). 
\end{Question}

\section{Questions and comments}
We gather in this section a number of questions that sprang up naturally while writing this paper. 
\begin{Question}
  What is the $\gamma_g$-distance of $k_X$ and $k_Y$ for two closed sets $X,Y$.  If $X,Y$ do not have the same
  cohomology, the distance is infinite. So we can focus on the case when $Y=X_ \varepsilon$ is an
  $\varepsilon$-neighbourhood of $X$. When is it true that $\gamma_g(k_X, k_{X_ \varepsilon} ) \leq C \varepsilon $ ?
  Or when do we have $\gamma_g-\lim k_{X_ \varepsilon }=k_X$ ?
\end{Question}

For the next four questions we choose on $N$ a real analytic structure.  We let $D_{lc}(N)$ be the category of limits
of constructible sheaves considered in Appendix~\ref{Appendix-B}.

\begin{Question}
  In  Appendix~\ref{Appendix-B} we describe the objects of $D_{lc}(\Real)$  with  microsupport  contained
  in $\{\tau\geq0\}$. Is there a similar description in the general case ?
\end{Question}

\begin{Question}
  Using Lemma~\ref{lem:limHom} and the proof of Lemma~\ref{lem:im_dim_finie} it is probably not hard to see that
  $D_{lc}(N)$ is a triangulated category.  If we have a map $f\colon N \to M$, can we find conditions so that
  $Rf_*(D_{lc}(N)) \subset D_{lc}(M)$, $f^{-1}(D_{lc}(M)) \subset D_{lc}(N)$?  
  
  For example this is true for a map of
  the type $f\times \id_\Real \colon N\times\Real \to M\times\Real$, if we consider the $\gamma_\tau$-distance,
  because $Rf_*$ and $f^{-1}$ are Lipschitz for this distance.  
  In general $f^{-1}$ is not Lipschitz:  take $f\colon \Real \to \Real^2$, $x \mapsto (x,0)$, and $D_\varepsilon$ the closed disc with center $(0,1)$ and
  radius $1-\varepsilon$; then $\gamma_g(k_{D_0}, k_{D_\varepsilon}) = \varepsilon$ but $\gamma_g(f^{-1}(k_{D_0}),
  f^{-1}(k_{D_\varepsilon})) = +\infty$ (here $g(q,p) = ||p||$). Hence it is not clear that $f^{-1}(D_{lc}(M))
  \subset D_{lc}(N)$.
\end{Question}

\begin{Question}
  Let $D_{dom}(N)$ be the subcategory of $D(N)$ generated in the triangulated sense by the $k_U$, where $U$ runs over
  the domains with smooth boundary.  It should not be difficult to prove that $D_{dom}(N)\subset D_{lc}(N)$ by
  approximating domains with smooth boundary by domains with analytic boundary.  Is it true that, for any subanalytic
  open subset $V$ of $X$ and any $\varepsilon>0$, there exists a domain $U$ with smooth boundary such that $U \subset
  V$ and $\fai{\varphi_g}^{\varepsilon}(k_U) \simeq k_{U_\varepsilon}$ with $U_\varepsilon$ open and $V \subset
  U_\varepsilon$ ?  In this case we can see that $D_{lc}(N)$ coincides with the $\gamma_g$-closure of $D_{dom}(N)$, which
  would imply that $D_{lc}(N)$ only depends on the smooth structure of $N$.  \end{Question}

\begin{Question}
  Let $D_{tri}(N)$ be the set of sheaves which are constructible with respect to some (non-fixed !)  triangulation of
  $N$.  Do the $\gamma_g$-closures of $D_{dom}(N)$ and $D_{tri}(N)$ coincide ? Are they equal to $D_{lc}(N)$?
\end{Question}

We remind the reader of the following (see \cite{K-S-distance}, theorem 2.11, and  \cite{Petit-Schapira}, corollary 2.4.2 )
\begin{prop} 
We have for all $1$-Lipschitz maps $f: M \longrightarrow N$ the inequality $$\gamma_g(Rf_*(\F), Rf_*(\G)) \leq \gamma_g(\F,\G)$$ 
\end{prop}

\begin{Question}
  Can we give conditions on $\F$,$\G$ (for example $\gamma_g(\F,\G) < \infty$) such that the opposite inequality
  $$ \gamma_g(\F,\G) \leq C\sup_{f\in Lip_1(X, {\mathbb R} )} \gamma_{|\tau|}(Rf_*(\F) , Rf_*(\G))  $$
  holds for some constant $C$?
\end{Question}

  \begin{Question}\label{Question-9.9}\footnote{This question is answered positively in \cite{AGHIV}}  Let $f$ be a continuous function on $N$. Then $\F_f=k_{\{(x,t) \mid f(x) \leq t\}}$ is a sheaf on $N\times {\mathbb R}$, and 
    Vichery (in \cite{Vichery}) defined a subdifferential as
    $$
    \partial f= \rho(SS(\F_f)\cap \{\tau=1\}) \subset T^*N,
    $$
\CorrectionBlue{where $\rho(x,t,\xi,\tau) = (x,\xi/\tau)$ is defined in~\S\ref{sec:proofmainthm}.}
    Its intersection with $T_x^*N$ is $\partial f(x)$. 
  On the other hand $df \in \widehat  \LL (T^*N)$, so we may ask whether
  $$ \rho( SS(\F_f)\cap \{\tau=1\}) =\partial f = \gamma-\supp (df)$$
     \end{Question}
This leads to more general questions. 
First notice the following
\begin{prop} 
There is a functor $Q\colon \widehat \LL (T^*N) \longrightarrow D_{lc}(N\times {\mathbb R} )$ extending the usual quantization from \cite{Guillermou, Viterbo-Sheaves}. Moreover $Q$ is an isometric embedding  from $(\widehat{\LL}(T^*N), \gamma)$ to 
$(D_{lc}(N\times {\mathbb R} ), \gamma_\tau)$. 
\end{prop} 
\begin{proof} 
Let $L_n$ be a Cauchy sequence in $\LL(T^*N)$. Then $\F_{L_n}$ is a Cauchy sequence in $D_{lc}(N\times {\mathbb R} )$. The map $L \longrightarrow \F_L$ is an isometric embedding from the metric $c$ to the (pseudo-) metric $\gamma_\tau$.  As we saw in Proposition \ref{Prop-6.13}
where $c$ is the metric defined on 
$\LL (T^*N)$  in \cite{Viterbo-gammas} by $c(\widetilde L_1,\widetilde L_2)=\max\{0,c_+(L_1,L_2)\} - \min\{0,  c_-(L_1,L_2) \}$. 
Since according to Proposition \ref{prop:unicite-limite} the metric $\gamma_\tau$ is a bona fide metric in $D_{lc}(N\times {\mathbb R} )$, limits are unique and the limit of $\F_{L_n}$ is well-defined. 
\end{proof} 

  \begin{Question}\footnote{This question is answered positively in \cite{AGHIV}}
Set $\rho(x,t;\xi,\tau) = (x;\xi/\tau)$. Given $\Lambda \in \widehat{\LL}(T^*N)$, do we have
    $\gamma-\supp(\Lambda) = \rho(SS^\bullet(Q(\Lambda)))$ ?

    A positive answer would also imply a positive answer to Question \ref{Question-9.9}, i.e.  $\partial f = \gamma-\supp (df)$ by
    taking $\Lambda_n = df_n$ for some sequence of differentiable functions converging to $f$.
  \end{Question}

\begin{Question}
  Can one characterize the sheaves $\F$ such that there exists an element
  $\widehat {SS}(\F)$ in $\widehat {\LL} (T^*N)$ so that
  $SS(\F)=\gamma-\supp (\widehat {SS}(\F))$ ?  In the previous question we
  considered the $\F$ which are $\gamma$-limits of sheaves $\F_j$ quantizing some
  exact Lagrangian. 
  
  Let us give an example of such a situation : let $M,N$ be closed manifolds. For a map $f\in C^\infty(M,N)$ set $\Lambda_f$ be the correspondence in $T^*M\times T^*N$ defined by 
  $$\Lambda_f= \left\{(x,p_x,y,p_y) \mid y=f(x), p_y\circ df(x)=-p_x\right\}$$
Now let $f_n$ be a  sequence of smooth maps $C^0$-converging to $f\in C^0(M,N)$. For $\F\in D^b(M)$ the sequence
  $(R(f_n)_*\F)_{n\geq 1}$ $\gamma$-converges to $Rf_*(\F)$ and $SS (R(f_n)_*\F)\subset \Lambda_{f_n}\circ SS(\F)$. 
   Thus $SS(R(f_n)_*\F)$ ) should $\gamma$-converge to a subset of $\Lambda_{f}\circ SS(\F)$. Of course since $f$ is only $C^0$, $\Lambda_{f} $ is not defined as a set, but it is well defined in $\widehat{\LL} (T^*N\times T^*M)$, hence 
  $\Lambda_f\circ SS(\F) \in \widehat{\LL}(T^*N)$. Thus we expect $\widehat {SS}(Rf_*(\F))=\Lambda_f\circ SS(\F)$ and 
  $SS(Rf_*\F)\subset\gamma-\supp (\Lambda_f\circ SS(\F))$. 
 \end{Question}

 \begin{Question}
   Let $\F_n$ be a sequence in $D(N)$ $\gamma_g$-converging to $\F_\infty$.  We set
   $\Lambda_n = SS(\F_n)$ and we assume that $\Lambda_n$ Hausdorff converges to
   $\Lambda_\infty$. We know that $SS(\F_\infty)\subset \Lambda_\infty$ but
  when is it true that $\Lambda_\infty = SS(\F_\infty)$?  We can prove
   it if $\Lambda_\infty$ is a smooth Lagrangian or more generally minimally $\gamma$-coisotropic.
     Indeed Proposition
   \ref{Prop-g-continuity} shows that $SS(\F_\infty) \subset \Lambda_\infty$ and the
   involutivity of the singular support gives the equality (by an argument similar to the
   one in the proof of Proposition \ref{Prop-9.13}).  Assuming $N = M\times\Real$ and
   the $\F_n$ are quantizations of exact Lagrangian, Proposition \ref{Prop-9.13} shows
   that this still hold when $\Lambda_\infty$ is the conification of some
   $ \varphi(L_0)$ for $L_0$ smooth Lagrangian, $\varphi $ a symplectic homeomorphism.
 \end{Question}
  
  Note however that the answer cannot be positive in general:
  Let $f(x)$  be a smooth bounded function on $\mathbb R$, such that $\max{f'} =
1$, $\min{f'} = -1$.
Let  $f_n(x) = f(nx)/n$ and $\F_n = k_{\{t\geq f_n(x)\}}$.
Then $\F_n$ converges to  $\F_\infty = k_{\mathbb R \times [0,\infty[}$ but
$SS(\F_n)$ converges to a set  $\Lambda_\infty$ bigger than $SS(\F_\infty)$. In fact \CorrectionBlue{$\Lambda_\infty$} is the cone 
over a ``cross'' $C$
in $T^*\Real$ with  $C = 0_\Real \cup (\{0\}\times [-1,1])$.

   \begin{prop}\label{Prop-9.13} 
  Let $L= \varphi(L_0)$ where $L_0$ is smooth Lagrangian and $\varphi \in  \mathcal {H}_\gamma (M, \omega)= \widehat\DHam (M,\omega)\cap Homeo(M)$. Then any closed proper subset of $L$ is not $\gamma$-coisotropic.
   \end{prop} 
   \begin{proof} 
  \CorrectionBlue{ Let $L'\varsubsetneq L$.} Since being $\gamma$-coisotropic is invariant by $\widehat\DHam(M,\omega)\cap Homeo(M)$ it is enough to prove the Proposition for $L_0$ i.e. in the smooth case. 
   Now there is a ball $B(z,r)$ in $L$ such that $\overline {B(z,r)}\cap L'=\{z'\}$. Moreover we may choose $r$ to be arbitrarily small. Let us consider a flow on $L$ such that $\rho^t(B(z,r)\subset B(z',r)$. 
   In a chart containing $B(z,2r)$ we may assume that $\rho^t$ is a translation. Then $\rho^1(L')\cap B(z',r)=\emptyset$. Since if $X_t$ is the vector field generating $\rho^t$, the Hamiltonian written in action-angle coordinates (i.e. $L$ is given locally by $p=0$) can be chosen of the form $\chi ( \vert p \vert ) \langle p , X_t(x)\rangle$ for some bump function $\chi$, and this can be made arbitrarily small. As a result, $L'$ is not $\gamma$-coisotropic. 
   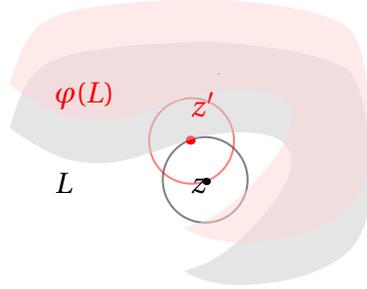
\begin{figure}[H]
\centering\begin{overpic}[width=6cm]{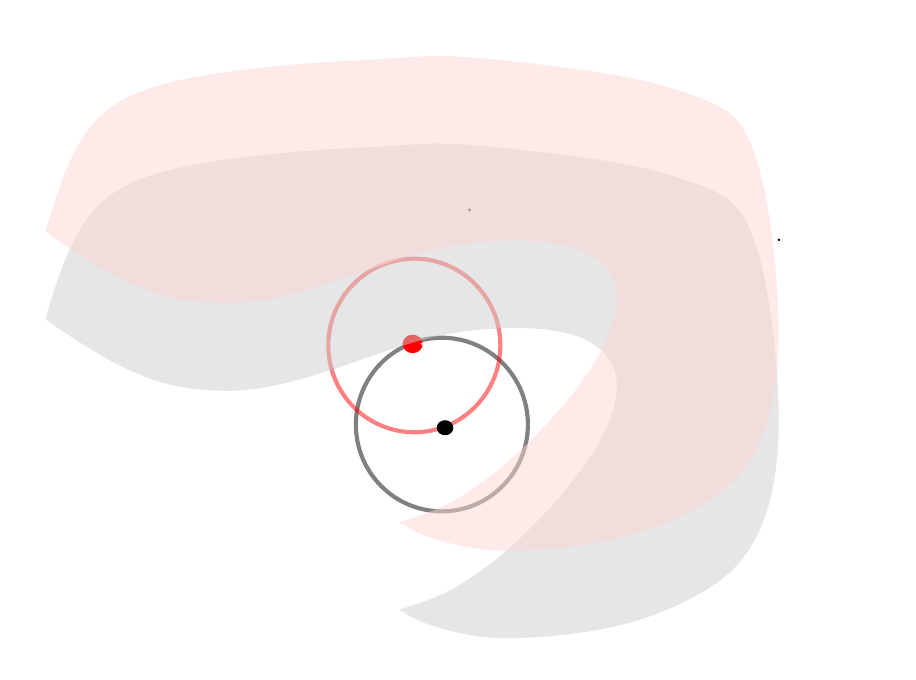}
 \put (45,42) {$\myRed {z'}$} 
  \put(15,45){$\myRed {\varphi(L)}$}
 \put(45,25){$z$}
   \put(15,25){$L$}
 \end{overpic}
\caption{ Proof of Proposition \ref{Prop-9.13}}
\label{fig-10}
\end{figure}

   \end{proof} 

\begin{Question} Let $V$ be $\gamma$-coisotropic of dimension $n$. Then for $z\in V$ and $ \varepsilon $ small enough we have
  $V\setminus B(z, \varepsilon )$ is not coisotropic.
  \end{Question}

\appendix
\section{ Persistence modules and Barcodes}\label{Appendix-barcodes-Floer}
For this section we refer also to \cite{Zhang-QTC}

We have already noticed after Definition~\ref{def:spectral-inv} that the spectral invariants of $L_1, L_2\in \mathcal
L(T^*N)$ are encoded in the sheaf over $\Real$ given by $\H = Rt_* \homst(\F, \G))$, where $t\colon
N\times \Real \to \Real$ is the projection, because of the isomorphism
$$
\rsect_{N\times \{t\}}(N\times\Real; \homst(\F_{L_1},\F_{L_2}))) \simeq \rsect_{\{t\}}(\Real; \H) \dispdot
$$
We recall that $\H$ belongs to the category $D_{\tau\geq0}(\Real)$ of sheaves with singular support in $\{\tau\geq
0\}$.

We remark that $D_{\tau\geq0}(\Real)$ contains the derived category of {\em persistence modules}, where by a
persistence module, we mean a constructible sheaf on $( {\mathbb R}, \leq)$. We refer to \cite{Barannikov,
  Chazal-CS-G-G-O, E-L-Z, K-S-distance, Z-C} for the theory and applications.  Such a persistence module is uniquely
defined by the finite-dimensional vector spaces $V_t= \mathcal F (]-\infty, t[)$ and the linear maps $r_{s,t}: V_t
\longrightarrow V_s$ defined for $s\leq t$, such that
\begin{enumerate} 
\item for $s<t<u$ we have $r_{u,t} \circ r_{s,t}=r_{s,u}$,
\item $\lim_{t>s} V_t=V_s$ where the limit is that of the directed system given by the $r_{s,t}$,
\item $r_{t,t}=\Id$.
\end{enumerate} 
It is a well-known result that persistence modules have a decomposition as sum of barcodes : this is Gabriel's
theorem (\cite{Gabriel} on quivers in the finite case and Crawley-Boevey's theorem (\cite{Crawley-Boevey}) in the
locally finite case. More precisely, using Gabriel's theorem and the fact that any complex of objects in an Abelian
category of homological dimension $1$ is the sum of its cohomology, we find: any $\F \in D_{\tau\geq0}(\Real)$ which
is constructible for a finite stratification can be written $\F \simeq \bigoplus_{j \in \I} k_{[a_j, b_j[}[-n_j]$,
where $\I$ is a finite family, and this decomposition is unique up to isomorphism.  (The hypothesis that the singular
support is in $\{\tau\geq0\}$ is useless here.)  We remark that the stratification consists only of a finite set of
points in $\Real$ and the connected components of its complement, where $\F$ is constant. This finite set is thus the
projection to $\Real$ of $SS^\bullet(\F)$.  If $L \in \LL(T^*N)$ is generic in the sense that $L$ meets $0_N$
transversely, then $Rt_*(\F_L)$ is of the above type, where the stratification is given by the set of actions of
the points in $L \cap 0_N$.

In~\cite{K-S-distance} it is explained how to go from the case of a finite stratification to the case of a locally
finite one (then $\I$ is at most countable).  By~\cite{Crawley-Boevey} the same result holds in fact (at least in
$D_{\tau\geq0}(\Real)$) under the more general hypothesis that $\F_x$ is finite dimensional for all $x\in \Real$.
However the sheaves we encounter in this paper are not necessarily of this type: if $f$ is the function $f(x) =
x^2\sin(1/x)$ and $\F = Rt_*(k_{\{t\geq f(x)\}})$, then $\F_0$ is infinite dimensional. But, by a small perturbation
of $L$, we see that all sheaves of the form $Rt_*(\F_L)$ are in the closure of the set of constructible sheaves.
Using Corollary~\ref{cor:decomp_limite} we deduce:

\begin{prop}\label{Prop-A1}
  Let $L_1, L_2\in \LL(T^*N)$. Then there exists an at most countable family $\I$ (which is finite in the case
  of transverse intersection) such that
  $$
  (Rt)_*(\homst(\F_{L_1}, \F_{L_2}))=\bigoplus_{j \in \I} k_{[a_j, b_j[}[-n_j] ,
  $$
  where for each $j\in \I$ we have $a_j \in \Real\cup\{-\infty\}$, $b_j \in \Real$, $n_j \in \Z$.  Moreover there is a
unique $j_-$ such that $a_{j_-}=-\infty$ and $n_{j_-} = -1$ and then $b_{j_-}=c_-(L_1,L_2)$ and a unique $j_+$ such
that $a_{j_+}=+\infty$ and $n_{j_+} = n-1$ and then $b_{j_+}=c_+(L_1,L_2)$.
\end{prop}

More generally, with the notations of the proposition we let $\I_\infty$ be the set of $j$ such that $a_j = -\infty$.
Then $H^*_{t_0}((Rt)_*(\homst(\F_{L_1}, \F_{L_2}))) \simeq \bigoplus_{j \in \I_\infty} k[-n_j]$, for $t_0
\ll0$, and we deduce that the spectral invariants $c(\alpha, L_1,L_2)$ are the $b_j$'s with $j\in \I_\infty$.

Note that the filtered Floer cohomology, $FH^*(L_1,L_2;-\infty ,t)$ defines a persistence module equal to the one obtained from 
$\homst(\F_{L_1}, \F_{L_2})$ as follows from Theorem \ref{Thm-Quantization}. 

It will be useful to remind the reader that 

\begin{lem}[see \cite{K-S-distance}, (1.10)]\label{lem:HomkIkJ}
  $$
  \RHom (k_{[a,b[}, k_{[c,d[}) \simeq
  \begin{cases}
    k  & \text{for \ } a\leq c <b\leq d , \\
  k[-1] \quad & \text{for \ } c < a \leq d < b , \\
    0  & \text{otherwise.}
  \end{cases}
  $$
\end{lem} 
\begin{proof}
  The formula $k_{[c,d[} \simeq \rhom(k_{]c,d]}, k_\Real)$ and the adjunction between $\otimes$ and $\rhom$
  give $\RHom (k_{[a,b[}, k_{[c,d[}) \simeq \RHom(k_I,k_\Real)$, where $I = [a,b\mathclose[ \cap \mathopen]c,d]$.  If
  $I$ is half closed or empty, the result is $0$. The two other cases correspond to $I$ open or $I$ closed.
\end{proof}

\section{Decomposition in the completion of \texorpdfstring{$D_{lc}( {\mathbb R} )$}{DlcR}}\label{Appendix-B}

\subsection{Limits of constructible sheaves}

We assume in this section that the $\gamma_g$-distance on $D_{g\geq 0}(N)$ is given by a {\em real analytic}
function $g \colon T^*N \setminus 0_N \to \Real$.  We recall that we work with coefficients in a field $k$.
We denote by $D_c(N)$ the subcategory of $D(N)$ of constructible sheaves, that is, the sheaves $\F$ for which
there exists a locally finite subanalytic stratification of $N$ such that the restriction of $\F$ to each
stratum is locally constant of finite rank (see~\cite[Chap. VIII]{K-S}).  We let $D_{lc}(N)$ be the set of
objects which are limits of objects in $D_c(N) \cap D_{g\geq 0}(N)$ with respect to the $\gamma_g$-topology.
Since $g$ is analytic, $\varphi = \varphi_g$ is also analytic and $K_\varphi \in D_c(N^2)$. Hence
$\fai{\varphi}^t(-)$ preserves $D_c(N) \cap D_{g\geq 0}(N)$. A sheaf $\F$ belongs to $D_{lc}(N)$ if and only
if, for any $\varepsilon>0$, there exists $\F'\in D_c(N) \cap D_{g\geq 0}(N)$ and morphisms
$\fai{\varphi}^{-\varepsilon}(\F) \to \F' \to \F \to \fai{\varphi}^\varepsilon(\F') \to
\fai{\varphi}^\varepsilon(\F)$) whose compositions are the morphisms $\tau$ (notice that
$\fai{\varphi}^\varepsilon(\F')$ itself belongs to $D_c(N)$).

Our aim in this section is to prove that the distance $\gamma_g$ is non degenerate on $D_{lc}(N)$.  For $\F,\G
\in D_{lc}(N)$ with $\gamma_g(\F,\G) = 0$ we will prove the existence of an isomorphism between $\F$ and $\G$
as follows: we first see that $\Hom(\F,\G)$ is almost $\varprojlim_n \Hom(\F, \fai{\varphi}^{1/n}(\G))$ and
then we use the morphisms in $\Hom(\F, \fai{\varphi}^{1/n}(\G))$ whose existence follows from the condition
$\gamma_g(\F,\G) = 0$.  Actually we will prove
$$
\Hom(\F,\G) \simeq \varprojlim_n \Hom(\F|_{U_n}, \fai{\varphi}^{1/n}(\G) |_{U_n})
$$
where $U_n$ is an increasing family of open subsets with compact closures.  The reason for introducing these
$U_n$'s is that this enables us to check the Mittag-Leffler criterion by approximating $\F$ and $\G$ by
constructible sheaves.  We often write $\Hom(\F|_U, \G|_U)$ for $\Hom_{D(U)}(\F|_U,\G|_U)$ when the working
category (here $D(U)$) is understood. We will also use the isomorphisms, for $\F, \G \in D(N)$ and $U\subset
N$ open,
$$
\Hom(\F,\G) \simeq H^0\RHom(\F,\G), \qquad \Hom(\F|_U,\G|_U) \simeq \Hom(\F_U, \G),
$$
where $\F_U = \F \otimes k_U$.

When cutting by the open sets $U_n$ we need to compare $\fai{\varphi}^a(\F_U)$ and $\fai{\varphi}^a(\F)$.
This is the object of the next lemma.

\begin{lem}\label{lem:bound_support}
  Let $V\subset U$ be open subsets of $N$ such that $\overline{V}$ is compact and $\overline{V} \subset U$.
  There exists $a_0>0$ such that, for any $a\in [-a_0,a_0]$ and $\F\in D(N)$, the morphisms
  $$
  (\fai{\varphi}^a(\F_V))_U \to \fai{\varphi}^a(\F_V),
  \qquad
  (\fai{\varphi}^a(\F_U))_V \to (\fai{\varphi}^a(\F))_V,  
$$
  are isomorphisms.
\end{lem}
\begin{proof}
  Set $\overline V_a = \bigcup_{0\leq b \leq a} \pi( \varphi_b(\pi^{-1}(\overline V)))$, where $\pi\colon
  T^*N\setminus 0_N \to N$ is the projection.  Since $\overline{V}$ is compact, there exists $a_0>0$ such that
  $\overline V_{a_0} \subset U$.
  
  Let $\G \in D(N\times [-a_0,0])$ be the sheaf obtained from $\F_V$ by applying the whole isotopy, that is,
  $\G|_{N\times \{t\}} = K_{\varphi_t}(\F_V)$.  Since $SS(\F_V) \subset \pi^{-1}(\overline V)$, we have
  $SS(\G) \subset T^*\overline V_{a_0} \times T^*[-a_0,0]$ outside the zero-section. Hence, for any $x\in
  N\setminus \overline V_{a_0}$, $\G$ is locally constant near $x \times [-a_0,0]$, hence $0$ near $x \times
  [-a_0,0]$ since $(\G)_{(x,0)} =0$.  It follows that the support of $\fai{\varphi}^a(\F_V)$ is contained
  in $U$ and this gives the first isomorphism.

  We also see that the support of $\fai{\varphi}^a(\F_{N\setminus U})$ avoids $V$, which gives the second
  isomorphism.
\end{proof}

\begin{lem}\label{lem:im_dim_finie}
  Let $\F,\G \in D_{lc}(N)$, let $\varepsilon>0$ and let $V \subset U$ be two open subsets of $N$.  We assume
  that $\overline{V}$ is compact and $\overline{V} \subset U$.  Then the morphism $\Hom_{D(U)}(\F|_U, \G|_U)
  \to \Hom_{D(V)}(\F|_V, \fai{\varphi}^\varepsilon(\G)|_V)$ induced by $\tau_{0,\varepsilon}(\G)$ and the
  restriction from $U$ to $V$ has finite dimensional image.
\end{lem}
\begin{proof}
  We recall that $\Hom_{D(U)}(\F|_U, \G|_U) \simeq \Hom_{D(N)}(\F_U, \G)$ and our morphism is given by
  composing with $\F_V \to \F_U$ and $\tau_{0,\varepsilon}(\G)$.  For $0<\eta <\varepsilon$ we have $\Hom(\F_V,
  \fai{\varphi}^\varepsilon(\G)) \simeq \Hom(\fai{\varphi}^{-\eta}(\F_V),
  \fai{\varphi}^{\varepsilon-\eta}(\G))$.  We choose a subanalytic open subset $W$ such that $\overline{W}$ is
  compact and $\overline{V} \subset W \subset \overline{W} \subset U$.  By Lemma~\ref{lem:bound_support}
  and the commutative diagram
$$
\xymatrix@C=25mm{
  (\fai{\varphi}^{-\eta}(\F_V))_W  \ar[r]^\sim \ar[d]
  &  \fai{\varphi}^{-\eta}(\F_V) \ar[r]^-{\tau_{-\eta,0}(\F_V)} \ar@{.>}[dl] &  \F_V \ar[d]  \\
  (\fai{\varphi}^{-\eta}(\F))_W \ar[rr]^-{(\tau_{-\eta,0}(\F))_W} && \F_W
}
$$
the composition $\fai{\varphi}^{-\eta}(\F_V) \to[\tau_{-\eta,0}(\F_V)] \F_V \to \F_W$ is factorized as
$\fai{\varphi}^{-\eta}(\F_V) \to (\fai{\varphi}^{-\eta}(\F))_W \to[(\tau_{-\eta,0}(\F))_W] \F_W$, for $\eta$
small enough.  Since $\F,\G \in D_{lc}(N)$, there exist constructible sheaves $\F'$, $\G'$ and morphisms
$\fai{\varphi}^{-\eta} \F \to[u_1] \F' \to[u_2] \F$, $\G \to[v_1] \G'
\to[v_2]\fai{\varphi}^{\varepsilon-\eta}(\G)$ such that $u_2 \circ u_1 = \tau_{-\eta,0}(\F)$, $v_2 \circ v_1 =
\tau_{0,\varepsilon-\eta}(\G)$. We sum up these morphisms in the diagram
  $$
  \xymatrix{
    \fai{\varphi}^{-\eta}(\F_V) \ar[r]\ar@{.>}[rrrrrdd]_\gamma &  (\fai{\varphi}^{-\eta}(\F))_W \ar[r]
    &  \F'_W \ar[r]\ar@{.>}[rrrd]_\beta &  \F_W \ar[r]  & \F_U \ar@{.>}[r]_\alpha &  \G  \ar[d] \\
    &&&&& \G' \ar[d] \\
    &&&&& \fai{\varphi}^{\varepsilon-\eta}(\G)
  }
  $$
  We deduce that the morphism of the lemma, which we have rewritten as
  $$
  \Hom(\F_U, \G) \to
  \Hom( \fai{\varphi}^{-\eta}(\F_V), \fai{\varphi}^{\varepsilon-\eta}(\G))
  \simeq  \Hom(\F_V, \fai{\varphi}^\varepsilon(\G))  \dispdot
  $$
  and corresponds to $\alpha \mapsto \gamma$ in the diagram, factorizes as $\alpha \mapsto \beta$ and $\beta
  \mapsto \gamma$ through $\Hom(\F'_W, \G')$ which is finite dimensional.
\end{proof}

We choose an increasing family of open subsets $U_n$, $n\in \N$, of $N$ such that $N = \bigcup_{n\in \N} U_n$,
each $\overline{U_n}$ is compact and $\overline{U_n} \subset U_{n+1}$.  For $\F,\G \in D_{lc}(N)$ we set
$$
A_n(\F,\G) = \Hom(\F_{U_n}, \fai{\varphi}^{1/n}(\G) )  \dispdot
$$
The morphisms $\tau_{1/n,1/(n-1)}(\G)\colon \fai{\varphi}^{1/n}(\G) \to \fai{\varphi}^{1/(n-1)}(\G)$ and
$\F_{U_{n-1}} \to \F_{U_{n}}$ induce $a_n\colon A_n(\F,\G) \to A_{n-1}(\F,\G)$.  Lemma~\ref{lem:im_dim_finie}
implies that the projective system $(A_n(\F,\G), a_n)$, satisfies the Mittag-Leffler criterion.

By Lemma~\ref{lem:F=holim-translates} we have an isomorphism $f\colon \hocolim \fai{\varphi}^{-1/n}(\F) \isoto
\F$.  The morphisms $f_{n-1} \colon \fai{\varphi}^{-1/(n-1)}(\F_{U_{n-1}}) \to
\fai{\varphi}^{-1/n}(\F_{U_{n}})$ give a similar injective system $(\fai{\varphi}^{-1/n}(\F_{U_n}), f_n)_n$
and the morphisms $\fai{\varphi}^{-1/n}(\F_{U_n}) \to \F$ yield
$$
f'\colon \hocolim \fai{\varphi}^{-1/n}(\F_{U_n}) \to \F
$$
which is also an isomorphism: indeed this can be checked locally and we can restrict to a relatively compact
open subset $V$; for $n$ big enough we have $(\fai{\varphi}^{-1/n}(\F))|_V \simeq
(\fai{\varphi}^{-1/n}(\F_{U_n}))|_V$ by Lemma~\ref{lem:bound_support} and $f'$ coincides with $f$.  Applying
$\RHom(-,\G)$ we obtain
\begin{equation}\label{eq:homotlimit_rhomFG}
  \RHom(\F,\G) \isoto \holim_n \RHom(\F_{U_n}, \fai{\varphi}^{1/n}(\G)) \dispdot
\end{equation}
We can deduce $H^i\RHom(\F,\G) \simeq \varprojlim_n H^i\RHom(\F_{U_n}, \fai{\varphi}^{1/n}(\G))$ using the
Mittag-Leffler criterion. Since this criterion is not often stated in the framework of homotopy limits, we
recall it in the next lemma (which is a variation on the usual Mittag-Leffler condition; see for
example~\cite[Prop.~1.12.4]{K-S}).  Here $\Ab$ denotes the category of Abelian groups and $D(\Ab)$ its derived
category.
\begin{lem}\label{lem:ML}
  Let $(A_n, a_n)$, $a_n\colon A_n \to A_{n-1}$, be a projective system in $D(\Ab)$ and define $\holim_n A_n$
  by the distinguished triangle $\holim_n A_n \to \prod_n A_n \to[\alpha] \prod_n A_n \to[+1]$ where $\alpha =
  \id - \prod_n a_n$.  We assume that $H^iA_n$ satisfies the Mittag-Leffler criterion, for some $i$, that is,
  for each $n$ the image of $H^iA_m \to H^iA_n$ stabilizes when $m\to\infty$. Then $H^{i+1}(\holim_n A_n)
  \simeq \varprojlim_n H^{i+1}A_n$.
\end{lem}
\begin{proof}
  In $\Ab$ the product $\prod$ is exact, hence $H^i$ and $\prod$ commute and we have a long exact sequence $\prod_n
  H^iA_n \to[H^i\alpha] \prod_n H^iA_n \to H^{i+1}(\holim_n A_n) \to \prod_n H^{i+1} A_n \to \prod_n H^{i+1} A_n$.
  Hence the statement is equivalent to the surjectivity of $H^i\alpha$.  We write for short $B_n = H^iA_n$, $b_n =
  H^ia_n$ and $B_n^\infty = \operatorname{im}(B_m \to B_n)$ for $m\gg n$.  We let $b_n^\infty$, $\bar b_n$ be the
  maps induced by $b_n$ on $B_n^\infty$, $B_n / B_n^\infty$.  We have $H^i\alpha = \id - \prod b_n$ and the
  commutative diagram
  $$
  \xymatrix@C=1.5cm{
    0 \ar[r]  &  \prod B_n^\infty  \ar[r] \ar[d]^{\id - \prod b_n^\infty}
    &   \prod B_n \ar[r] \ar[d]^{\id - \prod b_n}
    &   \prod B_n/B_n^\infty  \ar[r] \ar[d]^{\id - \prod \bar b_n} & 0  \\
        0 \ar[r] &  \prod B_n^\infty  \ar[r] &   \prod B_n \ar[r]
    &   \prod B_n/B_n^\infty  \ar[r] & 0\dispdot }
  $$
  We see that $\id - \prod b_n^\infty$ is surjective because the $b_n^\infty$'s are surjective.  Let us write $t =
  \prod \bar b_n$.  By the Mittag-Leffler criterion, for each $n$, there exists $m\geq n$ such that $\bar b_{n+1}
  \circ \bar b_{n+2} \circ \dots \circ \bar b_m$ is the zero map.  It follows that $t' = \sum_{i=0}^\infty t^i$ is
  well-defined.  Hence $\id - t$ is an isomorphism, with inverse $t'$.  We deduce that $\id - \prod b_n$ is also
  surjective, as required.
\end{proof}

We apply this result to the homotopy limit in~\eqref{eq:homotlimit_rhomFG}.  By Lemma~\ref{lem:im_dim_finie}
the maps $H^i\RHom(\F_{U_{n+1}}, \fai{\varphi}^{1/(n+1)}(\G)) \to H^i\RHom(\F_{U_n}, \fai{\varphi}^{1/n}(\G))$
have finite dimensional image (recall that $H^i\RHom(X,Y) \simeq \Hom(X,Y[i])$ in any derived category).
Hence we can indeed apply Lemma~\ref{lem:ML} and we obtain
\begin{lem}\label{lem:limHom}
  Let $\F,\G \in D_{lc}(N)$ and let $U_n$, $n\in \N$, be a sequence of open subsets as above.  We have a natural
  isomorphism
  $$
  \Hom(\F,\G) \isoto  \varprojlim_n \Hom(\F_{U_n}, \fai{\varphi}^{1/n}(\G))  \dispdot
  $$
\end{lem}

We recall that $A_n(\F,\G) = \Hom(\F_{U_n}, \fai{\varphi}^{1/n}(\G) )$.  We set
$$
B_n(\F,\G) = \operatorname{Im}(\Hom(\F, \G) \to A_n(\F,\G) )  \dispdot
$$
which is a finite dimensional vector subspace of $A_n(\F,\G)$ by Lemma~\ref{lem:im_dim_finie}.  For $u\in
\Hom(\F, \G)$, we denote its image in $B_n(\F,\G)$ by
$$
[u]_n \in B_n(\F,\G)  \dispdot
$$
The map $a_n\colon A_n(\F,\G) \to A_{n-1}(\F,\G)$ induces $b_n\colon B_n(\F,\G) \to B_{n-1}(\F,\G)$. Since
$A_n(\F,\G)$ satisfies Mittag-Leffler, the images of the maps $A_r(\F,\G) \to A_n(\F,\G)$, for a given $n$ and
$r\geq n$, stabilize when $r\to\infty$; hence there exists $r(n)$ such that
\begin{equation}\label{eq:rang_stable}
  B_n(\F,\G) = \operatorname{Im}( A_{r(n)}(\F,\G) \to A_n(\F,\G) )  \dispdot
\end{equation}

\begin{lem}\label{lem:compoHom0}
  For $\F, \G, \H \in D_{lc}(N)$ the composition induces a well-defined map, for any $n$,
  \begin{align*}
    \circ_n \colon B_n(\F,\G) \times B_n(\G,\H) &\to B_n(\F,\H) \\
    (\bar u, \bar v) &\mapsto  \bar v \circ_n \bar u := [v \circ u]_n ,
  \end{align*}
  where $u,v$ in $\Hom(\F,\G)$, $\Hom(\G,\H)$ satisfy $[u]_n = \bar u$, $[v]_n = \bar v$. This turns
  $B_n(\F,\F)$ into an algebra which is a finite dimensional quotient of $\Hom(\F,\F)$ and whose unit is
  $\tau_{0,1/n}(\F)$ (assuming $\F\not=0$ and $n$ is big enough so that $\tau_{0,1/n}(\F) \not=0$ -- see Lemma
  \ref{lem:proprietes-gammag}). This also turns $B_n(\F,\G)$ into a finite dimensional quotient of
  $\Hom(\F,\G)$ which is a left $B_n(\F,\F)$-module and right $B_n(\G,\G)$-module.
\end{lem}
\begin{proof}
  It is enough to see that if $[u]_n=0$ or $[v]_n=0$, then $[v\circ u]_n = 0$. This follows from the
  commutative diagram (displayed twice for readability)
  $$
  \xymatrix@C=1cm{
    \F_{U_n} \ar[r]^{u'} \ar[d] & \G_{U_n} \ar[r]^{v'} \ar[d] & \H_{U_n} \ar[d] \\
    \F \ar[r]^u  & \G \ar[r]^v \ar[d]_{\tau_{0,1/n}(\G)} & \H \ar[d]_{\tau_{0,1/n}(\H)} \\
    & \fai{\varphi}^{1/n}(\G) \ar[r]_{ \fai{\varphi}^{1/n}(v)} & \fai{\varphi}^{1/n}(\H))
  }
 \xymatrix@C=6mm{
   \F_{U_n} \ar[r] \ar[ddr]_{[u]_n} \ar[ddrr]^{w}  & \G_{U_n} \ar[r] \ar[ddr]^{[v]_n} & \H_{U_n}  \\
   \\
     & \fai{\varphi}^{1/n}(\G) \ar[r] & \fai{\varphi}^{1/n}(\H))
  }  
  $$
  where $w = [v \circ u]_n$.
\end{proof}

\begin{lem}\label{lem:surj-Hom0prime}
  Let $\F , \G\in D_{lc}(N)$ be such that $\gamma_g(\F,\G) = 0$.  We let $B_n^\times(\F,\F)$ be the subset of
  invertible elements in $B_n(\F,\F)$. We also define
  \begin{multline*}
    C_n = \big\{(u,v) \in B_n(\F,\G) \times B_n(\G,\F) ; 
 v \circ_n u = \tau_{0,1/n}(\F)$, $u \circ_n v = \tau_{0,1/n}(\G) \big\} \dispdot
\end{multline*}
  Then
  \begin{enumerate}
  \item\label{lem:surj-Hom0prime1} for any $n$ the set $C_n$ is non empty, 
  \item\label{lem:surj-Hom0prime2} for any $n$ and any $(u,v) \in C_n$, the maps
    \begin{alignat*}{2}
\alpha_{u,v} \colon  B_n^\times(\F,\F) &\to C_n, &\qquad \beta_{u,v} \colon C_n &\to B_n^\times(\F,\F),  \\
 a &\mapsto (u\circ_n a, a^{-1} \circ_n v) &  (u',v') &\mapsto v\circ_n u'
    \end{alignat*}
are mutually inverse bijections,
  \item\label{lem:surj-Hom0prime3} for $m \geq n$ the natural map  $C_m \to C_n$ is surjective.
  \end{enumerate}
\end{lem}
\begin{proof}
  (\ref{lem:surj-Hom0prime1}) Let $r(n)$ be as in~\eqref{eq:rang_stable} and define $s(n)$ in the same way,
  with $\F$ and $\G$ switched.  We choose $m\geq \max\{r(n), s(n)\}$.  Since $\gamma_g(\F,\G) = 0$, we can
  find $u_1\colon \F \to \fai{\varphi}^{1/2m}(\G)$ and $v_1\colon \G \to \fai{\varphi}^{1/2m}(\F)$ such that
  $u_1 \circ \fai{\varphi}^{1/2m}(v_1)$ and $v_1 \circ \fai{\varphi}^{1/2m}(u_1)$ are the morphisms $\tau$.
  We let $u_2\in A_{2m}(\F,\G)$ be the composition of $\F_{U_{2m}} \to \F$ and $u_1$ and we let $u$ be the
  image of $u_2$ in $A_n(\F,\G)$. Then $u$ belongs to $B_n(\F,\G)$.  We define $v$ in the same way.  Since
  $v_1 \circ u_1 = \tau_{0,1/m}(\F)$ we have $v \circ_n u = \tau_{0,1/n}(\F)$.  Similarly $u \circ_n v =
  \tau_{0,1/n}(\G)$.  Hence $(u,v) \in C_n$.
  
  \smallskip\noindent (\ref{lem:surj-Hom0prime2}) We first remark that $\beta_{u,v}(u',v')$ belongs to $
  B_n^\times(\F,\F)$: indeed the inverse of $v\circ_n u'$ is $v' \circ_n u$.  Now we compute
  $\beta_{u,v}(\alpha_{u,v}(a)) = v \circ_n u \circ_n a = a$ and $\alpha_{u,v}(\beta_{u,v}(u',v')) = (u
  \circ_n v \circ_n u', v'\circ_n u \circ_n v) = (u',v')$.
  
  \smallskip\noindent (\ref{lem:surj-Hom0prime3}) We pick $(u,v) \in C_m$ and let $(\bar u, \bar v)$ be its image by
  the natural map $C_m\to C_n$.  We obtain a commutative diagram
  $$
  \xymatrix{
    B_m^\times(\F,\F) \ar[r]^-{\alpha_{u,v}} \ar[d]_q & C_m \ar[d] \\
    B_n^\times(\F,\F) \ar[r]^-{\alpha_{\bar u,\bar v}} & C_n
  }
  $$
  where the horizontal maps are bijections.  The map $B_m(\F,\F) \to B_n(\F,\F)$ is surjective by definition.
  Hence the map $q$ is surjective by Lemma~\ref{lem:inversible-algebre} and the result follows.
\end{proof}

\begin{lem}\label{lem:inversible-algebre}
  Let $A$, $B$ be finite dimensional algebras over $k$ and let $f\colon A \to B$ be a surjective algebra morphism.
  Then the map induced on the sets of invertible elements $f^\times \colon A^\times \to B^\times$ is also surjective.  
\end{lem}
\begin{proof}
  We pick $x\in A$ and assume that $f(x)$ is invertible. We have to prove that there exists $z\in\ker(f)$ such that
  $x+z$ is invertible.  We let $A' = \langle x \rangle$ be the subalgebra of $A$ generated by $x$ and define
  similarly $B' = \langle f(x) \rangle$.  The maps $u\colon k[X] \to A$, $X \mapsto x$, and $f\circ u$ identify $A'$
  and $B'$ with quotients of $k[X]$.  Since $k[X]$ is principal, we can write $A' = k[X]/\langle P Q \rangle$, with
  $x=[X]_{A'}$, and $B' = k[X]/\langle Q \rangle$, with $f(x) = [X]_{B'}$.

  Since $f(x)$ is invertible in $B$, it is not a zero divisor and we have $X\nmid Q$.  Let us write $P = X^n P'$ with
  $X\nmid P'$.  We set $R = X + P'Q$. Then $f([R]_{A'}) = f(x)$. We claim that $[R]_{A'}$ is invertible, which
  concludes the proof.  We have to check that $\langle R, PQ \rangle = k[X]$.  Let $D$ be a generator of $\langle R,
  PQ \rangle$.  Then $D \mid X^nR - PQ = X^{n+1}$. Either $X\mid D$, but then $X \mid R$ and finally $X\mid P'Q$
  which is false, or $D \in k^\times$ and we are done.
\end{proof}

\begin{prop}\label{prop:unicite-limite}
  We recall that we take coefficients in a field.  Let $\F, \G$ be elements in $D_{lc}(N)$. Assume that that
  $\gamma_g(\F,\G) = 0$. Then $\F \simeq \G$.
\end{prop}
\begin{Question}
  Is the Proposition true if we replace the field by $\mathbb Z$ or a more general ring ?
\end{Question}

\begin{proof}
  If $\F=0$, then $\gamma_g(0,\G) = 0$ and we obtain $\G=0$ by Lemma~\ref{lem:proprietes-gammag}.  We assume
  $\F\not=0$, hence $\tau_{0,1/n}(\F) \not=0$ for $n$ big enough.  We use the notation $C_n$ of
  Lemma~\ref{lem:surj-Hom0prime}.  By this lemma the inductive system $\cdots\to C_n \to C_{n-1} \to\cdots$ is
  made of surjective maps between non empty sets for $n\gg0$. Hence its limit is non empty and we deduce a
  pair $(u,v) \in \Hom(\F,\G) \times \Hom(\F,\G)$ such that $([u]_n, [v]_n) \in C_n$ for all $n$, that is,
  $[v\circ u]_n = \tau_{0,1/n}(\F)$ and $[u\circ v]_n = \tau_{0,1/n}(\G)$.  Let $\C$ be the cone of $u$.  We
  want to prove that $\C\simeq 0$. This can be done locally and we can restrict to some open subset $V$ with
  compact closure.  For any $n$ big enough so that $V \subset U_n$ we have the following morphism of triangles
  given by the morphisms $\tau_{0,1/n}(-)$ (recall that we work on $V$, so $\F_{U_n} \simeq \F$, $\G_{U_n}
  \simeq \G$):
  $$
  \xymatrix{
    \F \ar[r]^u \ar[d] &  \G \ar[r] \ar[d] \ar[dl]_{[v]_n}  &   \C \ar[r] \ar[d]  &  \F[1] \ar[d]  \\
    \fai{\varphi}^{1/n}(\F) \ar[r] &  \fai{\varphi}^{1/n}(\G) \ar[r]
    &   \fai{\varphi}^{1/n}(\C) \ar[r] &  \fai{\varphi}^{1/n}(\F)[1]  ,
  }
  $$
  where both maps $\tau_{0,1/n}(\F)$ and $\tau_{0,1/n}(\G)$ factorize through $[v]_n$.  Using the same argument as in
  the proof of Lemma~\ref{lem:comparaison-distances}-(\ref{lem:comparaison-distances-i}) we can deduce that
    $\tau_{0,2/n}(\C) = 0$.  Hence $\gamma_g(0,\C) \leq 2/n$, for any $n$. Thus the cone of $u$ vanishes and $u$ is
    an isomorphism.
\end{proof}

\subsection{Decomposition in \texorpdfstring{$D_{lc}(\Real)$}{DlcR}} \label{subsection-2.5}

\newcommand{\Mod}{\operatorname{Mod}}
\newcommand{\Mat}{\operatorname{Mat}}
\newcommand{\Intplus}{Int^+}

Let $\Mod(\Real)$ be the category of sheaves of $k$-vector spaces on $\Real$.  We consider it as usual as the full
subcategory of $D(\Real)$ of complexes concentrated in degree $0$.  A sheaf $\F \in \Mod(\Real)$ which is
constructible with respect to a finite stratification of $\Real$ is described by a quiver representation of type
$A_n$.  Hence Gabriel theorem implies that $\F$ is decomposed as $\F \simeq \bigoplus_{I \in \I} k_I$ where $\I$ is a
finite family of intervals.  In this section we extend this result to $\gamma_g$-limits of constructible sheaves
under the assumption that the microsupport is contained in $\{\tau\geq0\}$.

Let $\Intplus$ be the set of intervals of the form $[a,b[$, $a\in \Real$, $b\in \Real \cup\{+\infty\}$.  We endow
$\Intplus$ with the order
\begin{equation}
  \label{eq:ordre-intervalles}
  [a,b\mathclose[ \leq [a',b'[
\qquad  \text{if and only if} \qquad
  a\leq a' \text{ and } b\leq b' \dispdot
\end{equation}
By Lemma~\ref{lem:HomkIkJ}, for $I,I' \in \Intplus$ we have $\Hom(k_I , k_{I'}) \simeq k$ canonically if
and only if $I\cap I' \not= \emptyset$ and $I \leq I'$ and in this case we denote by $e_I^{I'}$ its generator.  Otherwise
$\Hom(k_I, k_{I'}) \simeq 0$ and we set $e_I^{I'}=0$. 
For $I \leq I' \leq I''$ we have $e_{I'}^{I''} \circ e_I^{I'} = e_I^{I''}$.  For two finite families $\I$, $\I'
\subset \Intplus$ and $\G = \bigoplus_{I \in \I} k_I$, $\G' = \bigoplus_{J \in \I'} k_J$, we obtain a natural morphism
$$
E = E_{\I}^{\I'} \colon \Mat_{lt}(\I' \times \I) \to \Hom(\G,\G'), \quad
(a_{I',I})_{I'\in \I', I\in \I} \mapsto \sum a_{I',I} e_I^{I'},
$$
where $\Mat_{lt}$ denotes the space of lower triangular matrices ($a_{I',I}\not=0$
implies $I \leq I'$).  The map $E$ is surjective and, for three families $\I$, $\I'$,
$\I''$, we have $E_{\I'}^{\I''}(B) \circ E_{\I}^{\I'}(A) = E_{\I}^{\I''}(BA)$.

\begin{lem}\label{lem:forme-canonique-morphisme}
  Let $\I$, $\I' \subset \Intplus$ be two finite families and $\G = \bigoplus_{I \in \I} k_I$, $\G' = \bigoplus_{J
    \in \I'} k_J$.  Let $\varepsilon>0$ be given and let $u \colon \G \to \G'$, $v \colon \G' \to
  T_{\varepsilon*}(\G)$ be such that $v\circ u = \tau_{0,\varepsilon}(\G)$.  We assume that each $I\in \I$ is of
  length $> \varepsilon$.  Then there exist an isomorphism $\phi\colon \G' \isoto \G'$ and an injective map $\sigma
  \colon \I \to \I'$ such that $I \leq \sigma(I)$ for all $I\in \I$ (for the order $\leq$ defined
  in~\eqref{eq:ordre-intervalles}) and $\phi \circ u \colon \G \to \G'$ is the sum of the natural morphisms $k_I \to
  k_{\sigma(I)}$.
\end{lem}
\begin{proof}
  We set $\I'' = \{T_\varepsilon(I)$; $I\in \I\}$.  In view of the discussion
  after~\eqref{eq:ordre-intervalles} we can find lower triangular matrices
  $A \in \Mat_{lt}(\I' \times \I)$, $B \in \Mat_{lt}(\I'' \times \I')$ representing
  $u$, $v$.  Then $BA \in \Mat_{lt}(\I'' \times \I)$ is lower triangular.  We
  identify $\I$ and $\I''$ through $T_\varepsilon$ and we claim that
  $BA \in \Mat_{lt}(\I \times \I)$ is still lower triangular.  We take
  $I_1, I_2 \in \I$ and assume that
  $0\not= (BA)_{I_1,I_2} = \sum_{J\in \I'} B_{T_\varepsilon(I_1),J} A_{J,I_2}$.
  There exists $J$ such that $B_{T_\varepsilon(I_1),J}, A_{J,I_2} \not= 0$, hence
  $I_2 \leq J \leq T_\varepsilon(I_1)$.  If we assume moreover that
  $I_2 \not\leq I_1$ we must have $T_\varepsilon(I_1) \cap I_2 \not=\emptyset$
  because $I_1$ and $I_2$ are of length $>\varepsilon$.  Hence
  $e_{I_2}^{T_\varepsilon(I_1)} \not=0$ and we deduce that the decomposition of
  $v\circ u$ on the basis of the non zero $e_I^{I''}$'s has a non zero coefficient on
  $e_{I_2}^{T_\varepsilon(I_1)}$. Since $v\circ u = \tau_{0,\varepsilon}(\G)$ this
  means that $I_1 = I_2$, contradicting $I_2 \not\leq I_1$.  Hence $I_2 \leq I_1$,
  as claimed.

  Finally $BA$ is lower triangular (in $\Mat_{lt}(\I \times \I)$) and has entries $1$
  along the diagonal since it represents $\tau_{0,\varepsilon}(\G)$. Hence $BA$ is
  invertible.  In particular $|\I| \leq |\I'|$, $A$ has rank $|\I|$, we can decompose
  $\I'$ as $\I' = \I \sqcup \I'_1$ and find an invertible matrix
  $A' \in \Mat_{lt}(\I' \times \I')$ such that the block $\I\times \I$ in $A'A$ is
  the identity matrix. Now we let $\sigma$ be the map given by the partition
  $\I' = \I \sqcup \I'_1$ and $\phi$ be the morphism induced by $A'$.
\end{proof}

\begin{prop}\label{prop:decomp_limite}
  Let $\F \in \Mod(\Real)$ such that $SS(\F) \subset \{(t;\tau) \in T^*\Real;$ $\tau\geq 0\}$.  We assume that
  $\F$ is a $\gamma_\tau$-limit of constructible (for finite stratifications) sheaves $\F_n \in \Mod(\Real)$
  such that $SS(\F_n) \subset \{(t;\tau) \in T^*\Real;$ $\tau\geq 0\}$.  Then there exists an at most
  countable set of intervals $\I$ such that $\F \simeq \bigoplus_{I \in \I} k_I$.
\end{prop}
We notice that the singular support hypothesis in the proposition implies that our
intervals are of the type $[a,b[$, $[a,\infty[$ or $\mathopen]-\infty,b[$.
\begin{proof}
  (i) Up to taking a subsequence we can assume $\gamma_\tau(\F,\F_n) \leq \varepsilon_n = 2^{-n}$.  For each $n$ there
  exists a finite family $\I_n \subset \Intplus$ such that $\F_n := \bigoplus_{I \in \I_n} k_I$.  By
  Proposition~\ref{prop:gammalimit=colimit} we can assume, up to translating $\F_n$ by less than $\varepsilon_n$,
  that there exist compatible morphisms $f_n\colon \F_n \to \F_{n+1}$ and that $\hocolim \F_n \isoto \F$.  Since
  $\F_n \in \Mod(\Real)$, we have $\hocolim \F_n \simeq \varinjlim \F_n$.  We can also assume (maybe changing
  $\varepsilon_n$ by $4\varepsilon_n$) that there exist morphisms $g_n \colon \F_{n+1} \to T_{\varepsilon_n *}(\F_n)$
  such that $g_n \circ f_n = \tau_{0,\varepsilon_n}(\F_n)$.
  
  \medskip\noindent(ii) It remains to check that $\varinjlim \F_n$ is decomposed.  For this we modify the $f_n$'s
  using Lemma~\ref{lem:forme-canonique-morphisme}.  We let $\J_n \subset \I_n$ be the set of intervals of length $>
  2\varepsilon_n$ and we set $\widetilde \F_n = \bigoplus_{I \in \J_n} k_I$.  We first put the $f_n$'s in ``diagonal
  form'', $\widehat f_n$ (see the diagram~\eqref{eq:diag_decomp_limite}).  Lemma~\ref{lem:forme-canonique-morphisme}
  applied with $f_0|_{\widetilde\F_0} \colon \widetilde\F_0 \to \F_1$ gives an injective map $\sigma_0 \colon \J_0
  \to \I_1$ and an isomorphism $\phi_1 \colon \F_1 \isoto \F_1$ such that $\phi_1 \circ f_0|_{\widetilde\F_0}$ has a
  ``diagonal form'' (that is, it is the sum of the natural morphisms $k_I \to k_{\sigma_0(I)}$).  We set $\widehat
  f_0 = \phi_1 \circ f_0$ and $f'_1 = f_1 \circ \phi_1^{-1}$.  Now we apply the lemma with $f'_1|_{\widetilde\F_1}
  \colon \F_1 \to \F_2$ and obtain $\phi_2 \colon \F_2 \isoto \F_2$ and $\sigma_1 \colon \I_1 \to \I_2$. We set
  $\widehat f_1 = \phi_2 \circ f'_1$ and $f'_2 = f_2 \circ \phi_2^{-1}$. We go on inductively and end up with a
  sequence of injective maps $\sigma_n \colon \J_n \to \I_{n+1}$ and morphisms $\widehat f_n\colon \F_n \to \F_{n+1}$
  such that $\widehat f_n$ restricted to $\widetilde \F_n$ has a ``diagonal form''.
  \begin{figure}[ht]
  \begin{equation}\label{eq:diag_decomp_limite}
    \begin{split}
  \xymatrix{
    \widetilde \F_0 \ar@{=}[r]  \ar[d]_{\widetilde f_0} & \bigoplus_{I_0 \in \J_0} k_{I_0} \ar[r] \ar[dr]
    & \bigoplus_{I_0 \in \I_0} k_{I_0} \ar@{=}[r] \ar[d]^{\widehat f_0} & \F_0 \ar[d]^{f_0} \\
    \widetilde \F_1 \ar@{=}[r]  \ar[d]_{\widetilde f_1} & \bigoplus_{I_1 \in \J_1} k_{I_1} \ar[r] \ar[dr]
    & \bigoplus_{I_1 \in \I_1} k_{I_1}  \ar@{-}[r]_-\sim^-{\phi_1} \ar[d]^{\widehat f_1} & \F_1 \ar[d]^{f_1} \\
    \widetilde \F_2  \ar@{=}[r]  \ar[d]_{\widetilde f_2} & \bigoplus_{I_2 \in \J_2} k_{I_2} \ar[r] \ar[dr]
    & \bigoplus_{I_2 \in \I_2} k_{I_2}  \ar@{-}[r]_-\sim^-{\phi_2} \ar[d]^{\widehat f_2} & \F_2 \ar[d]^{f_2} \\
    \widetilde \F_3  \ar@{=}[r] & \bigoplus_{I_3 \in \J_3} k_{I_3} \ar[r]
    & \bigoplus_{I_3 \in \I_3} k_{I_3} \ar@{-}[r]_-\sim^-{\phi_3}  & \F_3  \\
  }
\end{split}
\end{equation}
\end{figure}
  
\medskip\noindent(iii) With our definition of $\J_n$ we have moreover $\sigma_n(\J_n) \subset \J_{n+1}$.  Indeed,
modifying $g_n$ as $\widehat g_n = T_{\varepsilon_n *}(\phi_n) \circ g_n \circ \phi_{n+1}^{-1}$, we see that
$\widehat g_n \circ \widehat f_n = \tau_{0,\varepsilon_n}(\F_n)$.  It follows that, for any $[a,b\mathclose[ \in
\J_n$ and $[a',b'\mathclose[ = \sigma_n([a,b[)$, we must have $a\leq a' \leq a+\varepsilon_n$ and $b\leq b' \leq
b+\varepsilon_n$.  Hence $b'-a' \geq b-a-\varepsilon_n \geq 2\varepsilon_n - \varepsilon_n = 2 \varepsilon_{n+1}$,
which shows that $[a',b'\mathclose[ \in \J_{n+1}$.

  We can then define $\widetilde f_n = \widehat f_n|_{\widetilde \F_n} \colon \widetilde \F_n \to \widetilde
  \F_{n+1}$ and $\widetilde f_n$ has a ``diagonal form'' with respect to the inclusion $\sigma_n|_{\J_n} \colon \J_n
  \hookrightarrow \J_{n+1}$.  For a given $I \in \J_n$ we write $\sigma_k \circ \sigma_{k-1} \circ\dots\circ
  \sigma_n(I) = [a_k,b_k[$; the sequences $(a_k)$, $(b_k)$ are non decreasing and have limits, say $a_\infty$,
  $b_\infty$.  Then $\varinjlim k_{[a_k,b_k[} \simeq k_{[a_\infty, b_\infty[}$. We set $I_\infty = [a_\infty,
  b_\infty[$.  We let $\J$ be the increasing union of the $\J_n$'s. Then each element of $\J$ corresponds to an
  interval $I_\infty$ as above and we have proved $\varinjlim \widetilde \F_n \simeq \bigoplus_{I_\infty \in \J}
  k_{I_\infty}$.  By Lemma~\ref{lem:gammalim=hocolim0} the maps $\widetilde \F_k \to \varinjlim \widetilde \F_n$ and
  $\F_k \to \varinjlim \F_n$ have cones which go to $0$ when $k\to \infty$. It follows by
  Lemma~\ref{lem:proprietes-gammag} that the natural map $\varinjlim \widetilde \F_n \to \varinjlim \F_n$ has a cone
  which is arbitrarily close to $0$, hence it is an isomorphism.
\end{proof}

\begin{cor}\label{cor:decomp_limite}
  We recall that $k$ is a field.  Let $\F \in D_{lc}(\Real) \cap D_{\tau\geq0}(\Real)$ be a limit of
  constructible (for finite stratifications) objects of $D_{\tau\geq0}(\Real)$.  Then there exists an at most
  countable set of intervals $\I$ and integers $d_I$, $I\in \I$, such that $\F \simeq \bigoplus_{I \in \I}
  k_I[d_I]$.
\end{cor}
\begin{proof}
  As in the case of constructible sheaves the result follows from Pro\-position~\ref{prop:decomp_limite} and a
  decomposition of a complex as sum of its cohomology. However we didn't prove that limits of constructible
  sheaves have no $\operatorname{Ext}^2$ so we have to be careful to decompose the cohomology before we take
  the limit, as follows.
  
  There exists a sequence of constructible objects $\F_n \in D_{\tau\geq0}(\Real)$ which $\gamma_\tau$-converges to $\F$.
  Since $\operatorname{Ext}^2(\G,\G') = 0$ for any two constructible $\G,\G' \in \Mod(\Real)$, we have $\F_n \simeq
  \bigoplus_{i \in \Z} H^i\F_n[-i]$ (see~ \cite[Cor.~13.1.20]{K-S06} for example).

  In $D_{\tau\geq0}(\Real)$ we have $\fai{\varphi}^a(-) = T_{a*}(-)$, where $T_a$ is as usual the translation by $a$,
  and this functor commutes with the cohomology, that is, $H^i(T_{a*}(\G)) \simeq T_{a*}(H^i(\G))$.  It follows that
  $\gamma_\tau(H^i\G, H^i\G') \leq \gamma_\tau(\G,\G')$, for any $\G, \G' \in D_{\tau\geq0}(\Real)$.  Hence $H^i\F_n$
  $\gamma_\tau$-converges to $H^i\F$.

  It follows that $\F_n \simeq \bigoplus_i H^i\F_n[-i]$ $\gamma_\tau$-converges to $\bigoplus_i H^i\F[-i]$.  By uniqueness
  of the limit (Proposition~\ref{prop:unicite-limite}) we have $\F \simeq \bigoplus_i H^i\F[-i]$. Now the result
  follows from Proposition~\ref{prop:decomp_limite}.
\end{proof}

\section{Comparing different interleaving metrics on \texorpdfstring{$D (N\times {\mathbb R})$}{D(NtimesR)}.}

To prove Theorem \ref{Thm-sheaf-gamma} we have used the distance $\gamma_\tau$ on sheaves over $N\times\Real$ and the
fact that its restriction to sheaves associated with Hamiltonian maps coincides with the spectral distance (see
Proposition~\ref{prop:gamma-et-gammag2}).  We can also consider the distance $\gamma_g$ on sheaves, where $g$ is some
Riemannian metric on $N$.  For the sake of completeness we describe the distance $\gamma_g$ from the point of view of
spectral distance and we prove that the distances $\gamma_\tau$ and $\gamma_g$ induce equivalent topologies (when
restricted to sheaves associated with Hamiltonian maps).

In the following we let $H \colon T^*N \to \Real$ be a function which coincides with $(q,p) \mapsto ||p||_g$ outside
a compact neighbourhood of $0_N$ and $\varphi = \varphi_H$ denotes its Hamiltonian flow.  We will use spectral
invariants and sheaves for $\varphi^s$ but in~\S\ref{sec:qhi} and \S\ref {sec:specinv} we only defined them for
compactly supported isotopies.  However we can extend them to a positive Hamiltonian $H$ as follows. Let $H_r$ be a
truncation of $H$ to the disc bundle of $T^*N$ of radius $r$, so that $H_r$ is an increasing sequence of Hamiltonians
converging to $H$.  Then the homogeneous lift of $s \mapsto \varphi_{H_r}^{-s} \circ \varphi_{H_{r+1}}^{s}$ to
$T^*(N\times\Real) \setminus 0_{N\times\Real}$ is a non negative Hamiltonian isotopy. Hence we have a natural map
$\K_{\id} \to \K_{\varphi_{H_r}^{-s}} \circ \K_{\varphi_{H_{r+1}}^{s}}$, or equivalently, $\K_{\varphi_{H_r}^{s}} \to
\K_{\varphi_{H_{r+1}}^{s}}$.  We define $\K_{\varphi_s}$ as the homotopy colimit of this inductive system (say $r$
runs over the integers).  In the same way $c_+(\varphi_r^{-s}\psi)$ is a decreasing sequence, and we say that
$c_+(\varphi^{-s}\psi)=0$ if $\lim_r c_+(\varphi_r^{-s}\psi)=0$. We shall prove that the sequence is in fact
stationary, so this means $c_+(\varphi_r^{-s}\psi)=0$ for $r$ large enough.

Now Lemma \ref{Lemma-5.5} extends as follows: there exist morphisms $\K_\psi \longrightarrow \K_{\varphi^s}
\longrightarrow \K_{\varphi^{2s}\circ \psi}$ if and only if $c_-(\varphi^s\psi)=0$ and $c_-(\varphi^s\psi^{-1})=0$.
Using Corollary \ref{cor:restriction-a-infini} we obtain
\begin{lem} 
  If $c_-(\varphi^s\psi)=0$ and $c_-(\varphi^s\psi^{-1})=0$, the composition of the map of the map $\K_\psi
  \longrightarrow \K_{\varphi^s} $ and $\K_{\varphi^s} \longrightarrow \K_{\varphi^{2s}\circ \psi} $ coincides with
  the canonical map $\K_\psi \longrightarrow \K_{\varphi^{2s}\psi} $.
\end{lem}

Now we may define the distance $\gamma^g$ on $\Ham(T^*N)$ as follows
\begin{defn} 
We set $$c_-^g(\psi)= -\inf \left\{s\geq 0 \mid c_-(\varphi^s\psi)=0\right \}$$
and $$c_+^g(\psi)=-c_-^g(\psi^{-1})=\inf \left\{s\geq 0 \mid c_+(\varphi^{-s}\psi)=0\right \}$$ Finally we set
$$\gamma^g(\psi)=c_+^g(\psi)-c_-^g(\psi)$$
\end{defn} 

\begin{prop} We have
\begin{enumerate} 
\item The function $\gamma^g$ defines a metric by $\gamma^g(\psi_1,\psi_2)=\gamma^g(\psi_1\psi_2^{-1})$. 
\item The topologies defined by $\gamma^g$ and $\gamma$ coincide.
\end{enumerate} 
\end{prop} 
\begin{proof} 
(1) Apply section 3 of \cite{Petit-Schapira} to the Kernels  $\K_{\varphi^s}$ (or the symplectic flow $\varphi^s$).

\smallskip\noindent (2)
We shall first prove that $ c_+(\psi) \leq C_{W} c_+^g(\psi) $ and $\gamma(\psi) \leq C_W \gamma^g(\psi)$ for all $\psi$ such that $\supp (\psi) \subset K$. 
We consider $c_+(\varphi^{-s}\psi)$ and remember that $\varphi^s$ is generated by a positive homogeneous Hamiltonian $H$. 
By a classical argument (see e.g. \cite{Viterbo-STAGGF}, lemma 4.7 p. 699) the map  $s \mapsto c_+(\varphi^{-s}\psi)$ is piecewise $C^1$ and where it is $C^1$ we have
$$ \frac{d}{ds} c_+(\varphi^{-s}\psi)_{\mid s=s_0} = -K_s(z)$$ where $K_s$ is the Hamiltonian associated to the map $\varphi^{-s}\psi$ and $z$ is some fixed point of 
$\varphi^{-s}\psi$ having non-negative action. Indeed, if $S_s$ is a Generating function Quadratic at infinity for $\varphi^{-s}\psi$ we know that $c_+(\varphi^{-s}\psi) $ is a critical value of 
$S_s$ and is non-negative. It thus corresponds to fixed points of $\varphi^{-s}\psi$ with non-negative action and its derivative is 
$ \frac{d}{ds} S_s(\xi_z)$ where $\xi_z$ corresponds to $z$, one of the fixed points such that $S_s(\xi_z)=c_+(\varphi^{-s}\psi)$. 

In our case, if $z$ is outside the support of $\psi$, the fixed point is a fixed point of $\varphi^{-s}$ and these have action $0$ since $H$ is $1$-homogeneous (as on an orbit
$p\dot q-H(q,p)=p \frac{\partial H}{\partial p}- H(q,p)=0$). So we only need to estimate $H$ in the support of $\psi$ and we set $W$ to be a neighbourhood of this support. Then if $C_W=\sup \{ H(x,p) \mid (x,p)\in W\}$ we have  
$$ \frac{d}{ds} c_+(\varphi^{-s}\psi)_{\mid s=s_0}\geq -C_W$$
so that $$c_+(\varphi^{-s}\psi)-c_+(\psi)\geq -C_Ws$$ and if $c_+(\varphi^{-s}\psi)=0$ we have
$$C_Ws \geq c_+(\psi)$$ hence 
$$c_+(\psi) \leq C_W c^g(\psi)$$
Applying the same to $\psi^{-1}$, we get 
$$\gamma(\psi) \leq C_{W}\gamma^{g}(\psi)$$
Note that we also proved that $\gamma (\varphi_{r}^{-s}\psi)$ is stationary for $r$ such that $B(0,r)$ contains the support of $\psi$. 

Conversely, let us write $\varphi^{s}=  \xi_{r}^{s}\circ \eta_r^s$ where $\xi_r^s$ (resp. $\eta_r^s$)  is the flow of $f_r(H)$ (resp. $(1-f_r)(H)$) where $f_r$ is a function given by 
\begin{enumerate} 
\item $f_r(t)=t$ for $t\geq 2r$
\item $f_r= \frac{3r}{2}$ for $t\leq r$ 
\item $f_r$ is non-decreasing 
\end{enumerate} 
 Note that the flows $\xi_r^s, \eta_r^s$ are autonomous, commute, and $f_r(H)$ is bounded from below by $ \frac{3r}{2}$ and $\eta_r^s$ is supported in a ball of radius $K\cdot r$.  
 
 Then 
$$c_{+}(\varphi^{-s}\psi)=  c_{+}(\xi^{-s}\eta_{r}^{-s}\psi) \leq c_{+}(\eta_{r}^{-s}\psi)-\frac{3r}{2} s$$ as long as $c_{+}(\varphi^{-s}\psi)>0$
so we have 
$$\frac{3r}{2} c_{+}^g(\psi) \leq c_{+}(\eta_{r}^{-s}\psi) \leq \pi K^2\cdot r^2+ c_+(\psi)$$
As a result if $c_+(\psi_j)$ converges to zero,  let $r_j=\sqrt{c_{+}(\psi_{j)}}$ . Then 
$$c_{+}^g(\psi_j) \leq \frac{2}{3} \left (\pi K\cdot r_j + \frac{c_+(\psi_j)}{r_j}\right ) \leq \frac{2}{3}(\pi K+1)  \sqrt{c_{+}(\psi_{j})}$$
and we deduce that $\lim_{j}c_{+}^g(\psi_j)=0$. 
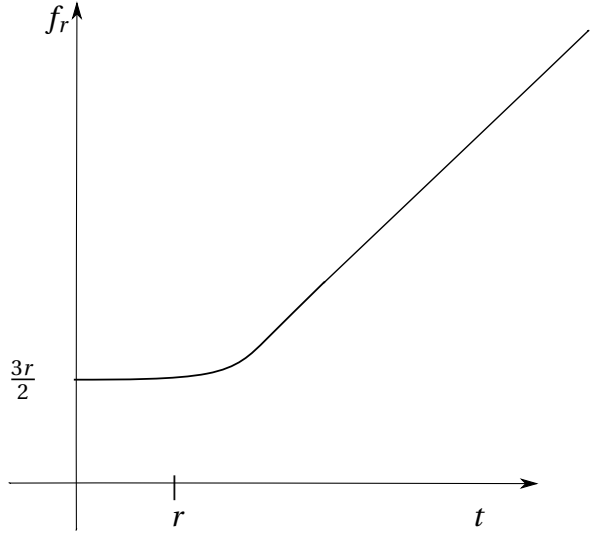
\begin{figure}[ht]
 \begin{overpic}[width=8cm]{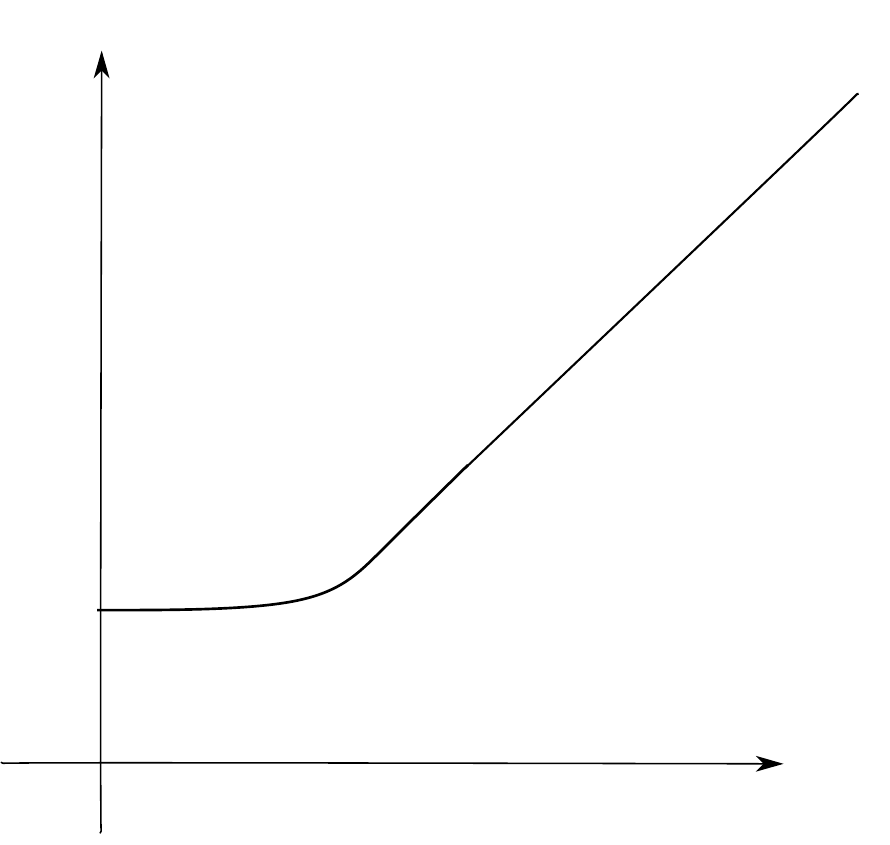}
  \put (6,85) {$f_{r}$} 
   \put (0,25) {$ \frac{3r}{2} $} 
   \put (77,2) {$t$} 
   \put(27,7) {$\mid$}
    \put (27,2) {$r$} 
 \end{overpic}
\caption{The function $f_{r}$}
\label{Fig-f_{r}}
\end{figure}
\end{proof}

\begin{rem} \begin{enumerate} 
\item 
We do not know if the metrics $\gamma$ and $\gamma^g$ are equivalent, only that they define the same topology. 
\item  Here is an analogy for the difference between $\gamma$ and $\gamma^g$. 
Consider the following norm on $C_b^0(X, {\mathbb R} )$ : we first define $$\max_f (g)=\inf\left\{ s\geq 0 \mid \max (g-sf)\leq 0\right \}$$ and 
$$\min_f(g)-\max_f(-g)=\inf\left\{ s\geq 0 \mid \min (g+sf)\geq 0\right \}$$ Then we set $ \Vert g \Vert _f= \max_f(g)-\min_f(g)$. Provided 
$$0<\inf_{x\in X} f(x) \leq \sup_{x\in X} f(x) < +\infty$$ it is obvious that  $ \Vert \bullet \Vert_f$ and $ \Vert \bullet \Vert $ are equivalent. 

\end{enumerate}  
\end{rem}

\section{Local Floer cohomology}\label{Appendix-local-Floer}
Our goal is to prove the following statement, which was used in Subsection \ref{Subsection-Floer-Darboux-chart}. 

Let $\varphi$ be a Hamiltonian map generated by a Hamiltonian supported in a sufficiently small  Darboux chart, $V$. Since $\Gamma(\varphi)$ coincides with the diagonal, $\Gamma(\Id)$ outside
$W=V\times V$, we can  consider two types of spectral numbers. 
%


Consider $\varphi^{t}$ as a compact supported Hamiltonian isotopy defined on $ {\mathbb R}^{2n}$ and denote by $\widetilde\varphi^{t}$ this isotopy.
 Let $\widetilde S(q,P,\xi)$ be a Generating Function Quadratic at Infinity for $\Gamma(\widetilde\varphi)$ that we compactify as a Lagrangian in $T^{*}S^{2n}$. In the sequel we will shorten Generating Function Quadratic at Infinity  to \GFQI and refer to \cite{Viterbo-STAGGF} for their properties. 
 
 Let $S : \Delta_{M}\times {\mathbb R}^{k} \longrightarrow {\mathbb R} $ be a \GFQI for $\Gamma (\varphi)$ over $\Delta_{M}$. 
 Then we want to prove that  $$c(1_{S^{2n}}, \widetilde S)= c(1_{\Delta}, S)$$
 Since by \cite{Viterbo-FCFH2} $c(1_{\Delta},S)$ corresponds to the $c_{-}(\varphi)$ obtained by using Floer cohomology  (note that this is also a consequence of Theorem \ref{Thm-Quantization}, (\ref{3})), this will imply 
 $$c_{-}(\widetilde \varphi^{1})=c(1_{S^{2n}}, \widetilde S)=c(1_{\Delta},S)=c_{-}( \varphi^{1})$$

We are thus in a special case of the more general situation where $L$ is Hamiltonianly isotopic to the zero section $0_{N}$ by a Hamiltonian isotopy supported in $W$

 \begin{defn} 
 Let $S(q,P,\xi)$ be a \GFQI  for $\Gamma(\varphi)$  over a bounded connected codimension $0$ submanifold $W$. We set $c^{W}_{-}(\varphi)=c(1_{W}, S)$ where $1_{W}\in H^{0}(W)$. the spectral number $c(1_{W}, S)$ was defined  in \cite{Viterbo-stochastic}, Definition 5.1  and $c(\mu_{W},S)=c_{+}^{W}(L)$ where $\mu_{W}\in H^{n}(W, \partial W)$ is the fundamental class.  They are obtained by setting 
 \begin{enumerate} 
 \item  For $\alpha \in H^*(W)$ $$c(\alpha, S)= \inf \left \{ t \mid T\otimes \alpha \neq 0 \;\text{in}\; H^*(S_{\mid W}^t, S_{\mid W}^{-\infty}) \right\}$$
\item For $\alpha \in H_c^*(W)=H^*(W,\partial W)$  $$c(\alpha, S)= \inf \left \{ t \mid T\otimes \alpha \neq 0 \;\text{in}\; H^*(S_{\mid W}^t, S_{\mid W}^{-\infty}\cup S_{\mid \partial W}^t)\right \}$$
\end{enumerate} 

\end{defn} 

 \begin{prop} Let $W$ be a codimension zero connected submanifold of the closed manifold $N$. 
 Let $\varphi^t$ be an isotopy supported in $T^{*}W \subset T^{*}N$. Let $S(q;\xi)$ be a \GFQI for $L=\varphi^{1}(0_{N})$ normalized so that outside $W$ the critical points of $S$ have critical value $0$. 
 Set
 $c_{-}^{W}(L)=c(1_{W}, S)$. Then we have $c_{-}^{W}(L)=c_{-}(L)$ and by duality $c_{+}^{W}(L)=c_{+}(L)$. 
 \end{prop} 
 \begin{proof} 
 We prove the first equality. Since $L$ coincides with the zero section outside of $W$, we know that $c_{-}(L)\leq 0$. Let $\psi^{t}$ be flow of a fiber preserving vector field on $S^{- \varepsilon } $ where $ \varepsilon $ is some small positive real number. 
Let $\psi^{t}$ be a flow on $N \times {\mathbb R}^{k}$ such that
 \begin{enumerate} 
 \item $\psi^{t}$ is the flow of $ - \frac{ \nabla_{\xi}S(x,\xi)}{ \left \vert  \nabla_{\xi}S(x,\xi) \right \vert^{2}}$ outside of $W\times {\mathbb R}^{k}$
  \item $S(\psi^{t}(x,\xi))\leq S(x,\xi)$ for $t\geq 0$
 \end{enumerate} 
 This is easy to construct by noticing that in  $(N\setminus W)\times {\mathbb R}^{k}\cap S^{- \varepsilon }$ the vector field $ - \frac{ \nabla_{\xi}S(x,\xi)}{ \left \vert  \nabla_{\xi}S(x,\xi) \right \vert^{2}}$ is well defined, since $S$ has no fiberwise critical points in $(N\setminus W)\times {\mathbb R}^{k}$. Note that $\psi^{t}$ is then obtained by multiplying this vector field by a smooth function equal to $1$ on $(N\setminus W)\times {\mathbb R}^{k}\cap S^{- \varepsilon }$ and vanishing outside a neighbourhood of it. 
 Remember that we denote by $S^{-\infty}$ the set $S^{-c}$ for some large $c$. Then for $t< - \varepsilon $ we have 
 $$H^{*}(S^{t}, S^{-\infty})=H^{*}(\psi^{s}(S^{t}) , \psi^{s}(S^{-\infty}))$$ but since for $s$  large enough 
$$\psi^{s}(S^{t}) \subset S^{t}_{\mid W\times {\mathbb R}^{k}}\cup S^{-\infty}$$

we have by excision,  $$ H^{*}(\psi^{s}(S^{t}), \psi^{s}(S^{-\infty}))\simeq H^{*}(\psi^{s}(S^{t}_{\mid W\times {\mathbb R}^{k}}), \psi^{s}(S^{-\infty}_{\mid W\times {\mathbb R}^{k}}))$$ and applying the isotopy $\psi^{-s}$ we get that for $t<- \varepsilon $
 $$H^{*}(S^{t}, S^{-\infty}) \simeq H^{*}(S^{t}_{\mid W\times {\mathbb R}^{k}}, S^{-\infty}_{\mid W\times {\mathbb R}^{k}})$$
 We therefore have the following diagram for $t< - \varepsilon  $
 $$ \xymatrix{
 H^{*-d}(N) \ar[d]^{\simeq} \ar[r] & H^{*-d}(W) \ar[d]^{\simeq} \\
 H^{*}(S^{\infty}, S^{-\infty})  \ar[r]\ar[d] & H^{*}(S^{\infty}_{\mid W\times {\mathbb R}^{k}}, S^{-\infty}_{\mid W\times {\mathbb R}^{k}})\ar[d]\\
 H^{*}(S^{t}, S^{-\infty}) \ar[r]^{\simeq} & H^{*}(S^{t}_{\mid W\times {\mathbb R}^{k}}, S^{-\infty}_{\mid W\times {\mathbb R}^{k}})
 }$$
 Since the top horizontal map sends  $1_{N}$  to $1_{W}$, we have that (when $t< - \varepsilon $) the composition of the two vertical left-hand side maps sends $1_{N}$ to zero  if and only if the
  composition of the two vertical left-hand side maps sends $1_{W}$ to zero. In other words if $t< - \varepsilon $ and $t< c(1_{N},S)$ then $t\leq  c(1_{W},S)$   so that 
 $\min\{c(1_{N},S), - \varepsilon \} \leq c(1_{W},S)$ and conversely if   $t< - \varepsilon $ and  $t< c(1_{W},S)$ then $t\leq c(1_{N},S)$ in other words $\min\{c(1_{W},S), - \varepsilon \} \leq c(1_{N},S)$. It is however well-known that $c(1_{W},S)\leq 0$ and $c(1_{N},S) \leq 0$ and since $ \varepsilon >0$ is arbitrarily small,  we get $c(1_{W},S)=c(1_{N},S)$. 
 \end{proof} 
 From this it immediately follows 
 \begin{cor} 
 Let $U$ be a codimension $0$ submanifold of the closed manifolds $N_{1}, N_{2}$  and $L_{1}=\varphi_{1}^{1}(0_{N_{1}})$, $L_{2}=\varphi_{2}^{1}(0_{N_{2}})$ where $\varphi_{1}^{t}=\Id$ and $\varphi_{2}^{t}=\Id$ outside of $T^{*}U$ and $\varphi_{1}^{t}=\varphi_{2}^{t}$ on $T^{*}U$. Then $$c_{\pm}^{1}(L_{1})=c_{\pm}^{2}(L_{2})$$ In other words $c_{\pm}$ does not depend on the extension of $\varphi_{1}^{1}(0_{U})=\varphi_{2}^{1}(0_{U})$. 
 \end{cor} 
 
 Applying this to $\Gamma(\varphi)$ we get, provided the small Darboux chart is such that $V\times V$ is contained in a Weinstein neighbourhood of the diagonal (i.e. a neighbourhood of the diagonal that is symplectomorphic to a neighbourhood of the zero section in $T^{*}\Delta$). Using the remark at the beginning of this Appendix, we get 
 \begin{cor}Let $\varphi^{t}$ be a Hamiltonian isotopy of $(M,\omega)$ supported in a small Darboux chart and consider $\varphi^{t}$ as a compact supported Hamiltonian isotopy defined on $ {\mathbb R}^{2n}$. Denote by $\widetilde\varphi^{t}$ this isotopy and 
  $\widetilde S(q,P,\xi)$ be a \GFQI for $\Gamma(\widetilde\varphi)$ that we compactify as a Lagrangian in $T^{*}S^{2n}$. 
 Let $S : \Delta_{M}\times {\mathbb R}^{k} \longrightarrow {\mathbb R} $ be a \GFQI for $\Gamma (\varphi)$ over $\Delta_{M}$. Then 
 $$c_{-}(\widetilde \varphi^{1})=c(1_{S^{2n}}, \widetilde S)=c(1_{\Delta_M},S)=c_{-}( \varphi^{1})$$
 We similarly get $$c_{+}(\widetilde \varphi^{1})=c_{+}( \varphi^{1})$$
 \end{cor} 
\begin{color}{black}
\section{ Duality for sheaf spectral invariants and connection with  Floer's spectral theory}
\newcommand{\DD}{\operatorname{D}}

Our goal in this Appendix is twofold. On one hand we look for a replacement of the duality yielding the equality $c_-(\varphi)=-c_+(\varphi^{-1})$ in the sheaf framework. As will be clear the duality involves spectral invariants associated to compact supported cohomology classes. On the other hand we shall prove in general that the Floer cohomology of a symplectic map is isomorphic to the corresponding sheaf cohomology. In particular this implies that the spectral invariants coincide, which we only proved for symplectic maps supported in a Darboux ball in Subsection \ref{Subsection-5.2}.

Recall that the Tamarkin category $\cD(N)$ is the left orthogonal of $\SD_{\{\tau\leq 0\}}(k_{N\times\Real})$
and the convolution $k_{N\times [0,+\infty[} \cstar - \colon \SD(k_{N \times \Real}) \to \SD(k_{N \times
  \Real})$ is a projector onto $\cD(N)$.  We can define a duality in $\cD(N)$ as follows.  We first recall the
usual duality functors, for a manifold $M$, $\DD'(F) = \rhom(F,k_M)$ and $\DD(F) = \rhom(F,\omega_M)$, where
$\omega_M \simeq or_M[d_M]$ is the dualizing sheaf.  We define $i \colon N\times \Real \to N\times\Real$ by
$i(x,t) = (x,-t)$.  For $F \in \cD(N)$, we set
$$
\DD^T(F) = k_{N\times [0,+\infty[} \cstar i^{-1} \DD'(F)
$$
For $L \in \LL(T^*M)$ and $F_L = Q(L)$ we then have $\DD^T(F_L) \simeq F_{-L}$.

\subsection{Not really a duality}\label{sec:notduality}
If we restrict to constructible objects, we also have a biduality statement $F \isoto \DD^T(\DD^T(F))$ and
$\DD^T(-)$ is a contravariant equivalence of categories.  We deduce, for $F,G$ constructible
$$
\RHom(F, T_{c*}(G)) \simeq \RHom(\DD^T(T_{c*}(G)), \DD^T(F)) \simeq \RHom( \DD^T(G), T_{c*}(\DD^T(F)))
$$
It follows that $c(\alpha, F ,G) = c(\alpha, \DD^T(G) ,\DD^T(F))$.  In particular $c(\alpha, L, L') =
c(\alpha, -L', -L)$.

\subsection{Compact support}
For $F, G \in \cD(N)$ which coincide with $k_{N\times\Real}$ near $N\times \{+\infty\}$ we recall
that $c(\alpha, F ,G)$ is obtained as follows.  For $c\leq d$ and $d$ big enough we have
$$
u_c^1 \colon \RHom(F, T_{c*}(G)) \to \RHom(F, T_{d*}(G)) \simeq \rsect(N;k_N)
$$
and, for $\alpha\in H^i(N;k_N)$, we can define $c(\alpha, F,G)$ as the infimum on the set of $c$ such that
$\alpha$ is in the image of $H^i(u_c^1)$.  To express a dual statement we set for short
$$
\RHom_c(F, G) = \rsect_c(N\times \Real; \rhom(F,G))[1]
$$
Then $\RHom_c(F, T_{d*}(G) ) \simeq \rsect_c(N; k_N)$ for $d\ll 0$.  For $d\leq c$ and $d$ small enough we
have
$$
u_c^2 \colon \rsect_c(N;k_N) \simeq \RHom_c(F, T_{d*}(G)) \to \RHom_c(F, T_{c*}(G))
$$
and we can define spectral invariants in this way, say $c'(\alpha, F,G)$, for $\alpha\in H^i_c(N;k_N)$: we
let $c'(\alpha, F,G)$ be the infimum on the set of $c$ such that $H^i(u_c^2)(\alpha) = 0$.  If $N$ is compact,
we have $c'(\alpha, F,G) = c(\alpha, F,G)$: we can check that there is a distinguished triangle
$$
\RHom(F, T_{c*}(G)) \to[u_c^1]  \rsect(N;k_N) \to[u_c^2] \RHom_c(F, T_{c*}(G)) \to[+1]
$$
and, taking the long exact sequence, we see that a given $\alpha \in H^*(N)$ belongs to
$\operatorname{im}(u_c^1)$ exactly when $u_c^2(\alpha) = 0$.

\subsection{Duality}\label{sec:duality}
For $F,G$ constructible and with disjoint microsupports we also have $\DD'\rhom(F,G) \simeq \rhom(G,F)$.
Tensoring with $p^{-1}\omega_{N}$ and using the Poincaré-Verdier duality we get $(\RHom(F, G))^* \simeq
\RHom_c(G, F \otimes p^{-1}\omega_{N})$.  Hence for $F,G$ constructible, with microsupports in
$\{\tau>0\}$, and for a generic $c$ we have the commutative diagram
$$
\xymatrix@C=20mm{
\RHom(F, T_{c*}(G)) \ar@{-}[d]^\wr  \ar[r]^{u_c^1} & \rsect(N;k_N) \ar@{-}[d]^\wr \\
(\RHom_c(G, T_{-c*}(F) \otimes p^{-1}\omega_{N}) )^* \ar[r]^-{(u_c^2)^*} & \rsect_c(N; \omega_N)^*
}
$$
To find $c(1, F,G)$ we look for the values $c$ such that $H^0(u_c^1)$ is surjective and, similarly,
$c'(\mu_N, F,G)$ is related with the injectivity of $H^{d_N}(u_c^2)$.

If $N$ is orientable, then $\omega_{N} \simeq k_N[d_N]$ and we deduce $c(1,F,G) = -c'(\mu_N, G, F)$.

If $N$ is compact and orientable, we obtain finally
$$
c(1,F,G) = -c'(\mu_N, G, F) = -c(\mu_N, G, F) = -c(\mu_N, \DD^T(F), \DD^T(G))
$$

\subsection{Case of Hamiltonian maps}
Now we work on $N^2$ and our sheaves $\K,\ldots \in \mathcal{D}(N^2)$ represent graphs of Hamiltonian maps.
Hence the hypothesis ``$F \simeq k_{N\times\Real}$ near $N\times \{+\infty\}$'' is changed into ``$\K \simeq
k_{\Delta_N \times\Real}$ near $N^2\times \{+\infty\}$''.  Most (iso)morphisms from the above paragraphs~\S\ref{sec:notduality}-\S\ref{sec:duality} have an analog in this framework. In particular we have,
denoting by $\omega_{\Delta_N | N^2}$ is the relative dualizing sheaf, $\omega_{\Delta_N | N^2} \simeq or_N[-d_N]$,
\begin{gather*}
\DD^T(\K_\varphi) \simeq \K_{\varphi^{-1}} \otimes \omega_{\Delta_N| N^2} \\  
  c(\alpha, \K_\varphi ,\K_\psi) = c(\alpha, \DD^T(\K_\psi) ,\DD^T(\K_\varphi))
  = c(\alpha, \K_{\psi^{-1}} , \K_{\varphi^{-1}})
\\
\begin{split}
u_c^1 \colon \RHom(\K, T_{c*}(\K')) &\to \RHom(\K, T_{d*}(\K')) \\
&\simeq \rsect(N^2;k_{\Delta_N}) \simeq  \rsect(N;k_N)
\end{split}
\\
\begin{split}
u_c^2 \colon \rsect_c(N^2;k_{\Delta_N}) \simeq \rsect_c(N;k_N)
  \simeq \RHom_c(\K&, T_{d*}(\K')) \\
  &\to \RHom_c(\K, T_{c*}(\K'))
\end{split}
\end{gather*}
But now we don't have $\DD'\rhom(\K, \K') \simeq \rhom(\K', \K)$, not even up to shift
by $d_N$ (the microsupports are not disjoint even after a small translation). So we don't have duality; instead we get
``new'' invariants. Let us set
$$
\RHom^D(\K, \K') = \rsect_c(N^2\times \Real; \DD \rhom(\K', \K)) 
$$
Then $\RHom^D(\K, T_{d*}(\K')) \simeq \rsect_c(N^2; \omega_{\Delta_N}) \simeq \rsect_c(N; \omega_N)$ for
$d\ll 0$ and we can use this do define $c''(\alpha, \K, \K')$ for $\alpha \in H^*_c(N; \omega_N) \simeq
(H^*(N; k_N))^*$ (in the same way as $c'(\alpha, \K, \K')$).  So the complete set of invariants for
$(\K, \K')$ is now $c(\alpha, \K, \K')$, $c''(\beta, \K, \K')$, for $\alpha \in H^*(N; k_N)$
and $\beta \in H^*_c(N; \omega_N)$.

We have the diagram
$$
\xymatrix@C=20mm{
\RHom(\K, T_{c*}(\K')) \ar@{-}[d]^\wr  \ar[r]^-{u_c^1} & \rsect(N^2;k_{\Delta_N}) \ar@{-}[d]^\wr \\
(\RHom^D(\K', T_{c*}(\K)) )^* \ar[r]^-{(u_c^3)^*} & \rsect_c(N^2; \omega_{\Delta_N})^*
}
$$
which says that $c''(\mu_N,\K, \K') = -c(1, \K', \K)$.  In particular $c''(\mu_N,\K_\varphi,
\K_\psi) = -c(1, \K_\psi, \K_\varphi) = -c(1, \K_{\varphi^{-1}}, \K_{\psi^{-1}} )$.

\subsection{Sheaf and Floer cohomology for Hamiltonian maps}
Let $\varphi$ be a Hamiltonian map of $T^*N$ with compact support. Let us consider 
the cohomology groups $H^*(N\times N \times [a,b[, \K_\varphi, \K_\psi)$ by which we mean \linebreak
$H^*(R\Gamma_{N\times N \times [a,b[}(R\Hom (\K_\varphi, \K_\psi)))$. 

Similarly we can consider the Floer cohomology $FH^*(\Gamma(\varphi), \Delta_N; a,b)$ that is the Floer cohomology filtered between action values $a$ and $b$ with $-\infty<a<b<+\infty$.  
In \cite{Viterbo-Sheaves} it is proved that for $L$ a closed Lagrangian in $T^*N$ and $F_L$ its quantized sheaf, we have  $$H^*(N\times [a,b[, F_{L_1}, F_{L_2})=FH^*(L_1,L_2;a,b)$$
The issue here is that the graph of $\varphi$ is not closed (it is diffeomorphic to the diagonal of $\overline{T^*N}\times T^*N$). So our goal is to prove
\begin{thm}\label{thm-6.1}
    We have for $\varphi\in DHam_c(T^*N)$ that 
   $$ FH^*(\Gamma(\varphi), \Delta_N; a,b)=H^*_{N\times N \times [a,b[}(R\Hom (\K_\varphi, \K_\psi))$$
\end{thm}
Since the sheaf theoretic version of the spectral invariants $c_s(\alpha,\varphi_1, \varphi_2)$ is obtained by looking at the barcode associated to the persistence module $H^*(R\Gamma_{N\times N \times [a,b[}(R\Hom (\K_{\varphi_1}, \K_{\phi_2})))$, we obtain, denoting by $c(\alpha, \varphi_1,\varphi_2)$ the Floer theoretic spectral invariants associated to the persistence module $FH^*(\Gamma(\varphi_1), \Gamma(\varphi_2); a,b)$

\begin{cor}
    We have for $\alpha \in H^*(N)\setminus \{0\}$ $$c_s(\alpha, \varphi_1, \varphi_2)=c(\alpha, \varphi_1,\varphi_2)$$
\end{cor}
    The proof of the theorem is based on the analogous result for closed Lagrangians. Indeed, let $M$ be a Riemannian manifold and let $S$ be a submanifold of $M$, and 
    $d(x,S)$ be the distance from $x$ to $S$. Let $\chi_\rho : \mathbb R \longrightarrow \mathbb R$ be a function such that
    \begin{itemize}
        \item $\chi_\rho(t)=t$ for $t\leq \rho$ 
        \item $\forall t \in \mathbb R, 0\leq \chi'_\rho(t)$
        \item $\chi'_\rho(t)=0$ for $t\geq 2 \rho$
    \end{itemize}
    Finally we set $f_{\varepsilon, \rho}=\chi_\rho \left( \frac{d(x,S)^2}{ \varepsilon}\right)$ and $S_{ \varepsilon,\rho}=\gra(df_{\varepsilon,\rho})$.  Using that 
    $$df_{\varepsilon,\rho}(x)= \frac{1}{\varepsilon}\chi'_\rho\left (\frac{d(x,S)^2}{ \varepsilon }\right )d(x,S)\nu_S(x)$$ 
    where $\nu_S(x)$ is the normal to the level set of $d(x,S)$ through $x$, i.e. $\nu_S(x)=d(d(x,S))\in T^*_xM$, we see that, for $0<A<\sqrt{\rho/\varepsilon}$,
    \begin{itemize}
    \item $S_{\varepsilon,\rho}\cap \{(x,p)\mid \vert p \vert \leq A, \, d(x,S) \leq \sqrt{ \varepsilon\rho} \}$ corresponds to $\{x \mid d(x,S) \leq A \varepsilon\}$, and in this region $0\leq f_{ \varepsilon, \rho}(x) \leq A^2 \varepsilon$
        \item Over $A \varepsilon \leq d(x,S)\leq \sqrt{\rho\varepsilon}$ we have $S_{\varepsilon,\rho}\cap \{(x,p)\mid \vert p \vert \leq A\}=\emptyset$
        \item Over $\sqrt{ \varepsilon\rho} \leq d(x,S)$ we have 
        $f_{\varepsilon,\rho} (x) \geq \rho $
    \end{itemize}

\begin{figure}[ht]
\centering
 \begin{overpic}[width=9cm]{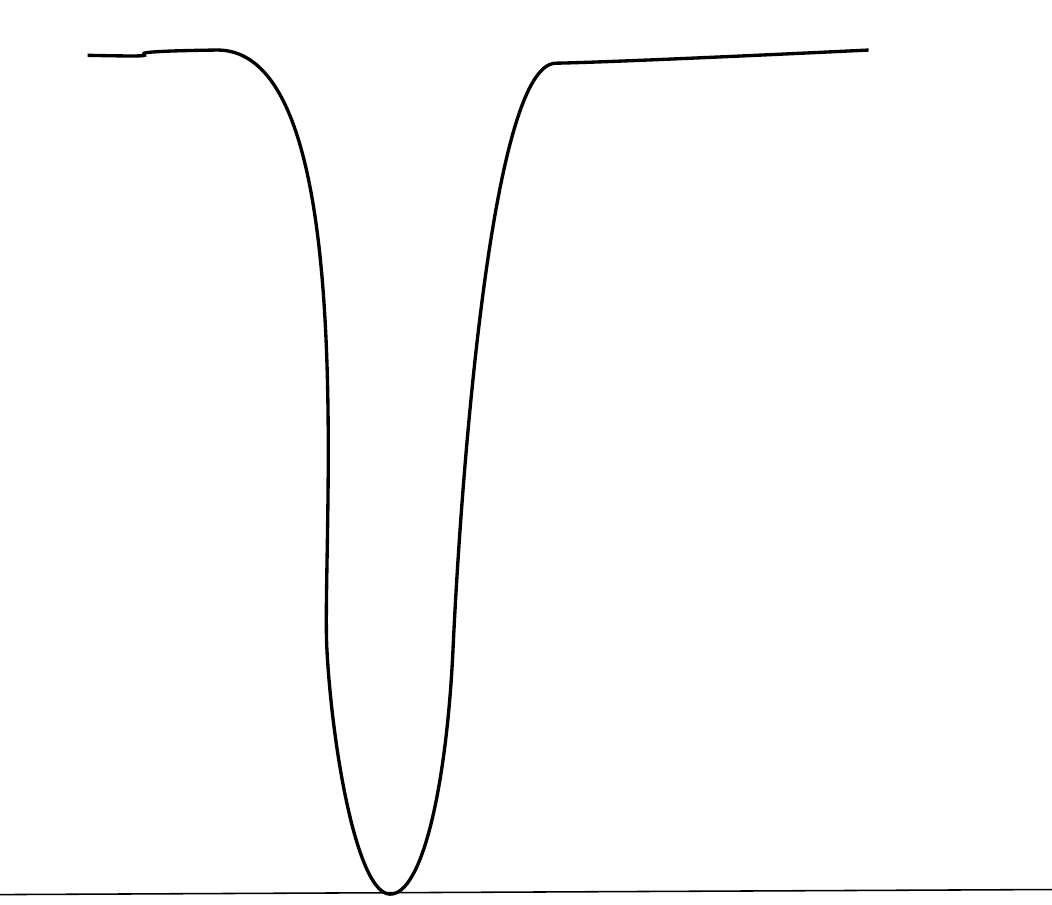}
  \put (20,85) {$\rho$} 
   \put (35,-5) {$S$} 
 \end{overpic}
 \caption{The function $f_{\varepsilon, \rho}$}\label{Figure-1}
\end{figure}

As a result for a fixed Lagrangian $L$ contained in $\vert p \vert \leq A/2$ and such that $f_L$ is bounded by $\rho/2$, the intersection points $L\cap S_{ \varepsilon, \rho}$ are 
\begin{itemize}
    \item over $\{x \mid d(x,S) \leq A \varepsilon\}$ and $0\leq f_{\varepsilon,\rho}(x)\leq A^2 \varepsilon$
    \item $\sqrt{\varepsilon\rho}\leq d(x,S)$ and then $f_{\varepsilon,\rho}(x)\geq \rho$ 
\end{itemize}
Let us set $V_\varepsilon=\gra (df_{\varepsilon,\rho})\cap \{x \mid d(x,S)\leq A \varepsilon\}$.
As a result, the intersection points of $L\cap D_{ \varepsilon, \rho}$ contributing to $FH^*(L, D_{\varepsilon, \rho}; a,b)$ are, for $\varepsilon$ small enough,  in one-to-one correspondence with the points $(x,p)$ of $L\cap V_{\varepsilon}$ such that $f_L(x,p)\in [a,b]$. 
\begin{prop}\label{Prop-6.3} For $a,b$ finite and $\rho$ sufficiently large we have
   $$\lim_{\substack{\varepsilon \to 0\\ \varepsilon'\to 0}} FH^*((\varphi\times \id)D_{\varepsilon,\rho},D_{\varepsilon',\rho};a,b) = FH^*(\Gamma(\varphi), \Delta_{T^*N};a,b) $$
\end{prop}
\begin{proof}
    Note that 
$\lim_{\varepsilon \to 0}V_{\varepsilon}=\nu^*S$, the limit being understood as $C^1$-con\-ver\-gence on compact subsets. 
Indeed $df_{\varepsilon, \rho}(x)=\frac{1}{\varepsilon}d(x,S)\nu_S(x)$
so if $x=\exp_y{(t\nu)}$ where $\nu\perp T_yS$, and $t\leq A \varepsilon$, 
$df_{\varepsilon, \rho}(x)= \frac{t}{\varepsilon}\nu_S(x)$ where $\nu_S(x)=d\exp(t\nu)\nu$. 

In local coordinates $y\in V, \nu\in (T_yS)^\perp$ we may consider the map $$(x,\nu) \mapsto (\exp_y( \varepsilon\nu), \nu)$$
Set $D_AT^*M=\{(q,p) \in T^*M \mid \vert p \vert 
\leq A\}$. The above map defines a smooth diffeomorphism $C^1$ close to the identity sending $\nu^*S \cap D_AT^*M$
to $S_{\varepsilon, \rho}\cap D_AT^*M$. 

Now we set $S=\Delta_N$ and $\nu^*\Delta_{N}=\Delta_{T^*N}$ and $D_{\varepsilon,\rho}=S_{\varepsilon, \rho}$. 
We have to be careful with the definition of $FH^*(\Gamma(\varphi), \Delta_{T^*N};a,b)$ since $\Gamma(\varphi)$ and $\Delta_{T^*N}$ coincide at infinity. Indeed we must then slightly perturb $\Gamma(\varphi)$ so that we have transversality. 
Then we slightly change the action of the intersection points at infinity. This action is zero, and we replace them with finitely many points with arbitrarily small action. We assume $a,b\neq 0$ and the perturbation so small that no critical value crosses $a$ or $b$. 

It is now enough to prove that - after this perturbation- the holomorphic strips involved in the definition of $FH^*(\Gamma(\varphi), \Delta_{T^*N};a,b)$
must remain in $D_AT^*M$. Here $a,b$ are fixed (and finite !), and we may increase $A$. This follows from Lemma \ref{Lemma-E3} below.

Finally, since $(\varphi\times \id) D_{\varepsilon, \rho}$ $C^1$ converges to $(\varphi\times \id)\Delta_{T^*N}$
and $D_{\varepsilon', \rho}$ $C^1$-converges to $\Delta_{T^*N}$, we may conclude our proof. 

\end{proof}

\begin{lem}\label{Lemma-E3} Let $N$ be a complete Riemannian manifold. 
Let $L$ be a Lagrangian coinciding with $0_N$ outside a compact set, and $f$ a $C^1$ bounded function without critical points at infinity (i.e. $10 \delta_0 > \vert df(x) \vert > \delta_0>0$  for $x$ large enough and some $\delta_0>0$). We define 
$FC_J^*(L,0_N;a,b)$ to be the limit as $ \varepsilon $ goes to $0$ of 
$FC_J^*(L, \gr (\varepsilon df); a,b)$. Then the holomorphic curves defining $FC_J^*(L, \gr (\varepsilon df); a,b)$ for $ \varepsilon $ small enough are contained in some bounded set independent from $ \varepsilon $. 
 \end{lem} 
  \begin{proof} 
  We shall choose an almost complex structure such that outside a bounded set, $J$ is analytic and in a neighbourhood $$W=D_r(T^*N)=\{(q,p)\in T^*N \mid \vert p \vert \leq r\}$$ of $0_N$,  there is an antisymplectic involution such that $\tau^*J=-J$. This is always possible by taking a complex structure that is a fixed point of $J \mapsto -\tau^*(J)$ where $\tau(q,p)=(q,-p)$. Since the set of almost complex structures is a convex set, such a fixed point for the above involution must exist. 
  
  We now claim that if $C_ \varepsilon $ is a holomorphic strip with boundaries in $L$ and $ \varepsilon \gra (df)$, and containing a point $x_ \varepsilon $ going to infinity as $ \varepsilon $ goes to $0$, then the area of $C_ \varepsilon  \cap W$ is going to infinity. Indeed, the areas of the curves is bounded by $b-a$, so in the limit $ \varepsilon $ going to $0$ we would get a curve $C_0$ with boundary in $L\cup0_N$. Then 
 using the reflection principle (see \cite{Sibony}, Theorem 3.5 and also \cite{Chirka-1, Chirka-2, Ivashkovich-Sukhov}) we get a closed curve in $W$ through a point $x_ \varepsilon $ and exiting from the ball of radius $R_ \varepsilon $  centered at $x_ \varepsilon $.  Here $\lim_{ \varepsilon \to 0} R_ \varepsilon =+\infty$ since we only need that in $B(x_ \varepsilon , R_ \varepsilon )$ $L=0_N$ and $J$ is analytic and invariant by $-\tau^*$. Then the area of such a curve goes to $\infty$, by the standard monotonicity argument\footnote{see \cite{White}, p. 8 and 9. We have   $\frac{\area (B(x_\varepsilon, r)\cap C}{\pi r^2}$ is increasing with $r$ and greater than $1$. A similar statement holds in the general case for $r$ less than the injectivity radius $\rho$. But we may cover $B(x_ \varepsilon , R_ \varepsilon )$ by balls of radius $\rho$, and $C$ will intersect an increasing number of these balls.}. But then the area of $C$ would be infinite. 
  
  \end{proof} 
Similarly setting $U_{\varepsilon, \rho}=\{(x,t)\mid f_{\varepsilon,\rho}(x)\leq t\}$, and noticing that 
$SS (k_{U_{\varepsilon,\rho}})=D_{\varepsilon,\rho}$, we prove
\begin{prop}\label{Prop-6.4}
    We have $$\lim_{\varepsilon\to 0}\lim_{\varepsilon'\to 0}H^*_{N\times N \times [a,b[}(RHom (K_\varphi\circ k_{U_{\varepsilon,\rho}}, k_{U_{\varepsilon',\rho}}))=H^*_{N\times N \times [a,b[}(RHom (\K_{\varphi} , \K_{\id} ))$$
\end{prop}
 \begin{proof}
     Indeed this follows from the fact that clearly $$\varinjlim_{\varepsilon\to 0}{k_{U_{\varepsilon,\rho}}}(W)=k_{V\times [0,+\infty[}(W)$$
on bounded sets (i.e. bounded in the $ \mathbb R$ component). But since the direct limits of sheaves are sheafification of the direct limit of the presheaves, we get 
$$\varinjlim_{\varepsilon\to 0}{k_{U_{\varepsilon,\rho}}}=k_{V\times [0,+\infty[}$$ and in our case
$$\varinjlim_{\varepsilon\to 0}{k_{D_{\varepsilon,\rho}}}=k_{\Delta_N\times [0,+\infty[}=\K_{\id}$$
Because direct limits is an exact functor, we get 
$$\lim_{\varepsilon'\to 0}H^*_{N\times N \times [a,b[}(RHom (K_\varphi\circ k_{U_{\varepsilon,\rho}}, k_{U_{\varepsilon',\rho}}))=H^*_{N\times N \times [a,b[}(RHom (\K_{\varphi}\circ k_{U_{\varepsilon,\rho}} , \K_{\id} ))$$ and taking the limit as $\varepsilon$ goes to $0$ we get the Proposition. 
\end{proof}
\begin{proof}[Proof of Theorem \ref{thm-6.1}]
Now since the quantization of $D_{\varepsilon, \rho}$ is $k_{U_{\varepsilon,\rho}}$ and the quantization of 
$(\varphi\times \id)D_{\varepsilon,\rho}$ is $\K_{\varphi}\circ k_{U_{\varepsilon,\rho}}$  the combination of Proposition \ref{Prop-6.3} and \ref{Prop-6.4} yields Theorem \ref{thm-6.1} 

\end{proof}
\end{color}
\printbibliography
\end{document}